\documentclass[11pt]{article}
\usepackage[margin=1in]{geometry}

\usepackage{listings}
\usepackage{geometry}

\usepackage{amsmath,amsthm,amsfonts,amssymb,amscd}
\usepackage[colorlinks=true, pdfstartview=FitV, linkcolor=blue,citecolor=blue, urlcolor=blue]{hyperref}
\usepackage[abbrev,lite,nobysame]{amsrefs}
\usepackage[usenames,dvipsnames]{color}
\usepackage{times}

\usepackage{mathtools}
\usepackage{mathabx}
\usepackage{float}
\usepackage{makeidx}
\usepackage{stmaryrd}
\usepackage{mathrsfs}
\usepackage{emptypage}
\usepackage{amsrefs}
\usepackage{xfrac}

\usepackage{upgreek}
\newcommand{\vorticity}{\upomega}

\renewcommand{\epsilon}{\varepsilon}

\usepackage{graphicx}

\def\les{\lesssim}

\def\eps{\varepsilon}
\renewcommand*{\div}{\ensuremath{\mathrm{div\,}}}

\newcommand{\norm}[1]{\left \| #1 \right\|} 
\newcommand{\abs}[1]{\left|#1\right|}

\newcommand*{\supp}{\ensuremath{\mathrm{supp\,}}}

\newcommand{\RR}{\mathbb R}
\newcommand{\NN}{\mathbb N}
\newcommand{\TT}{\mathbb T}
\newcommand{\OO}{\mathcal O}

\renewcommand*{\tilde}{\widetilde}

\renewcommand*{\bar}{\overline}

\newcommand{\T}{{\mathbb T}}

\newtheorem{theorem}{Theorem}[section]
\newtheorem{lemma}[theorem]{Lemma}

\newtheorem{proposition}[theorem]{Proposition}
\newtheorem{corollary}[theorem]{Corollary}
\theoremstyle{definition}

\newtheorem{conjecture}[theorem]{Conjecture} 
\newtheorem{remark}[theorem]{Remark}

\numberwithin{equation}{section}

\def\p{\partial}

\def\bb{w}
\def\f1r{{\frac{1}{r}}  }
\def\p{\partial}

\def\bb{w}
\def\f1r{{\frac{1}{r}}  }

\def\tt{\tilde{t} }

\def\tu{\tilde{u} }
\def\trho{\tilde{\rho}}

\usepackage{fancyhdr}
\title{{\bf Formation of shocks  for    2D isentropic compressible Euler  }}
 
\author{
{ \small {\bf Tristan Buckmaster}}\thanks{\footnotesize Department of Mathematics, Princeton University, Princeton, NJ 08544, \href{buckmaster@math.princeton.edu}{buckmaster@math.princeton.edu}.}
\and  
{\small {\bf Steve Shkoller}}\thanks{\footnotesize Department of Mathematics, UC Davis, Davis, CA 95616, \href{shkoller@math.ucdavis.edu}{shkoller@math.ucdavis.edu}.}
\and 
{\small {\bf Vlad Vicol}}\thanks{\footnotesize Courant Institute of Mathematical Sciences, New York University, New York, NY 10012, \href{vicol@cims.nyu.edu}{vicol@cims.nyu.edu}.}
}

\date{} %
\begin{document}

\maketitle

\begin{abstract}
We consider the 2D isentropic compressible Euler equations, with pressure law $p(\rho) = (\sfrac{1}{\gamma}) \rho^\gamma$, with $\gamma >1$. We  provide an elementary  constructive proof of shock formation from  smooth initial datum of finite energy, with no vacuum regions, and with {nontrivial vorticity}.   We prove that for initial data which has minimum slope $- {\sfrac{1}{ \eps}}$, for $ \eps>0$ taken sufficiently small relative to the $\OO(1)$ amplitude,  there exist smooth solutions to the Euler equations which form a shock in time $\OO(\eps)$. The blowup time and location  can be explicitly computed and solutions at the  blowup time are of cusp-type, with H\"{o}lder  $C^ {\sfrac{1}{3}}$ regularity. 

Our objective is the construction of solutions with inherent $\OO(1)$ vorticity at the shock. As such, rather than perturbing from an irrotational regime, we instead construct solutions with dynamics dominated by purely azimuthal wave motion. We   consider homogenous solutions to the Euler equations and use Riemann-type variables to obtain a system of forced transport equations. Using a transformation to modulated self-similar variables and pointwise estimates for the ensuing system of transport equations, we show the global stability, in self-similar time, of a smooth blowup profile.
\end{abstract}
\setcounter{tocdepth}{1}
\tableofcontents

\section{Introduction}

We consider the Cauchy problem for the two-dimensional isentropic compressible Euler equations
\begin{subequations}
\label{eq:Euler}
\begin{align}
\partial_t (\rho u) + \div (\rho\, u \otimes u) + \nabla p(\rho) &= 0 \,,  \label{eq:momentum} \\
\partial_t \rho  + \div (\rho u)&=0 \,,  \label{eq:mass}
\end{align}
\end{subequations}
where $u :\mathbb{R}^2  \times \mathbb{R}  \to \mathbb{R}^2  $ denotes the velocity vector field, $\rho: \mathbb{R}^2  \times \mathbb{R}  \to \mathbb{R}  _+$ denotes the
strictly positive density, 
and the pressure $p: \mathbb{R}^2  \times \mathbb{R}  \to \mathbb{R}  _+$ is defined by the ideal gas law
$$
 p(\rho) = \tfrac{1}{\gamma} \rho^\gamma\,, \qquad \gamma >1 \,.
$$
The sound speed $c(\rho) = \sqrt{ \sfrac{\p p}{\p \rho} }$ is then given by 
$c= \rho^ \alpha $ where  $\alpha = \tfrac{\gamma-1}{2}$.
The Euler equations \eqref{eq:Euler}  are a system of conservation laws:  \eqref{eq:momentum} is the conservation of momentum,  which can be equivalently written as $\p_t u + u \cdot \nabla u + \rho^{\gamma-2} \nabla \rho=0$, and \eqref{eq:mass} is
the conservation of mass.   

This paper is devoted to the construction of  solutions to \eqref{eq:Euler} which form a shock in finite time:
specifically,  starting from smooth initial data with $\OO(1)$ amplitude and a minimum slope of $- {\sfrac{1}{\eps}} $ with $\eps>0$ sufficiently small, we  construct solutions  to the 
2D Euler equations \eqref{eq:Euler} on a time 
interval $t_0 \le t \le T_*$, $t_0=-\eps$ and $T_*=\OO(\eps ^ {\sfrac{5}{4}}) $, 
 for which  $\rho( \cdot , t)$ and $u( \cdot , t) $ remain bounded,  while $| \nabla \rho( \cdot , t)| \to \infty $ and $| \nabla u( \cdot , t)| \to \infty $ as $t \to T_*$; 
moreover, no other type of singularity can form prior to $t = T_*$, and detailed information on the singularity formation at $t=T_*$ is provided, including blowup time, location, and
profile regularity.

We are particularly interested in devising solutions to \eqref{eq:Euler} which have {\it large\footnote{Due to the time rescaling symmetry of the Euler
equations, by which $u^\beta(x,t) = \beta ^{-1} u(x, \beta ^{-1} t)$ and $\rho ^\beta(x,t) = \beta^{- {\sfrac{1}{\alpha }} } \rho(x, \beta ^{-1} t)$ are also solutions to \eqref{eq:Euler}, 
$ \nabla u$ can be made smaller or larger by changing the time interval of the evolution.} vorticity} at the shock, by which we mean solutions which are not small perturbations of 
irrotational flows.   As such, our  strategy will be to construct solutions  that are perturbations of purely azimuthal wave motion whose simplest (constant) profiles are of the
$x^\perp$-type with $\OO(1)$ vorticity at this most basic level.    As we shall describe in great detail below, this is in  contrast to those solutions which are small perturbations of irrotational simple plane waves.

We are thus motivated to develop a framework of analysis for solutions which are perturbations of purely azimuthal waves. Obviously, polar coordinates provide a natural
setting  for describing such perturbative solutions, but more fundamentally, we  have discovered  that the use of homogeneous solutions to \eqref{eq:Euler} leads to a
remarkable reduction of the Euler dynamics precisely to this nearly-azimuthal wave regime, in which 
 bounded azimuthal waves steepen and then shock, while radial
waves (and their slopes) remain bounded.    Owing to the inherent vorticity in the most basic
wave motion, the solutions are fundamentally two-dimensional  in their evolution.   We provide a  precise description of the shock formation for such Euler solutions, including the {\it blowup time and location},  by a transformation to  self-similar variables that contain dynamically evolving modulation functions that keep track of the location, time, and amplitude of 
the
blowup.   At the blowup time $t=T_*$, the wave profile is of H\"{o}lder-class $C^{ {\sfrac{1}{3}} }$.  In the special case that the adiabatic exponent  $\gamma$ is equal to $3$ and for purely azimuthal initial velocity fields, a series of surprising cancellations reduces the 2D Euler dynamics to an  elementary study of  the Burgers equation.  The solution for
the  special case that
 $\gamma=3$  can be viewed as the purely azimuthal wave motion, and its shock formation is 
completely  characterized for all time.
 
\begin{theorem}[\bf Rough statement of  the main theorem]
For an open set of smooth initial data with $\OO(1)$ amplitude and with minimum initial slope given at initial time $t_0$ to equal  $ -{\sfrac{1}{\eps}} $, for $\eps>0$ taken sufficiently small, 
 there exist smooth solutions of the Euler equations with $\OO(1)$ vorticity, which form an asymptotically self-similar shock in finite time $T_* $, such that $T_*-t_0 = \OO( \eps)$.    The  solutions
 have $\OO(1)$ vorticity at the shock,  are dominated by
 azimuthal wave motion, and the   location and time of the first singularity  can be explicitly computed. The blowup profile at the
first singularity is shown to be a cusp with $C^ {\sfrac{1}{3}} $ regularity.
\end{theorem} 

The precise statement of the main theorem is given in Theorem \ref{thm:general}, while the special case that $\gamma=3$ is treated in  Theorem \ref{thm:g3}.

\subsection{A brief history of the analysis of shock formation for the Euler equations}
The  mathematical analysis of shock formation for the Euler equations has a  long and rich history, particularly in the case of  one space dimension, which allows the full power of the method of characteristics
to be employed.   In 1D,  the velocity $u$ is a scalar and \eqref{eq:Euler} takes the form  
$$\p_t u + uu_x + \rho^{\gamma-2}  \rho_x=0\,, \qquad \p_t \rho+ (\rho u)_x=0\,.$$
Riemann \cite{Ri1860} devised the two {\em invariant}  functions $z= u - \sfrac{c}{\alpha }$ and $w=u+\sfrac{c}{\alpha } $ which are constant along the characteristics
of the two wave speeds $ \lambda_1 = u-c$ and $ \lambda _2 = u + c$:
$$\p_t z + \lambda_1 z_x =0 \,, \qquad  \p_t w + \lambda _2 w_x =0 \,.$$
He proved that from smooth data, shocks can form  in finite time.   The 1D  isentropic Euler equations are an example of a $2 \times 2$ system of conversation
laws.  Using Riemann invariants,  Lax \cite{Lax1964} proved that finite-time shocks can form from smooth data for general  $2 \times 2$ genuinely nonlinear 
hyperbolic systems and  Majda \cite{Ma1984} gave a geometric proof which also allowed for $2 \times 2$ systems with  linear degeneracy;   John \cite{John1974} then proved finite-time
shock formation for $n\times n$ genuinely nonlinear hyperbolic systems;  Liu \cite{Li1979} then generalized this result.  Klainerman-Majda \cite{KlMa1980} proved the formation
of singularities for second-order quasilinear  wave equations which includes the nonlinear vibrating string.   See   the book of  Dafermos \cite{Da2010} for a more extensive bibliography  of 1D results.

In multiple space dimensions,  Sideris \cite{Si1985} proved that $C^1$  regular solutions to \eqref{eq:Euler} have a finite lifespan by establishing 
differential inequalities for certain integrals which lead to a proof by contradiction; in particular, he showed that $\OO( \exp(1/ \eps ))$  is an upper bound for the lifespan (of 3D flows) for
data of size $ \eps $.  The nature of the proof did not, however,  reveal the type of singularity that develops, but rather, that some  finite-time breakdown must occur.   

The first proof of shock formation for the compressible Euler equations in the multi-dimensional setting  was given by
Christodoulou \cite{Ch2007} for  relativistic fluids and with the restriction of {\it irrotational} flow.     Later Christodoulou-Miao \cite{ChMi2014} used the same framework to study shock formation in the non-relativistic setting and also for irrotational flow.   Christodoulou's
method is based upon a novel eikonal function (see also Christodoulou-Klainerman \cite{ChKl1993} and Klainerman-Rodnianski \cite{KlRo2003}), whose level sets correspond to characteristics of the flow; by introducing the inverse foliation density, a function which is inversely 
proportional to time-weighted derivatives of the eikonal function, Christodoulou proved that shocks form when the inverse foliation density vanishes (i.e., characteristics cross), and
that no other breakdown mechanism can occur prior to such shock formation.    The proof relies on the use of a geometric coordinate system, along which the solution has long time 
existence, and remains bounded, so that the shock is constructed by the singular (or degenerate) transformation from geometric to Cartesian coordinates.   
For the restricted shock development problem, in which the  Euler solution is continued past the time of first singularity but vorticity production is neglected, see the discussion in 
Section 1.6 of \cite{Ch2019}.   Starting with piecewise regular initial data for which there is a closed curve of discontinuity, across which the density and normal component
of velocity experience a jump, Majda \cite{Ma1983,Ma1983b,Ma1984}, proved (for more general flows than the 2D isentropic flows) that  such a shock can always be continued for a short
 interval of time, but with derivative loss.  For such shock initial data, M\'{e}tivier \cite{Met2001}  later reduced the
derivative loss to  only a $ {\sfrac{1}{2}} $-derivative.   Gues-M\'{e}tivier-Williams-Zumbrun \cite{GuMeWiZu2005} studied the existence and stability of this multidimensional
shock propagation problem in the vanishing viscosity limit.

A special feature of  irrotational flows is that the Euler equations can be expressed as a second-order quasilinear wave equation with respect to the velocity potential.   The first results 
on shock formation for 2D quasilinear wave equations which do not satisfy Klainerman's null condition  \cite{Kl1984} were established by Alinhac \cite{Al1999a,Al1999b}, wherein a detailed
description of the blowup was provided.   The geometric framework of \cite{Ch2007} has influenced more recent analysis of shock formation for quasilinear wave equations.  
Holzegel-Klainerman-Speck-Wong \cite{HoKlSpWo2016}  have explained the mechanism for stable shock formation for certain types of quasilinear wave equations with small data in three dimensions.    Speck \cite{Sp2016} generalized and unified earlier work on singularity formation for both
covariant and non-covariant scalar wave equations of a certain form.   He proved  that whenever the nonlinear terms fail 
Klainerman's  null condition \cite{Kl1984}, shocks develop in solutions arising from an open set of small data, and can thus be viewed as a converse to the well-known result of Christodoulou-Klainerman \cite{ChKl1993}, which showed that when the classic null condition is verified, small-data global existence holds.   For quasilinear wave equations that are  derived from the least action principle and which satisfy the null condition,  Miao-Yu \cite{MiYu2017} proved shock formation using the so-called short pulse  data.

The first proof of shock formation
for fluid flows with vorticity   was given by Luk-Speck  \cite{LuSp2018}, for the 2D isentropic Euler equations with vorticity. The presence of nontrivial 
vorticity in their analysis does not only allow for a much larger class of data, but also has two families of waves being propagated, sound waves and vorticity waves, thus allowing for 
multiple characteristics (wave speeds) to interact.   Their proof uses Christodoulou's geometric framework from \cite{Ch2007, ChMi2014}, but develops new methods to contend with the 
aforementioned vorticity waves, establishes new estimates for the regularity of  the transported vorticity-divided-by-density, and relies crucially on a new framework for describing the 2D 
compressible Euler equations as a coupled system of covariant wave and transport equations.

Luk-Speck consider in \cite{LuSp2018} solutions to Euler which  are {\it small} perturbations  of a subclass of outgoing simple plane waves.     
In the 2D  Cartesian plane, with coordinates $(x_1,x_2)$, an outgoing simple plane wave is defined as a solution to the Euler equations \eqref{eq:Euler} which moves to the right along
the $x_1$ axis,   does not depend on
$x_2$, and  has vanishing  first Riemann invariant $u^1 - c$.  The 
{\it smallness} of the perturbation of the plane wave is measured in terms of the ratio of the maximum wave amplitude to  the minimum (negative) slope of the initial wave profile.
Specifically, they construct solutions which are  small perturbations of the irrotational simple plane waves, in which  the transverse derivative (to the acoustic characteristics) of
$u^1$ blows up, while the tangential derivatives (to the acoustic characteristics)  of $(\rho,u^1,u^2)$ remain bounded, and  vorticity is non-vanishing  and small at the shock.

\subsection{Shock formation with vorticity and the perturbation of purely azimuthal waves}
Let us now describe the type of shock wave solutions that we construct and compare them with those of   \cite{LuSp2018}.
  As noted above, we do not consider perturbations of simple plane waves, but instead construct solutions which  are  perturbations of azimuthal waves.
   
Using 2D polar coordinates $(r,\theta)$, we denote the velocity components by $u=\left(u_r(r,\theta,t),u_\theta(r,\theta,t)\right)$.    We consider initial conditions $\left(\rho( \cdot , t_0), u_r( \cdot , t_0),
u_\theta ( \cdot , t_0)\right)$ which have $\OO(1)$ amplitude, but with $\p_\theta u_\theta( \cdot , t_0)$ and $\p_\theta \rho( \cdot , t_0)$
having a minimum (negative) value of $- {\sfrac{1}{\eps}} $, with $0< \eps \ll 1$ taken sufficiently small.   There are  two Riemann invariants for the azimuthal flow, which we write
as $ \mathcal{R} _\pm = u_\theta \pm \frac{2}{\gamma-1} \rho^{\sfrac{(\gamma-1)}{2}}$.
The solutions we construct satisfy the following  conditions:
\begin{enumerate}
  \setlength\itemsep{0em}
\renewcommand{\theenumi}{(\alph{enumi})}
\renewcommand{\labelenumi}{\theenumi}
\item \label{item1}
solutions $(\rho, u_r,u_\theta)$ have  $\OO(1)$ bounds in $L^ \infty $  for $t\in [t_0,T_*)$ with
linear variation in the radial  $r$ direction for $u_r$ and $u_\theta$ and $r^ {\sfrac{2}{(\gamma-1)}}$ variation for $\rho$;  
\item $\abs{ \p_\theta \mathcal{R} _+}$, $\abs{ \p_\theta u_\theta}$, and  $\abs{\p_\theta  \rho  }$ are $ \OO(\sfrac{1}{\eps})$ at initial time, and these quantities blow up at  time $t=T_*$ with
a rate proportional to $ {\sfrac{1}{(T_*-t)}} $, where $T_*-t_0 = \OO( \eps)$;
\item  the blowup profile is of  cusp-type with $u_\theta( \cdot , T_*)$  and $\rho( \cdot ,T_*) $ in the H\"{o}lder space $C ^ {\sfrac{1}{3}} $; 
\item $ \p_\theta \mathcal{R} _-$ remains bounded on on $[t_0, T_*) $;
\item  $\p_r$ of $(\rho,u_r,u_\theta)$ and $\p_\theta u_r$ are bounded on $[t_0, T_*) $;
\item \label{item6} the vorticity $\partial_r u_\theta - \tfrac{1}{r} \partial_\theta u_r+ \tfrac{1}{r} u_\theta$ is non-vanishing and bounded at the shock.
\end{enumerate}

There is some correspondence between the properties \ref{item1}--\ref{item6} of our solutions and the solutions constructed by Luk-Speck \cite{LuSp2018}, in that we are perturbing
purely azimuthal wave motion (in the $\theta$-direction), and in \cite{LuSp2018} they are perturbing simple plane wave motion (in the $x_1$-direction).   A primary difference 
is that  the purely azimuthal wave already has nontrivial vorticity, while the simple plane wave is irrotational, and so we 
are constructing solutions that are perturbations of flows with nontrivial vorticity.   Furthermore, our method allow us to provide a fairly detailed description of the blowup profile for
$u_\theta ( \cdot , T_*)$ and $\rho( \cdot , T_*)$: 
the slope becomes infinite along a line segment, and each function is $C^{\sfrac{1}{3}} $ in space.

As we shall next describe, the method we develop to construct   shock wave solutions 
is very different from the methods of \cite{Ch2007,ChMi2014,LuSp2018}; we rely upon a transformation to modulated self-similar variables together with the  
fact that 2D purely azimuthal wave motion is governed by the dynamics of the Burgers equation; we shall explain how our analysis relies on properties of 
nonlinear transport equations together with explicit properties of the asymptotically stable self-similar profile.

\section{Outline of the  proof}
\subsection{A new   class of solutions that shock}  In order to study perturbations of purely azimuthal waves, we write the Euler equations
 \eqref{eq:Euler} in polar coordinates for the variables $ (\rho, u_r, u_\theta)$ as the following system of conservation laws:
\begin{subequations}
\label{eq:Euler:polar}
\begin{align}
\left(\partial_t  + u_r\partial_r + \tfrac{1}{r} u_\theta \partial_{\theta}\right) u_r -\tfrac{1}{r}u_{\theta}^2+\rho^{\gamma-2} \partial_r \rho &=0 \,, \\
\left(\partial_t  + u_r\partial_r +\tfrac{1}{r} u_\theta \partial_{\theta}\right)u_\theta+\tfrac{1}{r}u_r u_\theta +\tfrac{1}{r}\rho^{\gamma-2} \partial_\theta\rho&=0 \,,  \\
\left(\partial_t  + u_r\partial_r + \tfrac{1}{r} u_\theta \partial_{\theta}\right) \rho + \rho\left( \tfrac{1}{r} u_r + \partial_r u_r + \tfrac{1}{r} \p_\theta u_\theta \right)  &=0 \,.
\end{align}
\end{subequations}
  These equations are solved
with  $\theta \in \T=[-\pi,\pi]$ , $r>0$ and $t\in[ t_0,T]$.
Defining the fluid vorticity  $\upomega = \tfrac{1}{r} \p_r (r u_\theta) - \tfrac{1}{r} \partial_\theta u_r$, we shall make use of the fact that $\upomega/\rho$ is transported as
\begin{align} 
\p_t  \tfrac{\upomega }{\rho }  + u \cdot \nabla \tfrac{\upomega }{\rho }  =0 \,. \label{omega_rho_transport}
\end{align} 

For initial density $\rho_0>0$ that has no vacuum regions, and for {\em nontrivial initial vorticity}
$$
 \upomega(r,\theta,t_0) =  \partial_r u_\theta(r,\theta,t_0) - \tfrac{1}{r} \partial_\theta u_r(r,\theta,t_0) + \tfrac{1}{r} u_\theta(r,\theta,t_0) \neq 0 \,,
$$
we construct smooth solutions to \eqref{eq:Euler} that form a shock in finite-time.  So that our solutions will be perturbations of azimuthal waves,  we  shall  consider  homogeneous solutions. 

To this end, motivated by the homogeneous solutions introduced for studying singularity formation in incompressible flows by Elgindi and Jeong \cite{ElJe2016}, 
we consider the new variables $\tu$ and $\trho$ such that
$$
u(r,\theta,t) = r \tu(r,\theta, t) \text{ and } \rho(r,\theta,t) =r^{ \frac{2}{ \gamma-1 }} \trho(r,\theta,t) \,,  
$$
and recalling that $\alpha = {\frac{\gamma -1}{2}}$, 
with respect to these new variables, the 
system \eqref{eq:Euler:polar} takes the form:
\begin{subequations}
\label{eq:Euler:polar2}
\begin{align}
\left(\partial_t  + \tu_r r\partial_r + \tu_\theta\partial_{\theta}\right) \tu_r  + \tu_r^2-\tu_\theta^2+  \tfrac{1}{ \alpha } \trho^{2 \alpha } +\trho^{2 \alpha -1 } r\p_r \trho    &=0\,,\\
\left(\partial_t  + \tu_r r\partial_r +  \tu_\theta\partial_{\theta}\right)\tu_\theta+2\tu_r \tu_\theta+   \trho^{2 \alpha -1 }\partial_\theta\trho&=0\,, \\
\left(\partial_t  + \tu_r r\partial_r + \tu_\theta\partial_{\theta}\right) \trho + \tfrac{\gamma}{ \alpha } \tu_r \trho +  \trho \left( r\p_r \tu_r + \p_\theta \tu_\theta\right) &=0 \, .
\end{align}
\end{subequations}
Notice that all  powers of $r$ have cancelled (expect for the $r \p_r$ operator which is dimensionless), and hence, 
 if at time $t=t_0$, the initial data is given as
\begin{equation}\label{eq:ansatz}
 \tu_r(r,\theta,t_0) = a_0(\theta) \,, \ \ 
 \tu_\theta(r,\theta,t_0) =  b_0(\theta)\,, \ \ 
 \trho(r,\theta,t_0) = P_0(\theta)\,,
\end{equation} 
where $a_0$, $b_0$, and $P_0$ are independent of $r$, then $\tu$ and $\trho$ remain independent of $r$ for as long as the solution stays smooth
(and hence unique), and thus
 the system \eqref{eq:Euler:polar2} reduces to
\begin{subequations}
\label{eq:Euler:polar3}
\begin{align}
\left(\partial_t  + b\partial_{\theta}\right) a  + a^2-b^2+  \alpha ^{-1} P^{2 \alpha }   &=0 \label{g3_a_evo}\\
\left(\partial_t  + b\partial_{\theta}\right)b+2a b+  P^{ 2 \alpha -1}\partial_\theta P&=0 \\
\left(\partial_t  + b\partial_{\theta}\right) P+  \tfrac{\gamma}{ \alpha } a P+  P  \p_\theta b &=0 \, ,
\end{align}
\end{subequations}
and then the solution to the Euler equations \eqref{eq:Euler:polar} is given by
\begin{equation}\label{scale0}
u_\theta(r,\theta,t) = r b( \theta, t) \,, \ \ u_r(r,\theta,t) = r a( \theta, t)\text{ and } \rho(r,\theta,t) =r^{1/ \alpha  } P(\theta,t) \,.
\end{equation} 

The fluid vorticity  and fluid divergence corresponding to the ansatz \eqref{eq:ansatz} are given by 
\begin{subequations} 
\begin{align}
\upomega(r,\theta,t) & = 2 b(\theta,t) - \p_\theta a(\theta,t) \,, \label{omega_abp} \\
\operatorname{div} u(r,\theta, t) &  = 2 a + \p_\theta b \,, \label{div_abp}
\end{align}
\end{subequations} 
so that the vorticity is  therefore nontrivial as long as $2 b\not\equiv \p_\theta a$.    Setting 
$$\varpi = \frac{ 2b - \p_\theta a}{P}\,,$$
from equation \eqref{omega_rho_transport}, we have that
\begin{equation}\label{eq:varpi}
\p_t\varpi + b \p_\theta \varpi = \tfrac{a}{ \alpha } \varpi \,.
\end{equation}

Next, we define the {\em Riemann invariants}  $w$ and $z$ associated to the tangential velocity $b$ and density $P$,  and their associated wave speeds 
$\lambda _1$, $\lambda _2$,  as
\begin{subequations} 
\label{eq:riemann}
\begin{alignat}{2}
w&= b+  {\frac{1}{\alpha }} P^ \alpha \,, \qquad  &&z= b- {\frac{1}{\alpha }}  P^ \alpha  \,,\\
\lambda _1&= b - P^ \alpha= \frac{1- \alpha }{2} w +  \frac{1+ \alpha }{2} z \,, \qquad  &&\lambda _2= b+ P^ \alpha =  \frac{1+ \alpha }{2} w +  \frac{1- \alpha }{2} z  \,.
\end{alignat}   
\end{subequations} 
 Then, the $(a,b,P)$-system \eqref{eq:Euler:polar2} can be written as the following system for the variables $(a,z,w)$:
\begin{subequations}
\label{eq:Euler:wza0}
\begin{align}
\big(\partial_t   +  \lambda _2\partial_{\theta} \big) w  + \tfrac{a}{2}  \big( (1-2 \alpha ) z + (3+2 \alpha )w\big)   &=0\,,\\
\big(\partial_t  +  \lambda _1 \partial_{\theta}\big)z + \tfrac{a}{2}  \big( (1-2 \alpha ) w + (3+2 \alpha )z\big) &=0\,, \\
\big(\partial_t   + \tfrac{w+z}{2}  \partial_{\theta}\big)a + a^2 - \tfrac{1}{4} (w+z)^2 + \tfrac{\alpha }{4} (w-z)^2  &=0 \, .
\end{align}
\end{subequations}
Notice that while $z$ and $w$ are not actual invariants, the advantage of the $(a,z,w)$-system is that no derivatives appear in the  forcing  of the transport.

In order to transform the $w$ and $z$ equations  into the form of a perturbed Burgers-type equation, we
define $t = \tfrac{1+ \alpha }{2} \tt $ so that $\p_t = \tfrac{1+ \alpha }{2} \p_{\tt}$.   For notational simplicity, we  shall write $t$ for $\tt$, in which case \eqref{eq:Euler:wza0}
becomes:
\begin{subequations}
\label{eq:euler:wza}
\begin{align}
\partial_t w +\left(w+\tfrac{1-\alpha}{1+\alpha}z\right)\partial_{\theta}w 
&= -a  \left(\tfrac{1-2\alpha}{1+\alpha} z+ \tfrac{3+2\alpha}{1+\alpha} w\right) 
\label{eq:w:evo} \,,\\
\partial_t z +\left(z+\tfrac{1-\alpha}{1+\alpha}w\right)\partial_{\theta}z 
&=  -a  \left( \tfrac{1-2\alpha}{1+\alpha} w+ \tfrac{3+2\alpha}{1+\alpha} z\right)
\label{eq:z:evo} \,, \\
\partial_t a +\tfrac{1}{1+\alpha} (w+z) \partial_{\theta}a  
&=-\tfrac{2}{1+\alpha}a^2+\tfrac{1}{2(1+\alpha)}(w+z)^2 - \tfrac{\alpha}{2(1+\alpha)}(w-z)^2   \,.
\label{eq:a:evo}
\end{align}
\end{subequations}
While the local-in-time  well-posedness in Sobolev spaces of the system \eqref{eq:euler:wza} follows from the well-posedness of the Euler equations, we shall take the opposite view
that solutions to the Euler equations are constructed from solutions of  \eqref{eq:euler:wza} together with \eqref{scale0} and \eqref{eq:riemann}.   

\begin{lemma}\label{lem:awz_WP}
For initial data $(w,z,a)|_{t=t_0}=(w_0,z_0,a_0) $ in $C^k(\T)$, $k \ge 1$, there exists a time $T$ depending on the $C^k(\T)$-norm of this data, such that there exists a unique solution $(w,z,a) \in C([t_0,T]; C^k(\T)$ to  \eqref{eq:euler:wza}.   Furthermore, the solution continues to exist on $[t_0,T_*]$  if 
\begin{align}
 \int_{t_0}^{T_*} \left(  \| \p_\theta w( \cdot ,t) \|_{L^ \infty (\T) } + \|\p_\theta z ( \cdot ,t) \|_{L^ \infty (\T) } + \|   a( \cdot ,t) \|_{L^ \infty (\T) }  \right) dt < \infty  \,.
 \label{eq:continuation:criterion}
\end{align}
\end{lemma}
 
\begin{proof}[Proof of Lemma~\ref{lem:awz_WP}]
We set $ \beta_0 = \tfrac{1-\alpha}{1+\alpha}$, $\beta_1=  \tfrac{1-2\alpha}{1+\alpha}$,  $\beta_2=  \tfrac{3+2\alpha}{1+\alpha}$,
$ \beta _3 =\tfrac{1}{1+\alpha}$,  and define the characteristics 
\[
\p_t \psi_w =  w \circ \psi_w + \beta_0 z \circ \psi_z \,, \ \  \p_t \psi_z = z \circ \psi_z + \beta_0 w \circ \psi_w \,, \ \  \p_t \psi_a =  \beta _3 (w \circ \psi_w + z \circ \psi_z) 
\]
which are the identity at time $t_0$. Letting $ \mathcal{W} = w \circ \psi_w$, $ \mathcal{Z} = z \circ \psi_z$, and $ \mathcal{A} = a \circ \psi_a$, the system \eqref{eq:euler:wza} is equivalent to 
\begin{align*} 
& \p_t \mathcal{W} = - \mathcal{A}( \beta _1 \mathcal{Z} + \beta _2 \mathcal{W} ) \,,  \  \p_t \mathcal{Z} = - \mathcal{A}( \beta _1 \mathcal{W} + \beta _2 \mathcal{Z} ) \,, 
\  \p_t \mathcal{A} =  \beta _3 \left[  -2 \mathcal{A} ^2 + \tfrac{1}{2} ( \mathcal{W} + \mathcal{Z} )^2 - \tfrac{\alpha }{2}( \mathcal{W} - \mathcal{Z} )^2 \right],
\end{align*} 
with initial data $(\mathcal{W} , \mathcal{Z} , \mathcal{A} )_{ t= t_0} = (w_0,z_0,a_0) \in C^k(\T)$.   Then, a standard Picard iteration
argument proves the existence, uniqueness, and well-posedness of this coupled system of six ODEs some time interval $[t_0,T]$, in the class $C([t_0,T],C^k(\TT))$.  
This local in time solution may be continued as long as the transport velocities remain bounded in $L^1_t {\rm Lip}_x$. Lastly, we have excluded $\norm{w}_{L^\infty}$ and $\norm{z}_{L^\infty}$ from \eqref{eq:continuation:criterion} because these remain finite if $a\in L^1_t L^\infty_x$, while $\norm{\p_\theta a}_{L^\infty}$ remains bounded due to  \eqref{eq:varpi}.
\end{proof}

From a solution 
$(w,z,a)$ of \eqref{eq:euler:wza}, we obtain a solution to the Euler equations \eqref{eq:Euler} using that
$b= \tfrac{w+z}{2}$, 
$P= \left(\tfrac{\alpha (w-z)}{2}\right)^{\sfrac 1 \alpha } $
and defining $(u,\rho)$ using \eqref{scale0}.
Given the Euler velocity field $u$, we define the Lagrangian flow $\eta_u$ as the solution to 
$\p_t \eta_u = u \circ \eta_u$ for $t>t_0$ with $\eta_u(r,\theta,t_0)=(r,\theta)$.  
We shall consider annular regions
\begin{align*}
A _{\underline{r},\overline{r}}= \{ (r,\theta) \colon \underline{r} < r < \overline{r}, \theta \in \TT \} 
\end{align*}
for radii $0< \underline{r} < \overline{r} < \infty$. 
Given  $0< R_0 < r_0 < r_1 <R_0$, we consider a small annulus 
$A_{r_0,r_1}$ properly contained in a large annulus $A_{R_0,R_1}$.   
We define the time-dependent domain\begin{equation}\label{Omega}
\Omega(t)= \eta_u(A_{r_0,r_1},t) \subset A_{R_0,R_1} \quad \text{ for } \quad t \in [t_0,T_*]\,,
\end{equation} 
where the inclusion holds for $T_*$ sufficiently small whenever $u\in L^\infty_t L^\infty_x$.

We shall construct solutions to \eqref{eq:euler:wza} which form a shock in finite time and satisfy  properties \ref{item1}-\ref{item6} listed above.  Before describing our method of construction which is
 based
on a  transformation into self-similar variables,
 there  is a singularly interesting choice for
the adiabatic parameter $\gamma$ which allows for a particularly simple construction of shock formation.   When $\gamma=3$, and hence $ \alpha =1$, it will be shown that the
system \eqref{eq:euler:wza} can be reduced exactly to  $\p_t w + w\p_\theta w=0$ with $a=0$ and $z=0$, in which case we have a purely azimuthal wave solution 
$(\rho,u_r, u_\theta) = \tfrac{1}{2} ( r w, 0, rw)$ with a precise time and location for the shock formation, coming from the well-known solution to the Burgers equation.  As noted above, 
we view this purely azimuthal wave as the polar analogue of the simple plane wave, because the radial velocity component vanishes as does the first Riemann invariant.

\subsection{A  transformation to self-similar variables with modulation functions}
Turning to the case of general  adiabatic exponent $\gamma>1$ for the Euler system \eqref{eq:Euler}, we shall next introduce a self-similar transformation~\cite{GiKo1985} with dynamic modulation variables~\cite{Merle96}.
 Let
 $$x(\theta,t):=\frac{\theta-\xi(t)}{ (\tau(t)-t)^{\frac32}} \,, \quad  s:=-\log(\tau(t)-t) \,,$$
 and define the new variables $(A,Z,W)$ by
 $$
 w(\theta,t)=e^{-\frac s2}W(x,s)+\kappa (t)\,, \quad
z(\theta,t)=Z(x,s)\,, \quad
a(\theta,t)=A(x,s) \,.
 $$
This is a self-similar transformation\footnote{We note that our use of self-similar variables to construct the blowup is in some ways  analogous  to the use of geometric coordinates in the construction scheme of \cite{Ch2007,ChMi2014,LuSp2018} wherein
the long time existence in geometric coordinates leads to a finite-time  blowup by the singular transformation back to Cartesian coordinates.
We also note that self-similar variables have been used  in a very different way to study the problem of self-similar 2D shock reflection off a wedge \cite{ChFe2010,ChFe2018}.} with three dynamic {\it modulation variables}, $ \xi(t)$, $\tau(t)$, and $\kappa(t)$,  each satisfying  relatively simple ordinary differential equations.  This technique was developed in the context of the Schr\"odinger equation ~\cite{Merle96,MeRa05,MeRaRo13}
the nonlinear heat equation~\cite{MeZa97}, the generalized KdV equation~\cite{MaMeRa14},  the nonlinear wave equation~\cite{MeZa15} and other dispersive problems, and it has recently been applied to solve  problems in fluid dynamics~\cite{DaMa2018,CoGhMa2018,CoGhIbMa18,CoGhMa19,El19,ChHoHu19}.  In all these cases, the role of the modulation variables is to enforce certain orthogonality conditions required to study perturbations of the self-similar blowup. In our context, the  modulation variables  $ \xi(t)$, $\tau(t)$, and $\kappa(t)$, respectively, control  precisely the shock location, blowup time, and wave amplitude. In the absence of these dynamic variables, the above rescaling coincides with the well-known self-similar transformation for the Burgers equation (see \cite{CaSmWa96,PoBeGuGr2008,EgFo2009,CoGhMa2018}), but the use of the modulation variables allows us to impose constraints on $W$ and its first and second derivatives at $x=0$.

Upon switching to self-similar variables,  the $(a,z,w)$-system \eqref{eq:euler:wza} is transformed to self-similar evolution equations for $(A,Z,W)$ detailed below in \eqref{eq:ssWZA}.  
As we have noted above, for the special case that $\gamma=3$, this system of  self-similar equations reduces to the self-similar Burgers evolution, and a key feature of our proof is
that the construction of shocks which are perturbations of purely azimuthal waves exactly coincides with the self-similar perturbation of the Burgers equation.   Of paramount importance
to our analysis, then, is the explicit representation  of the stable, steady-state, self-similar Burgers profile~\cite{CaSmWa96}
\begin{align}
 \bar W(x) = \left(- \frac x2 + \left(\frac{1}{27} + \frac{x^2}{4}\right)^{\sfrac 12}\right)^{\sfrac 13} - \left( \frac x2 + \left(\frac{1}{27} + \frac{x^2}{4}\right)^{\sfrac 12} \right)^{\sfrac 13} \, ,
 \label{eq:barW:def}
 \end{align} 
 solving the steady self-similar Burgers equation
\begin{align}
 -\frac 12 \bar W + \left( \frac{3x}{2} + \bar W \right) \partial_x \bar W = 0   \,.
 \label{eq:barW:dx}
\end{align}

Our proof of finite-time blowup for $\p_\theta u_\theta$  and $\p_\theta \rho$ relies upon showing that  $\p_\theta w$ has finite-time blowup, which in turn relies upon the global existence
of  solutions to the $(A(x,s),Z(x,s),W(x,s))$-system  \eqref{eq:ssWZA} for $ x \in \mathbb{R}$ and $ s \in [-\log (-t_0), \infty )$.   Since
 \begin{equation}\label{vladss}
 \p_\theta w(\theta, t) = e^s \p_x W (x, s)\,, \qquad  e^s= \frac{1}{\tau(t)-t}  \,,
\end{equation} 
by letting the blowup time  modulation variable $\tau(t)$ satisfy $\tau(0)=0$ and $\tau(T_*) = T_*$ and the blowup location modulation variable $\xi(t)$ satisfy $\xi(0)=0$ and
$\xi(T_*)=\theta_*$, we see that as $ s \to \infty $,  $|\p_\theta w(\theta_*,t) | \to \infty $ at a rate proportional to $ {\sfrac{1}{(T_*-t)}} $.  Note,  that all points $\theta$ which are not
equal to $\theta_*$, when converted to the self-similar variable $x$, are sent to $\pm \infty $ as $s \to + \infty $.  In the proof, we show that 
$\abs{\p_x W} \lesssim (1+ x^2)^{-{\sfrac{1}{3}} }$ and hence from
this bound,  it follows that  $\abs{W_x(e^{3s/2} (\theta-\xi),s)} \lesssim e^{-s} (\theta-\theta_*)^{- {\sfrac{2}{3}} } $, and from \eqref{vladss}, 
$\p_\theta w(\theta, t) $ does not blowup as
$t \to T_*$.

The  $(A,Z,W)$-system \eqref{eq:ssWZA} consists of transport type equations, which allow us to use $L^ \infty $-type estimates to construct global-in-time solutions in $C^4$.
We view the $W$ equation \eqref{e:W_eq} as producing the dominant dynamics, and the key to our analysis is a careful comparison of $W(x,s)$ with $\bar W(x)$.   In particular,
differentiation of the system  \eqref{eq:ssWZA} shows that the equations satisfied by $ \p_x^n W$, $\p_x^n Z$, and $\p_x^n A$ for $n=0,1,2,3,4$,  have either damping or anti-damping terms that depend on the solutions and their derivatives.   It is only when $n=4$ that a clear damping term emerges, while for $n=1$ and $n=2$, a very subtle analysis must be made
for the evolution equations of both  $\p_xW - \p_x \bar W$ and $\p_x^2 W- \p_x^2 \bar W$; a very delicate analysis allows us to find lower-bounds for the damping terms in
these equations by specially constructed rational functions
that are  found with the help of Taylor expansions of $\p_x \bar W$ near $x=0$ and $x= \infty $ (see, in particular,  \eqref{eq:V:main:damping} and \eqref{eq:Shaq}).
A bootstrap procedure is employed wherein we assume bounds for $(A,Z,W,\tau,\xi,\kappa)$ as well as their derivatives, and then proceed to close the bootstrap argument with even better bounds.

\subsection{Paper outline}  In Section \ref{sec:pure_azimuthal}, we consider the case that $\gamma=3$, and we have the simple
example of purely azimuthal   shock formation.   In this special case, the dynamics are reduced entirely to those of the Burgers equation.   The formation of shocks for
the 2D Euler equations with general adiabatic exponent $\gamma>1$ is then treated in Section \ref{sec:general_gamma}; a detailed description of the data is given,
the main theorem is stated, and the proof of  is given.  Concluding remarks are stated in Section \ref{sec:conclusion}.   We include Appendix \ref{sec:toolshed} which 
contains some important maximum-principle-type lemmas for solutions of non-locally forced and damped transport equations.

\section{Purely azimuthal waves and shocks: a simple example} \label{sec:pure_azimuthal}
In the case that $ \gamma=3$, some remarkable cancellations occur in the homogeneous solutions of the Euler equations which allow for
an exceedingly simple mechanism of shock formation, in which a smooth purely azimuthal wave travels around the circle, steepens and forms a shock wave which
can be continued for all time.   Our general construction of shock waves for all $\gamma>1$ will be a perturbation of this purely azimuthal shock wave solution, but we
shall first describe this simple solution.

For the most concise presentation, we shall consider the Euler equations posed on a two-dimensional annular domain
$A _{r_0,r_1} $
where $0< r_0 < r_1 < \infty$ with the standard no-flux boundary conditions
 $u_r|_{r=r_0} = u_r|_{r=r_1} = 0$.

In view of \eqref{eq:ansatz}, the no-flux boundary condition requires that  $a\equiv 0$ for all time.  Therefore, from equation \eqref{g3_a_evo}, we must have the relation 
\begin{equation}\label{g3_relation}
b^2 = \frac{2}{\gamma-1} P^{\gamma -1}
\end{equation} 
for all time. If we impose  condition \eqref{g3_relation} at $t=0$, an explicit computation verifies that the evolution equations 
 \eqref{eq:Euler:polar3} preserve the constraint \eqref{g3_relation} if and only if $\gamma = 3$, in which case, we have  that
$b= P$, 
and hence from \eqref{eq:riemann}, the Riemann invariants are given by
$$
w= 2b \ \ \text{ and } \ \ z=0 \,.
$$
Thus, with $a=0$ and $z=0$, the system \eqref{eq:euler:wza}
reduces to a single equation for the unknown $w$, which we identify as the 1D Burgers equation,
\begin{align}
 \partial_t w +  w \p_\theta w = 0 \, \qquad w(\theta, 0) =w_0(\theta) \,, \qquad \theta \in \T= [-\pi,\pi] \,,
\label{eq:w:Burgers}
\end{align}
solved on $\T$ with periodic boundary conditions.
It is well known that any initial datum $w_0$ which has a negative slope at a point forms a shock
(or infinite slope)
 in finite time. 
Note that for $\gamma=3$, the formula \eqref{omega_abp} shows that the vorticity $\vorticity = 2b=w$ and hence, $ \vorticity$ is nontrivial even for the purely azimuthal wave.
We shall sometimes use $w'$ to denote $\p_\theta w$.

\begin{theorem}[\bf Construction of the purely azimuthal shock]\label{thm:g3}
For $\gamma=3$, let $0<r_0< r_1$ be arbitrary, and consider initial datum $u_r = 0$, $u_\theta = \rho_0 = \tfrac 12 r w_0$, in $A_{r_0,r_1}$, where  $w_0 \in C^ \infty (\T)$ is such that $w_0 \ge \nu_0  >0$. Suppose that 
\begin{align} 
 \norm{w_0}_{L^\infty}  \leq 1  \,,  \label{t1_b1}
\end{align} 
and that there is a single point
$\theta_0 \in \T$ such that $w_0'(\theta_0) = \min_{\theta \in \T} w_0(\theta)$, and that 
\begin{align} 
\p_\theta w_0(\theta_0) = - \frac{1}{\eps }   \label{t1_b2}
\end{align} 
for some $\eps>0$. Then the solution $w$ of \eqref{eq:w:Burgers}, develops a singularity at time $T_* = \eps$ and angle $\theta_* = \theta_0 + \eps  w_0(\theta_0)$.
Moreover, the functions $u_r=0$, $u_\theta= \tfrac 12 r w(\theta,t)$, and $\rho= \tfrac 12 r w(\theta,t)$ form the unique smooth solution to the initial value problem for the Euler system \eqref{eq:Euler:polar} in  the domain $A_{r_0,r_1}$, on the time interval $[0,\eps)$. This solution satisfies the bounds
\begin{align} 
&  \sup_{t\in[0,T_*)} \left( \| \rho(\cdot , t)\|_{ L^ \infty (A_{r_0,r_1})} +  \| u(\cdot , t)\|_{ L^ \infty (A_{r_0,r_1})}\right) \leq 2 r_1  \,, \label{thm1_bnd1}\\
&  \sup_{t\in[0,T_*)} \left( \|\p_r \rho(\cdot , t)\|_{ L^ \infty (A_{r_0,r_1})} +  \|\p_r u(\cdot , t)\|_{ L^ \infty (A_{r_0,r_1})}\right) \leq 2 \,, \label{thm1_bnd1b}\\
& \lim_{t \to  T_*} \partial_\theta \rho(\theta_*, t) = \lim_{t \to  T_*} \partial_\theta u_\theta(\theta_*, t)  = - \infty  \, . \label{thm1_bnd2}
\end{align} 
The  vorticity and density satisfy
\begin{align} 
\nu_0 \leq  \vorticity(\theta , t) \leq  1, \qquad   \rho(r,\theta,t) \ge r_0 \tfrac{ \nu_0}{2}  \, ,\label{g3_vort_bounds}
\end{align} 
for all $\theta \in \TT$ and $t\in [0,\eps)$.
\end{theorem} 
\begin{proof}[Proof of Theorem~\ref{thm:g3}] 
For smooth initial datum $w_0$,  we solve \eqref{eq:w:Burgers}.    Differentiating \eqref{eq:w:Burgers} gives the
equation $\partial_t (\p_\theta \bb) +  \bb \p_\theta^2 \bb+ (\p_\theta \bb)^2 =0$.   Define the flow $\psi(\theta, t)$ by
$\p_t \psi(\theta,t) =  w(\psi(\theta,t),t)$ and $\psi(\theta,0)=\theta$.
Then  $\p_t (\p_\theta w\circ \psi ) + (\p_\theta w\circ \psi )^2 =  0$ so that $(\p_\theta w)\circ \psi = \tfrac{\p_\theta w_0}{t+ \p_\theta w_0}$ and $\psi(\theta,t) = \theta + t w_0(\theta)$. Hence from \eqref{t1_b2},  $\p_\theta w$ forms a shock at   time
 $T_*= \epsilon $ at the point $\theta_* = \theta_0 +t w_0(\theta_0)$, implying \eqref{thm1_bnd2}.
By the maximum principle and \eqref{t1_b1} we have
\begin{equation*}
\sup_{t\in [0, T_* )} \norm{w(\cdot , t)}_{L^\infty}  \leq  1 \qquad   \text{ and }  \qquad 
\min_{\theta \in \T} w(\theta,t) \ge  \nu_0 \,.
\end{equation*} 
The bounds \eqref{thm1_bnd1}--\eqref{thm1_bnd1b} and \eqref{g3_vort_bounds} follow directly from the definitions of $u_\theta, \rho, \vorticity$ and the above estimate. 
\end{proof}

\begin{remark}[\bf The Burgers solution continued after the singularity]
\label{rem:jump}
In Theorem~\ref{thm:g3} we have considered datum with a global (negative) minimum attained at a single point $\theta_0$, and thus $w_0''(\theta_0) = 0$ and $w_0'''(\theta_0)>0$. It is shown in~\cite[Proposition 9]{CoGhMa2018} that in the Burgers equation the finite time blowup arising from such initial datum is asymptotically self-similar and that the blowup profile is precisely the stable global-self similar profile $\bar W$ defined in \eqref{eq:barW:def}.  Moreover, at the blowup time $T_*=\eps$ the solution is H\"older $C^{\sfrac 13}$ smooth near the singular point.  To simplify the discussion, upon taking into account a Galilean transformation and a rescaling of the initial datum, we have that the blowup occurs at $\theta =0$ with speed $w(0,T_*) = 0$, and that $w(\theta,T_*) \sim \theta^{\sfrac 13}$ to leading order in  $\abs{\theta} \ll 1$. The solution of the Burgers equation may be continued in a unique way as an entropy  solution also after the blowup time $T_*$, starting from this H\"older $\sfrac 13$ initial datum, and we still denote this solution as $w(\cdot,t)$. We claim that instantaneously, for any $t>T_*$, the entropy  solution $w(\cdot,t)$ has a jump discontinuity, with the discontinuity propagating at the correct shock speed, given by the Rankine-Hugoniot condition. This phenomenon is explained in~\cite[Chapter 11]{EgFo2009}: for $t>T_*$ one may compute an explicit forward globally self-similar solution, and one notices that this  self-similar solution is not single-valued; we thus must have a jump in the solution $w$ at a location and a speed determined by the Rankine-Hugoniot condition. The argument in~\cite{EgFo2009} can be easily made precise by taking advantage of the Lax-Oleinik formula. For simplicity, let us consider initial datum $w(\theta,T_*) = \theta^{\sfrac 13}$, which allows us to perform explicit calculations. For $t>T_*$ the Lax-Oleinik formula tells us that the entropy  solution equals
\begin{align}
w(\theta,t) = \frac{\theta - (t-T_*)^{\sfrac 32} Y^3\left(\frac{\theta}{(t-T_*)^{\sfrac 32}}\right)}{t-T_*}
\label{eq:Burgers:ss}
\end{align}
where the function $Y = Y(q)$ is defined implicitly as the {\em the correct root} of the equation $Y^3 - Y = q$. This root is unique for $\abs{q} > \sfrac{2}{(3\sqrt{3})}$ and so the meaning of $Y(q)$ is clear; for $q \in [-\sfrac{2}{(3\sqrt{3})},0]$ we need to define $Y(q)$ as {\em the smallest root}, which is negative and has the limiting behavior $Y(0) = -1$;  while for $q \in (0,\sfrac{2}{(3\sqrt{3})}]$ the entropy solution requires us to take the {\em largest root}, which is positive and has the limiting behavior $Y(0^+)= +1$. Since the formula \eqref{eq:Burgers:ss} is explicit, it is easy to verify the above claims. We have $w(0^-,t) = w(0,t) = (t-T_*)^{\sfrac 12}$ and $w(0^+,t) = - (t-T_*)^{\sfrac 12}$. This shows that we have a discontinuity across the shock location $\theta = 0$, the left speed is larger than the right speed at the shock, and their average is $0$, which is why the shock location does not move with time. 
\end{remark}

\begin{remark}[\bf The Euler solution continued after the shock]\label{rem:Euler_continue}
For all $t \ge T_*$, let $\theta_*(t)$ denote the position of the discontinuity of $w( \cdot , t)$.   Now for all $\theta \neq \theta_*(t)$, $w( \cdot , t)$ is smooth and hence defines
a smooth solution to the Euler equations via the relations $\rho=u_\theta= \frac 12 r w$ and $u_r=0$.    By the Lax-Olienik formula, the shock moves with speed 
$\frac{d}{dt} \theta_*(t)= {\frac{1}{2}} (w^- + w^+)$, where 
$w^- =\lim_{ \theta \to \theta_*(t)^-} w(\theta,t)$ and $w^+=\lim_{ \theta \to \theta_*(t)^+} w(\theta,t)$. For $t>T_*$,  we denote by $\Gamma(t)$  the line segment given by
$\{(r,\theta) \colon \theta = \theta_*(t), r_1 \leq r \leq r_2 \}  $.  Then for a piecewise smooth function $f( \cdot , t) :A_{r_0,r_1} \to \mathbb{R}  $, which is discontinuous across $\Gamma(t)$, we let
$\llbracket f \rrbracket = f^{{-}}( \cdot , t) - f^{{+}}( \cdot ,t)$.   From the discontinuity of $w ( \cdot , t)$ we have that
$\llbracket \rho (\cdot,t) \rrbracket   > 0, \llbracket u_\theta \rrbracket  > 0, \llbracket u_r  \rrbracket  =0$.
Moreover, the Rankine-Hugoniot conditions require that 
$ \frac{d}{dt}\theta_*(t) = \frac{\llbracket \rho u_\theta \rrbracket}{\llbracket \rho \rrbracket}$.     But $\frac{\llbracket \rho u_\theta \rrbracket}{\llbracket \rho \rrbracket} = 
 {\frac{1}{2}} (w^- + w^+)$ and so the Rankine-Hugoniot condition is satisfied.  This shows that $(u_r,u_\theta,\rho)$ is a global entropy solution to the compressible Euler system with $\gamma=3$, which forms a shock at $T_*= \epsilon $, becomes discontinuous across the line segment $\Gamma(t)$ for times $t> \epsilon $, and propagates the shock with the correct shock speed.
\end{remark}
 
 \section{Formation of shocks for the Euler equations}\label{sec:general_gamma}
In this section, we construct a finite-time shock solution to the Euler equations for the general adiabatic constant $\gamma >1$. We achieve this by studying
the system of equations \eqref{eq:euler:wza} on the time interval $ - \eps \le t < T_*= \OO(\eps^{\sfrac 54})$, where $T_*$ is constructed in the proof  and $\eps \in (0,1)$ is a small parameter to be chosen later. We prove that a gradient blowup occurs at time $T_*$ for the variable $w$, whereas $\p_\theta z$ and $\p_\theta a$ remain bounded.

\subsection{Assumptions on the initial datum} 
In this subsection we describe the initial data that is used to construct the shock wave solutions to \eqref{eq:euler:wza}.
The initial time is given by $- \epsilon $, and the initial data is denoted as 
\[
w(\theta,-\eps) = w_0(\theta), \qquad z(\theta,-\eps) = z_0(\theta),\qquad a(\theta,-\eps) = a_0(\theta).
\] 
We  assume that $\partial_\theta w_0$ attains its global minimum at $\theta=0$, and moreover that 
\begin{align}
w_0(0) = \kappa_0 ,\qquad \partial_\theta w_0(0) = - \eps^{-1}, \qquad \partial_\theta^2 w_0(0) =0, \qquad  \partial_\theta^3 w_0(0) =  6 \eps^{-4} 
\,,
\label{eq:w0:power:series}
\end{align}
for some $\kappa_0 > 0$ to be determined later and whose main purpose is to ensure that the initial density is bounded from below by a positive constant (cf.~\eqref{eq:kappa:0}), and for an $0< \eps \ll 1$ to be determined.   We also assume that $w_0$ has its first four derivatives bounded as 
\begin{align}\label{eq:w0:bnd}
 \norm{\partial_\theta w_0}_{L^\infty}\leq \eps^{-1}, \quad \norm{\partial_\theta^2 w_0}_{L^\infty}\leq \eps^{-\sfrac 52}, \quad \norm{\partial_\theta^3 w_0}_{L^\infty}\leq 7 \eps^{-4},\quad \norm{\partial_\theta^4 w_0}_{L^\infty}\leq \eps^{-\sfrac{11}{2}}
\, ,
\end{align}
which are bounds consistent with \eqref{eq:w0:power:series}. 
In order to simplify the proof and to obtain a precise description of the solution's profile at the singular time (cf.~\eqref{eq:bootstrap:2*} and~\eqref{eq:cal:V:bootstrap} below), it is convenient to assume a slightly more precise behavior of $\partial_\theta w_0 $ near $\theta=0$. For this purpose we assume 
\begin{align}
 \abs{\eps (\partial_\theta w_0)(\theta)  -(\bar W_x)\left(\frac{\theta}{\eps^{\sfrac 32}}\right)} 
&\leq \min \left\{ \frac{(\frac{\theta}{\eps^{\sfrac 32}})^2}{40 (1 + (\frac{\theta}{\eps^{\sfrac 32}})^2 )} , \frac{1}{2 (8 + (\frac{\theta}{\eps^{\sfrac 32}})^{\sfrac 23})} \right\}
\label{eq:big:assume:ab}
\end{align}
for all $\theta \in \TT$, where  $\bar W$ is the stable globally self-similar solution to the Burgers equation defined in~\eqref{eq:barW:def}.

For $z$ and $a$ we assume that at the initial time we have
\begin{align}
\norm{z_0}_{C^n} + \norm{a_0}_{C^n}  \leq 1
\label{eq:z0:a0:Cn}
\end{align}
for $0 \leq n \leq 4$.
Furthermore, we assume that $w_0$, $z_0$, and $a_0$ all have compact support such that
\begin{align} 
\supp(w_0 (\theta) -\kappa_0) \cup \supp(z_0(\theta) ) \cup \supp(a_0(\theta) ) \subseteq  (-\sfrac{\pi}{2},\sfrac{\pi}{2})\,,  \label{eq:initialdata_support}
\end{align} 
and  in order to ensure the positivity of the initial density we assume that 
\begin{align}
\norm{w_0 (\cdot) - \kappa_0}_{L^\infty} \leq   \frac{\kappa_0}{2} \, ,
\label{eq:w0:C0}
\end{align}
and choose $\kappa_0$ suitably. Indeed, in order to ensure that $P_0(\theta) \geq \nu_0 >0$ for all $\theta \in \TT$, we simply choose any 
\begin{align}
\kappa_0 \geq 4 ( 2+ (\sfrac 2\alpha) (\sfrac{\nu_0}{2})^\alpha). 
\label{eq:kappa:0}
\end{align}
With this choice of $\kappa_0$, from \eqref{eq:riemann}, \eqref{eq:z0:a0:Cn}, and \eqref{eq:w0:C0} we have that $(\sfrac{2}{\alpha}) P_0^\alpha(\theta) = w_0(\theta) - z_0(\theta) \geq \sfrac{\kappa_0}{2} - 1 \geq (\sfrac{2}{\alpha}) \nu_0^\alpha $, thereby ensuring the desired strictly positive lower bound on the initial density.

\begin{remark}[\bf Consistency of the $w_0$ assumptions]
condition \eqref{eq:big:assume:ab}, which may be rewritten in terms of $x =  \theta \eps^{-\sfrac 32}$  as $\abs{\eps (\p_\theta w_0)(x \eps^{\sfrac 32}) 
- (\bar W_x)(x)}\leq \min\{ \frac{x^2}{40(1+x^2)}, \frac{1}{2(8 + x^{\sfrac 23})}\}$ for all $\abs{x} \leq \pi \eps^{-\sfrac 32}$, is consistent with 
\eqref{eq:w0:power:series}--\eqref{eq:w0:bnd} and with \eqref{eq:initialdata_support}--\eqref{eq:w0:C0}, meaning that we can find an open set of  initial conditions satisfying all of 
these assumptions. The first bound in the minimum of \eqref{eq:big:assume:ab} is required in order to ensure that near $\theta=0$ the deviation from the self-similar profile is 
parabolic; this is needed in view of \eqref{eq:w0:power:series} and the Taylor series   of $\bar W_x$ near the origin \eqref{eq:bar:Wx:zero}. The second condition 
in the minimum of \eqref{eq:big:assume:ab} is not required in order to prove a finite-time singularity theorem; rather, this assumption is needed  to characterize the blowup 
profile of $w(\theta,t)$ as $t\to T_*$ as being H\"older $C^{\sfrac 13}$ regular.  Lastly, we note that \eqref{eq:big:assume:ab} is consistent with $\p_\theta w_0$ being the derivative of a periodic function, which implies that it must have zero average and so $\p_\theta w_0$ cannot have a definite sign.  Since $\bar W_x(x) < 0$ for all $x\in \RR$, it is important that for $\abs{x}\gg 1$,
the envelope determined by the second term on the right side of \eqref{eq:big:assume:ab} allows $\p_\theta w_0$ to become positive.  Indeed, note that in the Taylor series   of $\bar W_x$ around infinity~\eqref{eq:bar:Wx:infinity},
the coefficient of $x^{-\sfrac 23}$ is $-\sfrac 13$, while the coefficient of $x^{-\sfrac 23}$ in the Taylor series about infinity of the right side of \eqref{eq:big:assume:ab} is 
$\sfrac 12 > \sfrac 13$, which allows $\p_\theta w_0$ to take on positive values.
\end{remark}

\begin{remark}[\bf $L^\infty$ estimates for the solution]
\label{rem:Linfinity}
Using assumptions \eqref{eq:z0:a0:Cn}, \eqref{eq:w0:C0}, and the fact that \eqref{eq:euler:wza} is a system of forced transport equations in which the forcing terms show no derivative loss, we deduce via the maximum principle that 
\begin{align}
\norm{w(t)}_{L^{\infty}} + \norm{z(t)}_{L^{\infty}} + \norm{a(t)}_{L^{\infty}}\leq   M
\label{eq:L:infinity}
\end{align}
holds for any $M\geq 4 + 2 \kappa_0$, and all times $t$ which are sufficiently small with respect to $\kappa_0$. This argument is detailed upon in Proposition~\ref{prop:ballistic} below, cf.~estimate~\eqref{eq:improved:L:infty}. In particular, these amplitude bounds hold for all $t\in [0,T_*)$ since $T_* = \OO(\eps^{\sfrac 54})$, and we take $\eps$ to be sufficiently small, in terms of $\kappa_0$.
\end{remark}

\begin{remark}[\bf The spatial support of the solution and an extension from $\TT$ to $\RR$]
\label{rem:support}
Using  \eqref{eq:L:infinity} we obtain that the transport speeds on the left side of \eqref{eq:euler:wza} are bounded solely in terms of $M$. Therefore, assuming $\eps$ to be sufficiently small depending on $M$ and using that the length of $[-\eps,T_*)$ is less than $2 \eps$, by finite speed of propagation the solution $(w,z,a)$ of \eqref{eq:euler:wza} restricted to the region $\TT \setminus [-\frac{3\pi}{4},\frac{3\pi}{4}] $ is uniquely determined by the initial data $(w_0,z_0,a_0)$ on the set $\TT \setminus [-\frac{\pi}{2},\frac{\pi}{2}] $, for all times $t \in [-\eps,T_*]$. In particular, as a consequence of the support assumption \eqref{eq:initialdata_support}, on the region $\TT \setminus [-\frac{3\pi}{4},\frac{3\pi}{4}]$,  the solution $(w,z,a)$ is constant in the angle $\theta$ (albeit a time dependent constant), for all times $t \in [-\eps,T_*]$.  Hence by abuse of notation we may extend the domain of $(w,z,a)$ to $\theta\in \RR$, by setting $w(\theta,t)=w(\pi,t)$, $z(\theta,t)=z(\pi,t)$, and $a(\theta,t)=a(\pi,t)$ for $\abs{\theta}>\pi$.  In what follows we adopt this abuse of notation, with the knowledge that the true solution is defined to be the periodization of the restriction to $[-\pi,\pi)$ of the extended solution. 
Also, we shall use implicitly throughout the proof that 
$
\supp(\partial_\theta w) \cup \supp(\partial_\theta  z) \cup \supp(\partial_\theta a) \subseteq  [-\sfrac{3\pi}{4},\sfrac{3\pi}{4}] \, .
$
\end{remark}

\subsection{Statement of the main result} \label{sec:main:results}

\begin{theorem}[\bf Formation of shocks for Euler] \label{thm:general}
Let $\gamma>1$, $ \alpha = {\tfrac{\gamma-1}{2}}$,  $0<R_0<r_0  < r_1< R_1 < \infty$, and $\nu_0>0$.   Then, there exist a sufficiently large $\kappa_0 = \kappa_0(\alpha,\nu_0) > 0$, a sufficiently large $M = M(\alpha,\kappa_0,\nu_0) \geq 1$, and a sufficiently small $\eps = \eps(\alpha,\kappa_0,\nu_0,M,R_0,R_1,r_0,r_1) \in (0,1)$ such that the following holds. 

\underline{Assumptions on the initial data.}\,  Consider initial datum  for the Euler equations  \eqref{eq:Euler:polar}, given at initial time $t_0=-\eps$ given as  follows:
$$u_{r}(r,\theta,t_0) = r a_0(\theta)\,,   u_{\theta}(r,\theta,t_0) = r b_0(\theta)\,,  \text{ and } \rho_0(r,\theta,t_0) =r^{\sfrac{1}{\alpha}} P_0(\theta) \ \text{ for } \ (r,\theta) \in A_{R_0,R_1} \,, $$
where $(a_0,b_0,P_0) \in C^ 4 (\TT)$ and  $P_0\ge \nu_0 >0$.    Define $w_0 = b_0 + \tfrac{1}{\alpha } P_0^\alpha $, $z_0=b_0 - \tfrac{1}{\alpha} P_0^\alpha$, and suppose that
 $(w_0,z_0,a_0)$ satisfy assumptions \eqref{eq:w0:power:series}--\eqref{eq:w0:C0}.

\underline{Shock formation for $(a,z,w)$-system \eqref{eq:euler:wza}.}\,  
There exists a unique solution $(a,z,w) \in C([-\eps,T_*); C^4(\T))$ to  \eqref{eq:euler:wza} which blows up in  asymptotically self-similar fashion  at time $T_*$ and angle $\theta_*$, such that:
\begin{itemize} 
\item the blowup time $T_* = \OO(\eps^{\sfrac 54})$ and angle  $\theta_*  = \OO( \epsilon )$ are explicitly computable, with $\theta_*  = \lim_{t\to T_*} \xi(t)$,
\item $\sup_{t\in[-\eps, T_*)} \left( \| a\|_{ W^{1, \infty }(\T)} + \| z\|_{ W^{1, \infty} (\T)} + \| w\|_{ L^\infty(\T)} \right) \leq  C(M),$
\item $\lim_{t \to T_*}   \p_\theta w(\xi(t),t) = -\infty $ and we have $\frac{1}{2(T_*-t)} \leq  \norm{\p_\theta w(\cdot,t)}_{L^\infty} \leq \frac{2}{T_*-t}$ as  $t \to T_*$,
\item $w( \cdot , T_*)$ has a cusp singularity of H\"{o}lder $C^ {\sfrac{1}{3}} $ regularity.
\end{itemize}

\underline{Shock formation for the Euler equations \eqref{eq:Euler:polar}.}\,  
Setting $ b= \tfrac{w+z}{2}$ and $P = (\frac{\alpha}{2}(w-z))^{\sfrac{1}{\alpha}}$, we define $(u_r,u_\theta,\rho)$ by \eqref{scale0}.
Consider the time-dependent domain $\Omega(t)$ defined in \eqref{Omega} such that $\Omega(t) \subset   A_{R_0,R_1}$ for all  $t\in [- \eps, T_*]$.
 Then, $(u_r,u_\theta,\rho) \in C\left([-\eps,T]; C^4(\Omega(t)) \right)$  is a unique solution to the Euler equations \eqref{eq:Euler} on the domain 
 $\Omega(t)$ for all $-\eps\le t \leq T$, for any $T< T_*$,  and
\begin{align} 
&\quad \lim_{t\to T_*} \p_\theta u_\theta ( r, \xi(t), t) = \lim_{t\to T_*} \p_\theta \rho ( r, \xi(t), t) = -\infty  \qquad \text{ for all } r \in \Omega(t) \,, \label{eq:thm:blowup} \\
&  \sup_{t\in[-\eps,T_*)}\sum_{k=0}^1 \left( \| \p_r^k\rho(\cdot , t)\|_{ L^ \infty (\Omega(t))} +  \| \p_r^ku(\cdot , t)\|_{ L^ \infty (\Omega(t))}\right) + 
\| \p_\theta u_r(\cdot , t)\|_{ L^ \infty (\Omega(t))}  \le  C(R_1,M) \label{eq:thm:bnd:1} \, .
\end{align} 
The shock occurs along the line segment
  $\Gamma(T_*):=\{(r,\theta) \in  \Omega(T_*) \colon \theta=\theta_*\}$. 
The graphs of the blowup profiles $u_\theta(r,\theta,T_*)$ and $\rho(r,\theta,T_*)$ are surfaces with cusps
along $\Gamma(T_*)$ and are H\"{o}lder $C^{ {\sfrac{1}{3}} }$ smooth.

\underline{Non-trivial vorticity and density at the shock.}\,  
The  vorticity   and density satisfy
$$  \frac{1}{M^2} \leq  \upomega(\theta , t) \le  M^2 \,, \qquad  \rho(r,\theta,t) \ge \frac{R_0^{\sfrac{1}{\alpha}} \nu_0}{2}>0 \,,$$
for all $(r,\theta) \in \Omega(t)$ and $t \in [-\eps, T_* )$.
\end{theorem}

\begin{remark}
\label{rem:Omega:T*}
With $u=(u_r,u_\theta)$, 
the flow $\eta_u$ solving $\p_t \eta_u = u \circ \eta_u$ with initial datum $\eta_u(r,\theta,-\eps) = (r,\theta)$ is well defined and smooth on the time interval $[-\eps, T]$ for all
$T< T_*$.   Moreover, since $\eta_u(r,\theta, t) = (r,\theta) + \int_{-\eps}^t (u \circ \eta_u)(r,\theta, s) ds$, by \eqref{eq:thm:bnd:1}, we see  that  
$$
\sup_{[-\eps,T_*)}\| \eta_u(\cdot , t)\|_{ L^ \infty (A_{r_0,r_1})} \le C
$$
Hence, by dominated convergence, we may define $\eta_u(r, \theta, T_*) = \lim_{t \to T_*} \eta_u(r, \theta,t )$.   Thus, the set $\Omega(T_*)$ is well defined.
\end{remark}

\begin{remark}  We have established that at the initial singularity time $t=T_*$, both $u_\theta$ and $\rho$ have cusp singularities with $C^ {\sfrac{1}{3}} $ regularity.   For the
case that $\gamma=3$, we have explained how this cusp singularity develops an instantaneous discontinuity and is propagated as a shock wave.   In Section \ref{sec:conclusion} we 
  conjecture that the same is true for the more general solution constructed in the previous theorem.  We note that Alinhac \cite{Al1999a, Al1999b} proved the formation of cusp-type singularities for solutions of a
quasilinear wave equation, but the Euler equations do not satisfy the structure of his equations.
\end{remark}

\begin{corollary}[\bf Open set of initial conditions]
\label{cor:open:set:IC}
The conditions on the initial data $(a_0,z_0,w_0)$ in Theorem \ref{thm:general} may be relaxed so that they may be taken to be in an open neighborhood in the $C^4$ topology.
\end{corollary}
\begin{proof}[Proof of Corollary~\ref{cor:open:set:IC}] 
First note that since the system \eqref{eq:euler:wza} has finite speed of propagation, the   support properties  of the initial data described in~\eqref{eq:initialdata_support} (see also~ Remark~\ref{rem:support}) are stable under small perturbations in the $C^4$ topology. Second, note that $\kappa_0$ and $\eps$ are free to be taken in an open set (sufficiently large, respectively sufficiently small), and hence the values of $w_0(0)$ and $\partial_\theta w_0(0)$ stated in \eqref{eq:w0:power:series} can be taken in an open set of possible values. Next, observe that if $\norm{\partial_\theta^4 w_0}_{L^\infty}\leq \eps^{-\sfrac{11}{2}}$ holds (condition which is stable under small $C^4$ perturbations) then a Taylor expansion around the origin yields
\begin{align*}
\partial_{\theta}^2 w_0(\theta)&=\partial_{\theta}^2 w_0(0)+\theta \partial_{\theta}^3 w_0(0)+\OO(\eps^{-\sfrac{11}{2}}\theta^2)\\
&=\partial_{\theta}^2 w_0(0)+6\eps^{-4}\theta +\theta(\partial_{\theta}^3 w_0(0)-6\eps^{-4})+\OO(\eps^{-\sfrac{11}{2}}\theta^2)
\, .
\end{align*}
Hence by continuity, for any $\bar \epsilon>0$ depending on $\eps$, if one assumes $\partial_{\theta}^2 w_0(0)$ and $\partial_{\theta}^3 w_0(0)-6\eps^{-4}$ to be sufficiently small, there exists an $\theta_0$ satisfying  $\abs{\theta_0}\leq \bar \eps$ such that $\partial_{\theta}^2 w_0(\theta_0)=0$. Hence by  the change of coordinates $\theta\mapsto \theta+\theta_0$, and taking $\bar \eps$ to be sufficiently small, we can relax the condition $\partial_\theta^2 w_0(0) = 0$ to the condition that  $\partial_\theta^2 w_0 = 0$ is in a sufficiently small neighborhood of $0$ and that $\partial_{\theta}^3 w_0(0)$ lies in a sufficiently small neighborhood of $6\eps^{-4}$. Next, note that the rescaling $w_0(\theta) \mapsto \mu^{-1}w_0(\mu \theta)$, rescales $\partial_\theta^3 w_0(0)$ and leaves $\partial_\theta w_0(0)$ unchanged.  Strictly speaking, such a rescaling would modify the domain; however, since our analysis only concerns a strict subset of the domain (due to \eqref{eq:initialdata_support}), and we have finite speed of propagation, as long as $\mu$ is sufficiently close to $1$ this $\mu$-rescaling does not pose an issue. Setting 
\[\tilde a(\theta,t)=\mu^{-1}a(\mu \theta,t),\quad\tilde w(\theta,t)=\mu^{-1}w(\mu \theta,t),\quad \tilde z(\theta,t)=\mu^{-1}z(\mu \theta,t)\,,\] 
the equation satisfied by $(\tilde a, \tilde w, \tilde z)$ is of the form  \eqref{eq:euler:wza}, with the right hand side rescaled by a factor of $\mu$. As long as $\mu$ is sufficiently close to $1$,   this rescaling has no effect on the proof of Theorem \ref{thm:general}. Thus the condition on $\partial_\theta^3 w_0(0)$ may be relaxed to the condition that $\partial_\theta^3 w_0(0)$ lies in a sufficiently small neighborhood of $6\eps^{-4}$. Finally, note that for $\theta$ small, \eqref{eq:big:assume:ab} is implied by \eqref{eq:w0:power:series} and \eqref{eq:w0:bnd}. For $\theta$ away from a small neighborhood of $0$, the condition \eqref{eq:big:assume:ab} is an open condition. Thus \eqref{eq:big:assume:ab} does not pose an impediment to taking the initial data to lie in an open set.
\end{proof}

\subsection{Self-similar variables and solution ansatz}\label{sec:self-similar}
For the purpose of satisfying certain normalization constraints on the developing shock,
we introduce three dynamic variables $\tau, \xi, \kappa \colon [-\eps,T_*] \to \RR$, and fix their initial values as at time $t= -\eps$ as
\begin{align}
\tau(-\eps) = 0, \qquad \xi(-\eps) = 0,\qquad \kappa(-\eps) = \kappa_0.
\label{eq:modulation:initial:time}
\end{align}
The blowup time $T_*$ and the blowup location $\theta_*$ are defined precisely in Remark~\ref{rem:blowup:details}. For the moment we only record that $T_* = \OO(\eps^{\sfrac 54})$, $\tau(T_*) = T_*$, and that by construction we will ensure $\tau(t) > t$ for all $t\in [-\eps,T_*)$ (see~Remark~\ref{rem:blowup:details} below).

We introduce the following self-similar variables
\begin{align}
x(\theta,t):=\frac{\theta-\xi(t)}{(\tau(t)-t)^{\frac32}} \, ,  \qquad s(t):=-\log(\tau(t)-t)\,.
\label{eq:x:s:def}
\end{align}
The blowup time is defined by the relation $\tau(T_*) = T_*$. In the self-similar time, the blowup time corresponds to $s\to + \infty$.  
We will use frequently the identities 
\begin{align*}
\tau-t=e^{-s},\qquad \frac{ds}{dt}=\frac{1-\dot \tau}{\tau-t}=(1-\dot \tau)e^s,
\end{align*}
where we adopt the notation $\dot f=\frac{df}{dt}$, and 
\begin{align*}
x = e^{\frac32s} (\theta-\xi(t)), \qquad 
\partial_{\theta} x= e^{\frac32s},\qquad  \partial_t x =\frac{-\dot \xi}{(\tau-t)^{\frac32}}-\frac{3(\dot \tau-1)(\theta-\xi)}{2(\tau-t)^{\frac52}}=-e^{\frac32s}\dot \xi+\frac32(1-\dot\tau)xe^s \, .
\end{align*}
Notice that at $t=- \eps $, we have $s= - \log \eps $ and hence $e^{-s}= \eps$.

Using the self-similar variables $x$ and $s$ we rewrite $w$, $z$ and $a$ as 
\begin{align}
w(\theta,t)=e^{-\frac s2}W(x,s)+\kappa (t)\,, \qquad 
z(\theta,t)=Z(x,s) \,, \qquad 
a(\theta,t)=A(x,s) \, .
\label{eq:ss:ansatz}
\end{align}
As mentioned in Remark~\ref{rem:support}, the functions $(W,Z,A)$ are defined on all of $\RR$, but they are constant in $x$ on the complement of the expanding set $\{ x\colon - \frac{3\pi}{4} e^{\sfrac{3s}{2}} \leq x \leq    \frac{3\pi}{4} e^{\sfrac{3s}{2}} \}$. 

Inserting the ansatz \eqref{eq:ss:ansatz} in the system \eqref{eq:euler:wza}, we obtain that  $W$, $Z$ and $A$ satisfy the equations 
\begin{align*}
&(1-\dot \tau)\left(\partial_s-\tfrac12\right) W+ \left(e^{\frac s2}\left(\kappa -\dot \xi+\tfrac{1-\alpha}{1+\alpha}Z\right)+\tfrac32(1-\dot\tau)x+W\right)\p_xW
\notag\\&\qquad\qquad
=- e^{-\frac s2} \dot \kappa  - A e^{-\frac s2} \left( \tfrac{1-2\alpha}{1+\alpha} Z-  \tfrac{3+2\alpha}{1+\alpha}(e^{-\frac s2}W+\kappa)\right) \\
&(1-\dot \tau)\partial_sZ+ \left(e^{\frac s2} \left(\tfrac{1-\alpha}{1+\alpha} \kappa -\dot \xi\right)+\tfrac{1-\alpha}{1+\alpha}W+\tfrac32(1-\dot\tau)x+  e^{\frac s2}Z\right)\p_xZ 
\notag\\&\qquad\qquad
= - A e^{-s} \left( \tfrac{1-2\alpha }{1+\alpha} (e^{-\frac s2}W+\kappa) -  \tfrac{3+2\alpha}{1+\alpha} Z\right) 
 \\
&(1-\dot \tau)\partial_s A+ \left(e^{\frac s2}\left(\tfrac{1}{1+\alpha}(Z+\kappa)-\dot \xi\right)+\tfrac{1}{1+\alpha}W+\tfrac32(1-\dot\tau)x\right)\p_xA
\notag\\&\qquad\qquad
= \tfrac{1}{2(1+\alpha)}e^{-s} \left(-4A^2+ (e^{-\frac s2}W + \kappa +Z)^2-\alpha  (e^{-\frac s2}W +\kappa -Z)^2\right)\,.
\end{align*}
It is convenient to introduce the transport speeds  
\begin{subequations}
\begin{align}
g_W &:= \tfrac{1}{1-\dot\tau} e^{\frac s2} \left(\kappa-\dot \xi+\tfrac{1-\alpha}{1+\alpha}Z\right),
\label{eq:gW:def}\\ 
g_Z &:=\tfrac{1}{1-\dot\tau} \left(e^{\frac s2}\left(\tfrac{1-\alpha}{1+\alpha} \kappa - \dot \xi\right)+\tfrac{1-\alpha}{1+\alpha}W \right),
\label{eq:gZ:def} \\
g_A &:= \tfrac{1}{1-\dot\tau} \left(e^{\frac s2}\left(\tfrac{1}{1+\alpha}(Z+\kappa)-\dot \xi\right)+\tfrac{1}{1+\alpha}W \right),
\label{eq:gA:def}
\end{align}
\end{subequations}
and the forcing terms
\begin{align*}
F_W &:=-\tfrac{e^{-\frac s2}}{(1+\alpha)(1-\dot \tau)} \left(  (1-2\alpha) A Z- (3+2\alpha) A (e^{-\frac s2}W+\kappa)   \right),\\
F_Z &:=-\tfrac{e^{-s}}{(1+\alpha)(1-\dot \tau)} \left(    (1-2\alpha)A (e^{-\frac s2}W+\kappa) - (3+2\alpha)A Z   \right),\\
F_A &:=\tfrac{e^{-s}}{2(1+\alpha)(1-\dot \tau)}\left(-4 A^2+ (e^{-\frac s2}W+\kappa+Z)^2-\alpha  (e^{-\frac s2}W+\kappa-Z)^2 \right) \, ,
\end{align*}
so that we can rewrite the evolution equations for $W$, $Z$ and $A$ as
\begin{subequations}
\label{eq:ssWZA}
\begin{align}
 \left(\partial_s-\tfrac12\right) W+ \left(g_W+\tfrac{3x}{2}+\tfrac{1}{1-\dot \tau} W\right)\p_xW  
&=- e^{-\frac s2} \tfrac{\dot \kappa}{1-\dot \tau} +  F_W \,, \label{e:W_eq}\\
 \partial_sZ+ \left(g_Z+\tfrac{3x}{2}+\tfrac{ 1}{1-\dot \tau} e^{\frac s2}Z \right)\p_xZ 
&= F_Z  \,,  \label{e:Z_eq}\\
 \partial_s A+ \left(g_A+\tfrac{3x}{2}\right)\p_xA
&= F_A \,. \label{e:A_eq}
\end{align}
\end{subequations} 
As long as the solutions remain smooth, the $(W,Z,A)$ system \eqref{eq:ssWZA} is equivalent to the original $(w,z,a)$ formulation in \eqref{eq:euler:wza}. In particular, the local well-posedness of \eqref{eq:ssWZA} from $C^4$-smooth initial datum of compact support follows from the corresponding well-posedness theorem for \eqref{eq:euler:wza}. The purpose of this section is to show that the dynamic modulation variables $(\kappa, \xi, \tau)$ remain uniformly bounded in $C^1$ and that the functions $(W,Z,A)$ remain uniformly bounded in $C^4$ for all $s\in [-\log \eps,\infty)$.
Taking into account the self-similar transformation \eqref{eq:x:s:def}--\eqref{eq:ss:ansatz}, and in view of the continuation criterion~\eqref{eq:continuation:criterion}, this means that no singularities occur prior to time $t= T_*$. Additionally, we will ensure that $\p_xW(0,s) = -1$ for all $s\geq -\log \eps$, which in turn implies through the self-similar change of coordinates that $\p_\theta w$ blows up as $\sfrac{-1}{(T_*-t)}$ as $t\to T_*$.

\begin{remark}[\bf The stable globally self-similar solution of the 1D Burgers equation]
\label{rem:Burgers:SS}
We view the   evolution \eqref{e:W_eq} as a perturbation of the 1D Burgers dynamics. Indeed, if we set $g_W = \dot \tau = \dot \kappa = F_W \equiv 0$ in \eqref{e:W_eq}, the resulting steady equation is  the globally self-similar version of the 1D Burgers equation as described in~\eqref{eq:barW:dx}. We recall that this steady globally self-similar solution $\bar W$ given explicitly by \eqref{eq:barW:def}, and that its Taylor series expansions of $\p_x\bar W$ at $x=0$ and $x= \infty $, respectively,  are given by
\begin{subequations}
\begin{align}
\p_x\bar W &=  -1 + 3x^2 - 15 x^4 + \OO(x^6) \quad \mbox{for} \quad |x|\ll 1 \,,
\label{eq:bar:Wx:zero} \\
\p_x\bar W &=  - \tfrac{1}{3 }x^{- \frac 23} - \tfrac{1}{9 } x^{- \frac 43} + \OO(x^{-\frac 83}) \quad \mbox{for} \quad |x|\gg 1
\label{eq:bar:Wx:infinity} \,.
\end{align}
\end{subequations}
In the proof of our estimates for $\p_x W$ and $\p_{xx} W$ we will use a number of properties for $\bar W$, which may be checked directly using its explicit formula~\eqref{eq:barW:def}.
\end{remark}

At this stage it is convenient to record the differentiated version of the system \eqref{eq:ssWZA}.
For  $n  \in \NN$, after applying $\partial_x^n$ to \eqref{eq:ssWZA}  we obtain from the Leibniz rule that
\begin{subequations}
\begin{align}
 \left(\partial_s+ \tfrac{3n -1}{2} + \tfrac{n+{\bf 1}_{n\neq1}}{1-\dot \tau} \partial_x W+ n \partial_x g_W  \right) \partial^n_x W
 &+ \left(g_W+\tfrac{3x}{2} +\tfrac{1}{1-\dot \tau}W \right) \partial^{n+1}_x W
= F_W^{(n)}
\label{eq:multi:W}\\
 \left(\partial_s+ \tfrac{3n}{2} + \tfrac{n+{\bf 1}_{n\neq 1}}{1-\dot \tau}  e^{\frac s2} \partial_x Z+ n \partial_x g_Z  \right) \partial^n_x Z
 &+ \left(g_Z+\tfrac{3x}{2}+\tfrac{1}{1-\dot \tau} e^{\frac{s}{2}}Z \right) \partial^{n+1}_x Z 
= F_Z^{(n)}
\label{eq:multi:Z}\\
\left( \partial_s + \tfrac{3n}{2}   + n \partial_x g_A   \right) \partial^n_x A 
&+ \left(g_A+\tfrac{3x}{2} \right)\partial^{n+1}_x A  
= F_A^{(n)}
\label{eq:multi:A}
\end{align}
\end{subequations}
where the forcing terms are given by 
\begin{align*}
 F_W^{(n)} 
 &:= \partial^n_x F_W 
 -  {\bf 1}_{n\geq 2}   \partial^{n}_x g_W \partial_x W - {\bf 1}_{n\geq 3} \sum_{k=2}^{n-1} {n \choose k}  \left( \tfrac{1}{1-\dot \tau} \partial^{k}_x W   
 +    \partial^{k}_x g_W \right) \partial^{n-k+1}_x W   \\
 F_Z^{(n)} 
 &:= \partial^n_x F_Z 
 -  {\bf 1}_{n\geq 2}   \partial^{n}_x g_Z \partial_x Z - {\bf 1}_{n\geq 3} \sum_{k=2}^{n-1} {n \choose k} \left(  \tfrac{1}{1-\dot \tau} e^{\frac s2} \partial^{k}_x Z + \partial^{k}_x g_Z \right) \partial^{n-k+1}_x  Z 
 \\
 F_A^{(n)} 
 &:= \partial^n_x F_A 
 - {\bf 1}_{n\geq 2} \sum_{k=2}^n {n \choose k}    \partial_x^{k} g_A\, \partial^{n-k+1}_x A     
 \, .
\end{align*}

\subsection{Constraints on $W$ at $x=0$ and the definitions of the modulation variables}
Inspired by the self-similar analysis of the 1D Burgers equation in~\cite{CoGhMa2018}, we impose the following constraints at $x=0$, which fully characterize the developing shock:
\begin{align}
W(0,s)=0,\qquad \p_xW(0,s)=-1,\qquad \p_x^2W(0,s)=0 \,.
\label{eq:constraints}
\end{align}
These constraints will fix our choices of $\tau(t)$, $\xi(t)$, and $\kappa (t)$.
In order to compactly write the computations in this section,  we shall denote
\begin{align}
\varphi^0(s) = \varphi(0,s)\,, \qquad \varphi_x(x,s) = \p_x \varphi(x,s) \,, \ \varphi_{xx}(x,s) = \p_x^2 \varphi(x,s)\,, \ \text{ etc.}
\label{eq:funky:notation}
\end{align}
for any function $\varphi = \varphi(x,s)$.

In view of \eqref{eq:constraints}, in  addition to \eqref{e:W_eq} we need to record \eqref{eq:multi:W} for $n = 1$ and $n = 2$. Using \eqref{eq:multi:W}  we spell out these two equations
\begin{subequations}
\begin{align}
\left( \partial_s+1 +\tfrac{1}{1-\dot \tau}W_x +\tfrac{(1-\alpha)}{(1+\alpha)(1-\dot \tau)}e^{\frac s2}Z_x  \right)W_x+ \left(g_W+\tfrac{3x}{2}+\tfrac{1}{1-\dot \tau}W\right)W_{xx}  
&= F_{W}^{(1)}\label{eq:Wx:eq}\\
\left( \partial_s+\tfrac52 +\tfrac{3}{1-\dot \tau} W_x +\tfrac{2(1-\alpha) }{(1+\alpha)(1-\dot \tau)}e^{\frac s2}Z_x  \right)W_{xx}+ \left(g_W+\tfrac{3x}{2} +\tfrac{1}{1-\dot\tau}W\right)W_{xxx}  
&= F_{W}^{(2)}\label{eq:Wxx:eq}
\end{align}
\end{subequations}
where the forcing terms are given by 
\begin{align}
F_{W}^{(1)} :=\partial_x F_W \, , \qquad \mbox{and} \qquad 
F_{W}^{(2)} :=\partial_{xx} F_W-\tfrac{1-\alpha}{(1+\alpha)(1-\dot \tau)}e^{\frac s2} Z_{xx}W_x \, .\label{e:ForcingW}
\end{align}

Using the notation~\eqref{eq:funky:notation}, and  inserting the constraints \eqref{eq:constraints} into \eqref{eq:Wx:eq} we arrive at 
\begin{align*}
 -  \dot \tau + \tfrac{(1-\alpha)}{1+\alpha} e^{\frac s2} Z_x^0(s)
 &=- (1-\dot \tau) F_W^{0,(1)}(s) \,,
\end{align*}
which implies  that
\begin{align}
\label{eq:tau:dot}
\dot \tau 
= \tfrac{1-\alpha}{1+\alpha} e^{\frac s2} Z_x^0(s)    
- e^{-\frac s2}  \left( \tfrac{1-2\alpha}{1+\alpha} (AZ)_x^0(s) - \tfrac{3+2\alpha}{1+\alpha} (\kappa A_x^0(s)- e^{-\frac s2}A^0(s))\right)    \,.
\end{align}
Plugging in the constraints \eqref{eq:constraints} into  \eqref{e:W_eq} and \eqref{eq:Wxx:eq},  we further obtain that
\begin{subequations}
\begin{align}
 -g_W^0(s)  
 &= F_W^0(s) -  \tfrac{1}{1-\dot \tau} e^{-\frac s2} \dot \kappa
 \label{eq:gW:1} \\
 g_W^0(s) W_{xxx}^0(s)  
 &= F_W^{0,(2)}(s)
  \label{eq:gW:2} 
\, .
\end{align}
\end{subequations}
Since we will prove that $W_{xxx}^0(s) \geq 5$, we solve the  system \eqref{eq:gW:1}--\eqref{eq:gW:2} as 
\begin{subequations}
\begin{align}
\dot \xi - \kappa - \tfrac{1-\alpha}{1+\alpha} Z^0(s)
&= - (1-\dot \tau)  e^{-\frac s2}  \frac{F_W^{0,(2)}}{W_{xxx}^0(s)} \,,
\label{eq:dot:xi}\\
\dot \kappa
&= (1-\dot \tau) e^{\frac s2} \left( F_W^0(s) + \frac{F_W^{0,(2)}}{W_{xxx}^0(s)} \right) \,.
\label{eq:dot:kappa}
\end{align}
\end{subequations}
The equations \eqref{eq:tau:dot}, \eqref{eq:dot:xi}, and \eqref{eq:dot:kappa} are the evolution equations for the dynamic modulation variables which are used in the proof. We also note here that in view of \eqref{eq:gW:def} and \eqref{eq:dot:xi} we may write 
\begin{align}
g_W(x,s) = \frac{F_W^{0,(2)}}{W_{xxx}^0(s)} + \frac{(1-\alpha) }{(1+\alpha)(1-\dot \tau)}e^{\frac s2}  \left( Z(x,s) - Z^0(s) \right),
\label{eq:gW:est}
\end{align} 
which provides us with a useful bound for $g_W$ for $|x|\les 1$.

\subsection{Bootstrap assumptions}
For the dynamic modulation variables,  we assume that
\begin{subequations}
\begin{align}
&\abs{\kappa(t)} \leq 2 \kappa_0, \qquad \abs{\tau(t)} \leq \eps^{\frac 54} , \qquad \abs{\xi(t)} \leq  6 M \eps 
\label{eq:speed:bound} \\
&\abs{\dot \kappa(t)} \leq M^3, \qquad \abs{\dot \tau(t)} \leq \eps^{\frac 14} , \qquad \abs{\dot \xi(t)} \leq 3M
\label{eq:acceleration:bound}
\end{align}
\end{subequations}
for all $t < T_*$.

Note that from \eqref{eq:L:infinity} and \eqref{eq:speed:bound} we deduce that (we use $\kappa_0 \leq M$)
\begin{align}
\norm{W(s)}_{L^{\infty}} \leq 2 M  e^{\frac s2}  \qquad \mbox{and} \qquad \norm{Z(s)}_{L^\infty} + \norm{A(s)}_{L^\infty} \leq M
\label{eq:bootstrap:-1}
\end{align}
for all $s\geq - \log \eps$. Therefore, no bootstrap assumptions are needed for the $C^0$ norms of $(W,A,Z)$.

For the higher order derivatives of $W$ we assume the following   estimates for all times $s\geq -\log \eps$ 
\begin{align}
\norm{\partial_x^3 W}_{L^{\infty}} \leq M^{ \frac 34 }, \quad \norm{\partial_x^4 W}_{L^{\infty}} \leq M \, .
\label{eq:bootstrap:1}
\end{align}
We further assume the more precise bounds
\begin{align}
\abs{W_{x}(x,s) - \bar W_x(x)} &\leq \frac{x^2}{20(1+x^2)}\,,
\label{eq:bootstrap:2*} \\
\abs{W_{xx}(x,s)} &\leq \frac{12 \abs{x}}{(1+x^2)^{\sfrac 12}}  \,,
\label{eq:bootstrap:2@} \\
\abs{W_{xxx}(0,s) - 6}  &\leq 1
 \label{eq:bootstrap:2}\,,
\end{align} 
where $\bar W$ is the exact self-similar solution of the Burgers equation given by~\eqref{eq:barW:def} (see \cite{CaSmWa96}). A comment is in order concerning \eqref{eq:bootstrap:2*}: this inequality and properties of the function $\bar W_x$ imply that
 \begin{align}
\norm{W_x(\cdot,s)}_{L^\infty} \leq 1 \qquad \mbox{for all} \qquad s \geq  -\log \eps.
 \label{eq:bootstrap:2**}
 \end{align}
Moreover, we note that \eqref{eq:bootstrap:2@} implies 
\begin{align}
\norm{W_{xx}(\cdot,s)}_{L^\infty} \leq  12  \qquad \mbox{for all} \qquad s \geq  -\log \eps.
 \label{eq:bootstrap:2@@}
\end{align}
  
For the functions $Z$ and $A$ our bootstrap assumptions are
\begin{align}
\norm{\partial_x^n Z}_{L^{\infty}}  + \norm{\partial_x^n A}_{L^{\infty}}   \leq M e^{- (\frac{1}{2}+\delta)s }\,,
\label{eq:bootstrap:3}
\end{align}
for $1\leq n \leq 4$, where $\delta = \delta(\alpha) >0$ is  defined as
\begin{align}
\delta = \frac{\min\{\alpha,1\}}{2(1+\alpha)} > 0 .
\label{eq:delta:def}
\end{align}
Note, that by definition,  we have $\delta \leq \frac 14$. Moreover, $\delta$ is independent of $\eps$ or $M$, and depends only on $\alpha$. We use essentially that $\gamma>1$ to ensure that $\delta>0$. 
 
\begin{remark}[\bf Estimating the blowup time and the blowup location]
\label{rem:blowup:details}
The blowup time $T_*$ is defined uniquely by the condition $\tau(T_*) = T_*$ which in view of \eqref{eq:modulation:initial:time} is equivalent to
\[
\int_{-\eps}^{T_*} (1-\dot \tau(t)) dt = \eps \, .
\]
We note that in view of the $\dot \tau$ estimate in \eqref{eq:acceleration:bound}, we have that $\abs{T_*} \leq 2 \eps^{\sfrac 54}$. We also note here that the bootstrap assumption  \eqref{eq:acceleration:bound} and the definition of $T_*$ ensures that  $\tau(t) > t$ for all $t\in [-\eps, T_*)$. Indeed, when $t = - \eps$ we have $\tau(-\eps) = 0 > -\eps$, and the function $t\mapsto \int_{-\eps}^t (1-\dot \tau) dt' - \eps = t -\tau(t)$ is strictly increasing.
The blowup location is determined by $\theta_* = \xi(T_*)$, which by \eqref{eq:modulation:initial:time} is the same as 
\[
\theta_* = \int_{-\eps}^{T_*} \dot \xi(t) dt \, .
\]
In view of \eqref{eq:acceleration:bound} we deduce that $\abs{\theta_*} \leq 6 M \eps$, so that the blowup location is $\OO(\eps)$ close to the origin. 

\end{remark}

\subsection{Closure of bootstrap} 

Throughout the proof we shall use the notation $\les$ to denote an inequality which holds up to a sufficiently large multiplicative constant $C>0$, which may only depend on $\alpha$ (hence on $\gamma$), but not on $s$, $M$, or $\epsilon$.  

\subsubsection{The $Z$  estimates}

First we consider the equation obeyed by $Z_x$, given by \eqref{eq:multi:Z} with $n=1$. Recalling \eqref{eq:gZ:def}, and appealing to the bootstrap assumptions \eqref{eq:acceleration:bound}, \eqref{eq:bootstrap:2*} (in fact, we use its consequence, the bound \eqref{eq:bootstrap:2**}), and \eqref{eq:bootstrap:3}, we see that the damping term in the $Z_x$ evolution may be bounded from below as
\begin{align}
\frac{3}{2} + \frac{e^{\frac{s}{2}} Z_x }{1-\dot \tau}  + \partial_x g_Z&=\frac{3}{2} + \frac{e^{\frac{s}{2}} Z_x }{1-\dot \tau} +
\frac{(1-\alpha)W_x}{(1-\dot\tau)(1+\alpha)} \notag\\
&\geq \frac{3}{2}  -  (1+2\eps^{\frac 14})\left(M \eps^{\delta} +\frac{\abs{1-\alpha}}{1+\alpha}\right) 
\geq \frac{1}{2}+ \delta 
\label{eq:Zx:damping}
\end{align}
for all $s\geq -\log\eps$, where we have used the parameter $\delta = \delta(\alpha)$ defined in \eqref{eq:delta:def} above. In deriving \eqref{eq:Zx:damping},  we have used that 
$$
(1+2\eps^{\frac 14})\left(M  \eps^\delta +\abs{\frac{1-\alpha}{1+\alpha}}\right) 
\leq (1+2\eps^{\frac 14})\left(M  \eps^\delta + 1- 2 \delta \right)
\leq 1-\delta   
$$
which is true as long as $\eps$ is taken to be sufficiently small, depending only on $\alpha$ (through $\delta$), and on $M$.

On the other hand, the forcing term in the $Z_x$ equation, $F_Z^{(1)} = \p_x F_Z$ may be estimated using 
\eqref{eq:L:infinity}, \eqref{eq:speed:bound}, \eqref{eq:bootstrap:1}, and \eqref{eq:bootstrap:3} as 
\begin{align}
\norm{F_Z^{(1)}}_{L^\infty}  
&\les \frac{e^{-s}}{1-\dot \tau} \left(   \norm{A_x}_{L^\infty}\left(  \norm{(e^{-\frac s2} W + \kappa)}_{L^\infty} +   \norm{Z}_{L^\infty}\right)+ \norm{A}_{L^\infty} \left(  e^{-\frac s2} \norm{W_x}_{L^\infty} +  \norm{Z_x}_{L^\infty} \right)   \right) \notag\\
&\les M e^{-s} \left( M e^{-(\frac 12 + \delta)s}  +     e^{-\frac s2}    \right) \notag   \\
&\les  M^2e^{-\frac {3}{2} s} \, .
\label{eq:Zx:forcing} 
\end{align}
With \eqref{eq:Zx:damping} and \eqref{eq:Zx:forcing}, from \eqref{eq:multi:Z} with $n=1$ and a standard maximum principle argument (cf.~Lemma~\ref{lem:transport}, estimate \eqref{eq:lem:transport:2}, with $\lambda_D = \frac 12 + \delta$, $\lambda_F=\frac 32$, and $s_0 = -\log \eps$), we obtain that
\begin{align}
\norm{Z_x(s)}_{L^\infty} 
&\les \norm{Z_x(-\log \eps)}_{L^\infty} e^{-(\frac 12+\delta)(s+\log \eps)} +  M^2  e^{(1-\delta)\log \eps} e^{-(\frac 12+\delta)s} 
\notag\\
&\les \left( \eps^{1-\delta}+M^2\eps^{1-\delta}\right) e^{-(\frac 12+\delta)s} \notag
\end{align}
where we used \eqref{eq:z0:a0:Cn} to deduce $ \norm{Z_x(-\log \eps)}_{L^\infty}  = \eps^{\frac 32} \norm{\partial_\theta z_0}_{L^\infty} \leq \eps^{\frac 32}$. Then, taking  $\eps$ sufficiently small in terms of $M$, and using $\delta \leq 1/4$ we obtain
\begin{align}
\norm{Z_x(s)}_{L^\infty} \leq \eps^{\frac 14} e^{-(\frac 12+\delta) s}   \leq \frac{M}{2} e^{-(\frac 12+\delta) s}\,,  \label{eq:Zx:bnd}
\end{align}
closing the bootstrap \eqref{eq:bootstrap:3} for $Z_x$.  

Similarly to the estimate for $\partial_x Z$, we note that for $2\leq n\leq 4$, the damping term in \eqref{eq:multi:Z} may be bounded from below as
\begin{align}
\frac{3n}{2} + \frac{n+1}{1-\dot \tau} e^{\frac s2} \partial_x Z+ n \partial_x g_Z
&\geq \frac{3n}{2} - n (1+2\eps^{\frac 14}) \norm{W_x}_{L^\infty} - (n+1) (1+2\eps^{\frac 14})  e^{\frac s2} \norm{\partial_x Z}_{L^\infty}
\notag\\
&\geq \frac{3n}{2} - n (1+2\eps^{\frac 14}) - 5 (1+2\eps^{\frac 14}) M \eps^\delta
\geq \frac{3}{4}\,,
\label{eq:Zx:multi:damping}
\end{align}
for all $s\geq -\log \eps$, 
by appealing to our bootstrap assumptions and by assuming $\eps$ is sufficiently small in terms of $M$.  On the other hand, using our bootstrap assumptions, and the strong bound established earlier in \eqref{eq:Zx:bnd}, one may show that the forcing term on the right side of \eqref{eq:multi:Z} may be estimated as
\begin{align}
\norm{F_Z^{(n)}}_{L^\infty}
&\les \norm{\partial^n_x F_Z }_{L^{\infty}}
 + \norm{ \partial^{n}_x g_Z}_{L^{\infty}}\norm{ \partial_x Z}_{L^{\infty}} + {\bf 1}_{n\geq 3} \sum_{k=2}^{n-1} \left(  e^{\frac s2} \norm{\partial^{k}_x Z}_{L^{\infty}} + \norm{\partial^{k}_x g_Z}_{L^{\infty}} \right) \norm{\partial^{n-k+1}_x  Z }_{L^{\infty}\notag}\\
&\les
 M^2 e^{-s} + M \eps^{\frac 14} e^{-(\frac 12 +\delta)} + {\bf 1}_{\{ n\geq 3\}} \sum_{k=2}^{n-1}  M \norm{\p_x^{n-k+1} Z}_{L^\infty} \notag\\
&\les M\left(\eps^{\frac 14} e^{-(\frac 12+ \delta)s} +  {\bf 1}_{\{ n\geq 3\}} \sum_{k=2}^{n-1}   \norm{\p_x^{n-k+1} Z}_{L^\infty}\right) \,,
\label{eq:Zx:multi:forcing}
\end{align}
where we have assumed $\eps$ to be sufficiently small, dependent on $M$ in order to bound the first term on the second line in terms of the second term. We also remark that since $\p_x^n Z(\cdot,-\log \eps) = \eps^{\frac{3n}{2}} \partial_\theta^n z_0(\cdot)$,  by \eqref{eq:z0:a0:Cn} we have 
\[\norm{\p_x^n Z(\cdot, -\log \eps)}_{L^\infty} \leq \eps^3\,,\]
 for all $n\geq 2$. 

Let us first treat the case $n=2$, when the second term on the right side of \eqref{eq:Zx:multi:forcing} is absent. 
Therefore, in view of \eqref{eq:Zx:multi:damping}--\eqref{eq:Zx:multi:forcing}, and applying Lemma~\ref{lem:transport} to the evolution equation for $\p_x^n Z$ given by \eqref{eq:multi:Z} (with $\lambda_D = \frac 34 $, ${\mathcal F}_0 = M\eps^{\frac{1}{4}}$, and $\lambda_F = \frac{1}{2}+\delta$), we arrive using \eqref{eq:lem:transport:1} at 
\begin{align}
\norm{\partial_x^2 Z(s)}_{L^\infty}
&\les \norm{\p_x^2 Z(\cdot, -\log \eps)}_{L^\infty} e^{-\frac 34 (s+\log\eps)} + M \eps^{\frac{1}{4}} e^{-(\frac 12 + \delta)s}
\notag\\
&\les \eps^{\frac{9}{4}} e^{-\frac 34s} +  M \eps^{\frac{1}{4}} e^{-(\frac 12 + \delta)s} \les  M \eps^{\frac{1}{4}} e^{-(\frac 12 + \delta)s}
 \label{vs1}
\end{align}
for all $s\geq - \log \eps$.  

 With \eqref{vs1}
 in hand, we return to treat the case $n=3$. Then the second term on the right side of \eqref{eq:Zx:multi:forcing} is estimated by a constant multiple of $M^2 \eps^{\frac{1}{4}} e^{-(\frac{1}{2}+\delta)s} $. Therefore, the total estimate on the force 
 for $\p_x^3 Z$ is given by  $\norm{F_Z^{(3)}}_{L^\infty} \les M^2 \eps^{\frac{1}{4}} e^{-(\frac{1}{2}+\delta)s} $. The only modification,  as compared to the case $n=2$,  is that 
$M$ becomes $M^2$. Therefore, an argument similar to the one yielding \eqref{vs1}
gives the estimate 
\begin{align}
\norm{\partial_x^3 Z(s)}_{L^\infty} \les M^2 \eps^{\frac{1}{4}} e^{-(\frac{1}{2}+\delta)s} \, .
\label{eq:Zx:multi:bound:sharp:3}
\end{align}
Using \eqref{vs1} and \eqref{eq:Zx:multi:bound:sharp:3}, we next return to the forcing estimate \eqref{eq:Zx:multi:forcing} for $n=4$. Similar arguments yield 
$\norm{F_Z^{(4)}}_{L^\infty} \les M^3 \eps^{\frac{1}{4}} e^{-(\frac{1}{2}+\delta)s}$, by taking $\eps$ to be sufficiently small, in terms of $M$. Yet another application of Lemma~\ref{lem:transport}, similarly to \eqref{vs1} implies that 
\begin{align}
\norm{\partial_x^4 Z(s)}_{L^\infty} \les M^3 \eps^{\frac{1}{4}}e^{-(\frac{1}{2}+\delta)s} \, .
\label{eq:Zx:multi:bound:sharp:4}
\end{align}
In conclusion,  assuming that $\eps$ is taken to be sufficiently small, dependent on $M$, then 
the bounds \eqref{vs1}, \eqref{eq:Zx:multi:bound:sharp:3}, and \eqref{eq:Zx:multi:bound:sharp:4} close the bootstrap assumptions for $\p_x^n Z$ (with $2\leq n \leq 4$) stated in \eqref{eq:bootstrap:3}.

\subsubsection{The $A$  estimates}
Next we turn to the $\p_x^n A$ estimates  for $1\leq n \leq 4$. These bounds are established very similarly to the $Z$ estimates proven earlier.  The damping term in \eqref{eq:multi:A} is estimated using \eqref{eq:bootstrap:2**} and \eqref{eq:bootstrap:3} as
\begin{align}
 \frac{3n}{2} + n \p_x g_A 
 &=  \frac{3n}{2} + \frac{n (W_x + e^{\frac s2} Z_x)}{(1+\alpha)(1-\dot \tau)}   \geq \frac{3n}{2} - \frac{n (1+2\eps^{\frac 14})(1+M \eps^{\delta})}{1+\alpha}  \geq  \frac{n}{2} + \delta  \, ,
 \label{eq:A:multi:damping}
\end{align}
upon taking $\eps$ small enough in terms of $\delta$ (as defined in \eqref{eq:delta:def} above) and in terms of $\alpha>0$ and $M$.  The forcing term on the right side of \eqref{eq:multi:A} may be bounded from above using our bootstrap assumptions as
\begin{align}
 \norm{F_A^{(n)}}_{L^\infty} 
 &\les  M^2 e^{-s} + M  {\bf 1}_{\{ n\geq 2\}} \sum_{k=2}^{n} \norm{\p_x^{n-k+1} A}_{L^\infty} 
 \,.
 \label{eq:A:multi:forcing} 
 \end{align}
Moreover, note that by \eqref{eq:z0:a0:Cn} we have 
\[\norm{\p_x^n A(\cdot, -\log \eps)}_{L^\infty} \leq \eps^{\frac 32}\,,\]
 for all $n\geq 1$. At this stage one may employ a 
similar scheme to the one employed in the $Z$ estimates. First, we treat the case $n=1$ since in that case the second forcing term on the right side of \eqref{eq:A:multi:forcing} is absent. With \eqref{eq:A:multi:damping} in mind we apply Lemma~\ref{lem:transport}, and deduce (similarly to \eqref{vs1}) that 
\begin{align}
\norm{\p_x A(s)}_{L^\infty}&\leq \eps^{\frac14} e^{-(\frac 12+\delta)s}  \, ,\label{eq:Ax:sharp:bnd} 
\end{align}
where again we absorbed $M^2$ and the implicit constants by assuming $\eps$ to be sufficiently small.
Using the bound \eqref{eq:Ax:sharp:bnd} we may return the case $n=2$, and use that the extra forcing term present on the right side of \eqref{eq:A:multi:forcing} is bounded a constant multiple of $M \norm{\p_x A}_{L^\infty} \les M \eps^{\frac 14} e^{-(\frac 12 + \delta)s}$, upon taking $\eps$ sufficiently small. This argument may be then iterated essentially because in the sum on the right side of \eqref{eq:A:multi:forcing}  we always have $n-k+1 \leq n-1$, so that only norms of $A$ that are already known to be small arise. Using Lemma~\ref{lem:transport} one may then show iteratively that 
\begin{align}
\norm{\p_x^n A(s)}_{L^\infty} \les M^{n-1} \eps^{\frac{1}{4}} e^{-(\frac 12 + \delta)s} \label{eq:Ax:sharp:bnd:n} 
\end{align}
for all $2\leq n \leq 4$. Taking $\eps$ sufficiently small, dependent on $M$, then \eqref{eq:Ax:sharp:bnd} and \eqref{eq:Ax:sharp:bnd:n} close the bootstrap assumptions on $\p_x^n A$ stated in \eqref{eq:bootstrap:3}.

\subsubsection{Bounds on the modulation variables $\tau$, $\kappa$, and $\xi$}

From \eqref{eq:tau:dot}, using the bounds \eqref{eq:L:infinity}, \eqref{eq:speed:bound}, \eqref{eq:Zx:bnd}, and \eqref{eq:Ax:sharp:bnd},  we obtain
\begin{align*}
\abs{\dot \tau }
&\les  e^{\frac s2} \norm{Z_x}_{L^{\infty}}+ e^{-\frac s2}  \norm{A}_{L^\infty} \left( \norm{Z_x}_{L^\infty} + e^{-\frac s2}\right)+ e^{-\frac s2} \norm{A_x}_{L^\infty} \left( \norm{Z}_{L^\infty} + \kappa_0\right) \notag \\
&\les  \eps^{\frac 14} e^{-\delta s} + M e^{- s}   \left( \eps^{\frac 14} e^{-\delta s} + 1 \right)+ \eps^{\frac 14} e^{-(1+\delta)s}   \left( M + \kappa_0\right) 
\end{align*}
The implicit constant is universal. 
Hence for   $s\geq -\log\eps$,  upon taking $\eps$ small to be sufficiently small solely in terms of $M$ and $\delta$, we obtain from the above that 
\begin{equation}
\abs{\dot \tau } \leq C \eps^{\frac 14} e^{-\delta s} \leq C \eps^{\frac 14 + \delta}   \leq \tfrac{1}{2} \eps^{\frac14}\,.
\label{eq:dot:tau:decay}
\end{equation}
Integrating in $t$ for $t\leq T_*$, and using that $\tau(-\eps) = 0$, we obtain
 \begin{equation*}
\abs{  \tau }\leq \tfrac{1}{2}\eps^{\frac54} \,,
\end{equation*}
proving the $\tau$ bounds in \eqref{eq:speed:bound}--\eqref{eq:acceleration:bound}. 

Aa consequence of \eqref{eq:dot:kappa}, \eqref{eq:L:infinity} and the bootstrap assumptions, by inspection we obtain
\begin{align*}
\abs{\dot \kappa} \leq 2 e^{\frac s2} \left( \abs{F_W^0(s)} + \abs{F_W^{0,(2)}(s)} \right)  \leq M^3 \,
\end{align*}
assuming that $M$ is taken to be sufficiently large (in terms of just universal constants). Integrating in $t$  from $-\eps$ to $T_*$,
 and assuming that $\eps$ is sufficiently small (in terms of $M$ and $\kappa_0$), yields
\begin{equation*}
\abs{\kappa (t)}\leq \tfrac{3}{2} \kappa_0\,.
\end{equation*}
This establishes the $\kappa$ bounds in \eqref{eq:speed:bound}--\eqref{eq:acceleration:bound}.

Similarly, from\eqref{eq:dot:xi}, \eqref{eq:L:infinity} and the bootstrap assumptions, by inspection we obtain
\begin{align*}
\abs{\dot \xi}\leq \abs{\kappa} + \abs{Z^0(s)} + e^{-\frac s2} \abs{F_W^{0,(2)}} \leq \tfrac{3}{2} (\kappa_0+M) \leq \tfrac{5}{2}M
\end{align*}
upon taking $\eps$ to be sufficiently small, in terms of $M$, and recalling cf.~Remark~\ref{rem:Linfinity} that $2 \kappa_0 \leq M$.  Integrating in $t$ from $-\eps$ to $T_*$, which obeys $\abs{T_*} \leq 2 \eps^{\sfrac 54}$,  and using that $\xi(-\eps) = 0$, we arrive at
\begin{equation*}
\abs{\xi (t)}\leq 5 M \eps \, ,
\end{equation*}
which proves the $\xi$ estimates in \eqref{eq:speed:bound}--\eqref{eq:acceleration:bound}.

\subsubsection{Estimates for $W$}

 \noindent
{\bf  The third derivative at $x=0$.}
Our first goal is to establish \eqref{eq:bootstrap:2}. The evolution of $\partial_x^3 W^0(s)$ is obtained by restricting \eqref{eq:multi:W} with $n=3$ to $x=0$, using the constraints \eqref{eq:constraints}, and the definition of $\dot \xi$ in \eqref{eq:dot:xi}. We obtain (noting that $\p_x^3 F_W$ also contains the term $\p_x^3 W$):
\begin{align}
&\left(\partial_s+4\left(1 - \frac{1}{1-\dot \tau}\right) + 3 \frac{e^{\frac s2}}{1-\dot\tau} \frac{1-\alpha}{1+\alpha}Z_x^0(s)  -  \frac{e^{-s}(3+2\alpha)}{(1+\alpha)(1-\dot \tau)}  A^0(s) \right)  W_{xxx}^0(s)
\notag\\
&\qquad 
= \frac{F_W^{0,(2)}(s)}{W_{xxx}^0(s)} W_{xxxx}^0(s) +  \frac{(1-\alpha)e^{\frac s2} Z_{xxx}^0(s)}{(1+\alpha)(1-\dot\tau)}-   \frac{e^{-\frac s2}(1-2\alpha) }{(1+\alpha)(1-\dot \tau)}  (A Z)_{xxx}^0(s) 
\notag\\
&\qquad\qquad
+ \frac{e^{-\frac s2}(3+2\alpha)}{(1+\alpha)(1-\dot \tau)}     \left(\kappa A_{xxx}^0(s) - 3 e^{- \frac s2} A_{xx}^0(s)   \right) \, .
\label{eq:Wxxx:0:evo}
\end{align}
We bound the terms of the above evolution  using \eqref{eq:L:infinity}, \eqref{eq:acceleration:bound}, \eqref{eq:bootstrap:1}, \eqref{eq:bootstrap:2}, \eqref{eq:bootstrap:2**}, and \eqref{eq:bootstrap:3}. After a calculation, we obtain that the right side of \eqref{eq:Wxxx:0:evo} is bounded by
\begin{align*}
&\les M \left( M^2 e^{-(\frac 12+\delta)s} + M e^{-(1+\delta)s} \right) +  M e^{-\delta s} +  M^2 e^{-(1+\delta)s} + e^{-\frac s2} \left( \kappa_0 M e^{-(\frac 12+ \delta)s} + M e^{-(1+\delta)s} \right) 
\notag\\
&\les  M e^{-\delta s} 
\end{align*}
where we have assumed $\eps$ to be sufficiently small such that the second term dominates all other terms. On the other hand, the damping term on the left side of \eqref{eq:Wxxx:0:evo} may be estimated in absolute value, upon appealing to the first inequality in \eqref{eq:dot:tau:decay}, by
\begin{align*}
\les \eps^{\frac 14} e^{-\delta s}+ M e^{-\delta s} + M e^{-s} \les M e^{-\delta s} 
\end{align*}
for $s\geq -\log \eps$. Therefore, by also appealing to the bootstrap assumption \eqref{eq:bootstrap:2}, we have proven that 
\begin{align*}
\abs{\partial_s W_{xxx}^0(s)} \les M e^{-\delta s} (\abs{W_{xxx}^0(s)} +1) \les M e^{-\delta s} \, .
\end{align*}
Recalling that $W_{xxx}^0(0) = 6$, and using the fundamental theorem of calculus in time, we obtain
\begin{align}
\abs{W_{xxx}^{0}(s) - 6} \les  M \int_{-\log \eps}^s e^{-\delta s'} ds' \les \frac{M}{\delta} \eps^{\delta}\leq   \eps^{\frac \delta 2}
 \label{eq:w:xxx:0:est}
\end{align}
upon taking $\eps$ to be sufficiently small, in terms of $M$ and $\delta$. Since $\eps<1$, we close the bootstrap \eqref{eq:bootstrap:2}.

 \vspace{.1 in}
\noindent
{\bf  The first derivative.}
We prove \eqref{eq:bootstrap:2*} in two steps, first for $|x|\leq \ell$ for some $\ell>0$ to be determined below (cf.~\eqref{eq:ell:def}), and then for $|x|\geq \ell$. 
Using a Taylor expansion around $x=0$ together with the constraints~\eqref{eq:constraints}, we obtain  
\begin{align*}
W_x(x,s)+1 - 3x^2 
&= x^2 \left(\frac{1}{2} W_{xxx}^0(s) - 3\right) + \frac{x^3}{6} W_{xxxx}(x',s) 
\end{align*}
for some $x'$ with $|x'| < |x|$. Using  \eqref{eq:w:xxx:0:est} 
 and \eqref{eq:bootstrap:1}  we arrive at 
\begin{align*}
\abs{W_x(x,s)+1 - 3x^2 } \leq  x^2 \left( \eps^{\frac{\delta}{2}}  + \frac{M |x|}{6}  \right) \leq  x^2 \left(  \eps^{\frac{\delta}{2}}   + \frac{M \ell}{6} \right) 
\end{align*}
for all $|x|\leq \ell$. Then, recalling \eqref{eq:bar:Wx:zero}, we see that the above estimate implies
\begin{align*}
\abs{W_x(x,s)- \bar W_x} \leq    x^2 \left(  \eps^{\frac{\delta}{2}}  + \frac{M \ell}{6} + 15 \ell^2 \right) \leq \frac{x^2}{40(1+x^2)}
\end{align*}
for all $\abs{x}\leq \ell$, as soon as we choose 
\begin{align}
 \ell 
 \leq \frac{1}{40M} \, ,  \label{eq:ell:def}
\end{align}
$M$ sufficiently large, and $\eps$ sufficiently small in terms of $M$ and $\delta$. Thus, we improve upon the bootstrap assumption \eqref{eq:bootstrap:2*} for $\abs{x}\leq \ell$, as desired.

It remains to establish  \eqref{eq:bootstrap:2*} for $\abs{x}\geq \ell$. For this purpose it is convenient to define
\[
\tilde W = W - \bar W\,,
\]
so that from \eqref{eq:Wx:eq} and the differentiated form of \eqref{eq:barW:dx}, $\tilde W_x$ is the solution of 
\begin{align}
&\left( \partial_s + 1 + \frac{\tilde W_x +  2 \bar W_x }{1-\dot \tau} + \frac{(1-\alpha) e^{\frac s2} Z_x}{(1+\alpha)(1-\dot\tau)} \right) \tilde W_x + \left( g_W + \frac{3x}{2} + \frac{W}{1-\dot \tau} \right) \tilde W_{xx} 
\notag\\
&\qquad =  \partial_x F_W   - \left( g_W + \frac{\tilde W +\dot\tau \bar W}{1-\dot \tau}   \right) \bar W_{xx} - \left( \frac{\dot \tau  \bar W_x}{1-\dot \tau}  +  \frac{(1-\alpha) e^{\frac s2} Z_x}{(1+\alpha) (1-\dot \tau)} \right)  \bar W_x \, .
\label{eq:tilde:Wx:evo}
\end{align}
Note that by \eqref{eq:constraints} and \eqref{eq:bar:Wx:zero}, we have $\tilde W(0,s) = \tilde W_x(0,s) = \tilde W_{xx}(0,s)= 0$.
Next, we define 
\begin{align*}
 V(x,s) = \frac{\tilde W_x (1+x^2)}{x^2}
\end{align*}
so that establishing \eqref{eq:bootstrap:2*} is equivalent  to proving  that $\abs{V} \leq \frac{1}{20}$ for all $s> -\log \eps$  and all $\abs{x}\geq \ell$.
 It is important here that we are avoiding  $x=0$ (since we concerned with $\abs{x} > \ell$), in view of the division by $x^2$. 
It follows from \eqref{eq:tilde:Wx:evo} and a short computation that
\begin{align}
&\partial_s V + \left(1 + \frac{\tilde W_x +  2 \bar W_x }{1-\dot \tau} + \frac{(1-\alpha) e^{\frac s2} Z_x}{(1+\alpha)(1-\dot\tau)} +\frac{2}{x(1+x^2)} \left( g_W + \frac{3x}{2} + \frac{W}{1-\dot \tau} \right)  \right) V \notag\\
&\qquad   + \left( g_W + \frac{3x}{2} + \frac{W}{1-\dot \tau} \right) V_{x} 
\notag\\
&  = \frac{(1+x^2)\partial_x F_W}{x^2}   - \left( g_W + \frac{\dot\tau \bar W}{1-\dot \tau}   \right) \frac{(1+x^2)\bar W_{xx}}{x^2} - \left( \frac{\dot \tau  \bar W_x}{1-\dot \tau}  +  \frac{(1-\alpha) e^{\frac s2} Z_x}{(1+\alpha) (1-\dot \tau)} \right) \frac{(1+x^2) \bar W_x }{x^2} \notag\\
&\qquad - \frac{1}{1-\dot \tau} \frac{(1+x^2)\bar W_{xx}}{x^2} \int_0^x V(x') \frac{(x')^2}{1+(x')^2} dx' \, .
\label{eq:V:evo}
\end{align}
The evolution equation for $V$ takes the form of a damped and non-locally forced transport equation, of the general form given in \eqref{eq:nonlocal:transport} below.  Our goal is to apply Lemma~\ref{lem:max:princ} to \eqref{eq:V:evo}. 

The main observation which allows us to bound the solution $V$ of \eqref{eq:V:evo} is that the {\em explicit formula} for $\bar W$ in \eqref{eq:barW:def} implies the lower bound
\begin{align}
 1 + 2 \bar W_x    +\frac{2}{x(1+x^2)} \left( \frac{3x}{2} + \bar W \right)   \geq \frac{6 x^2}{1+ 8 x^2}
 \label{eq:V:main:damping}
\end{align}
for all $x\in \RR$. Since we are analyzing $\abs{x}\geq \ell$, the above estimate yields a strictly positive  damping term in the $V$ equation. In order to see this, let us estimate 
the remaining terms in the damping factor for $V$ on the left side of \eqref{eq:V:evo}. We claim that for all $\abs{x}\geq \ell$,  we have that
\begin{align}
\abs{\frac{\tilde W_x +  2 \dot \tau \bar W_x }{1-\dot \tau} + \frac{(1-\alpha) e^{\frac s2} Z_x}{(1+\alpha)(1-\dot\tau)} +\frac{2}{x(1+x^2)} \left( g_W + \tilde W+ \frac{\dot \tau W}{1-\dot \tau} \right)  } \leq  \frac{5 x^2}{4(1+8 x^2)} +  \eps^{\frac{\delta}{2}} 
\, .
\label{eq:V:noise:damping}
\end{align}
Indeed, using the $\dot \tau$ estimate~\eqref{eq:acceleration:bound}, the fact that $\abs{\bar W_x} \leq 1$, and the bootstrap assumptions, we deduce that  
\begin{align}
\abs{\frac{\tilde W_x +  2 \dot \tau \bar W_x }{1-\dot \tau} + \frac{(1-\alpha) e^{\frac s2} Z_x}{(1+\alpha)(1-\dot\tau)} + \frac{2\tilde W}{x(1+x^2)}} 
&\leq  (1+ 2\eps^{\frac 14}) \left( \frac{3 x^2}{20(1+x^2)} + 2 \eps^{\frac 14} + M \eps^\delta \right)  \notag\\
&\leq   \frac{5 x^2}{4(1+8 x^2)} +  2M\eps^{\delta} \label{e:Fultz}
\end{align}
since $\eps$ is sufficiently small.
Here we have used that $\tilde W(0,s) =0$, and thus that 
$$
\abs{ \frac{2 \tilde W(x,s)}{x(1+x^2)}} \leq \frac{2}{\abs{x}(1+x^2)} \int_0^{\abs{x}} \abs{\tilde W_x(x',s)} dx' \leq \frac{1}{10 \abs{x}(1+x^2)} \int_0^{\abs{x}} \frac{(x')^2}{1+(x')^2} dx' \leq \frac{x^2}{10(1+x^2)} \, .
$$
Similarly, using the constraint \eqref{eq:constraints} and the bound \eqref{eq:bootstrap:2**},  we may directly estimate
\begin{align}
\frac{2 \abs{\dot \tau W(x,s)}}{x (1+x^2) (1-\dot \tau)} 
\leq   \frac{4 \eps^{\frac 14}}{x}\int_0^x \abs{W_x(x',s)} dx' \leq 4   \eps^{\frac 14}\,.\label{e:NY_Nicks}
\end{align}
Recall the identities   \eqref{eq:gW:est} and \eqref{e:ForcingW}. Note that by  \eqref{eq:bootstrap:3} we have $\abs{Z(x,s) - Z^0(s)} \leq M |x| e^{-(\frac 12 + \delta)s}$. Then, by appealing to  \eqref{eq:L:infinity}, \eqref{eq:bootstrap:3} and the constraints \eqref{eq:constraints},   we may deduce that
\begin{align}
\abs{g_W(x,s)} 
&\leq   \frac{(1-\alpha) e^{\frac s2}}{(1+\alpha)(1-\dot \tau)} \abs{Z(x,s) - Z^0(s)} +  \frac{\abs{F_W^{0,(2)}}}{W_{xxx}^0(s)} 
\notag\\
&\les \abs{x} M e^{-\delta s} + \norm{\partial_{xx} F^0_W}_{L^{\infty}}+e^{\frac s2}\norm{Z^0_{xx}}_{L^{\infty}} \notag\\
&\les \abs{x} M e^{-\delta s} + M^2 e^{-s}+Me^{-\delta s} \notag\\
&\les  \ell^{-1}M \abs{x} e^{-\delta s}
\label{eq:gW:Lip}
\end{align}
for any $\ell \leq \abs{x}$. Choosing $\eps$ sufficiently small in terms of $\delta$ and $M$, and  combining  \eqref{e:Fultz}--\eqref{eq:gW:Lip} yields the proof of \eqref{eq:V:noise:damping}. In turn, combining \eqref{eq:V:main:damping} and \eqref{eq:V:noise:damping} we obtain that the total damping term in \eqref{eq:V:evo} may be bounded from below as
\begin{align}
 1 + \frac{\tilde W_x +  2 \bar W_x }{1-\dot \tau} + \frac{(1-\alpha) e^{\frac s2} Z_x}{(1+\alpha)(1-\dot\tau)} +\frac{2}{x(1+x^2)} \left( g_W + \frac{3x}{2} + \frac{W}{1-\dot \tau} \right)  \geq \frac{9 x^2}{2( 1+ 8 x^2 )}
 \label{eq:awesome}
\end{align}
pointwise for all $\abs{x}\geq \ell$.  Here we have implicitly used that $\eps^{\frac{\delta}{2}} \leq \frac{\ell^2}{16} \leq \frac{x^2}{4(1+ 8 x^2)}$ for $\abs{x} \geq \ell$
since by \eqref{eq:ell:def}, $\ell$ is small enough when $M$ is large. From \eqref{eq:awesome} and the fact that the function $\frac{9 x^2}{2(1+ 8 x^2)}$ is monotone increasing in $\abs{x}$, we obtain that the damping term in \eqref{eq:V:evo} is bounded from below by $\lambda_D := \frac{9 \ell^2}{2(1+ 8 \ell^2)}$ for all $\abs{x}\geq \ell$, as required by \eqref{eq:max:princ:damping}.

Our next observation concerns the last term on the right side of \eqref{eq:V:evo}, which is nonlocal in $V$.  We may write this term as the integral of $V(x',s)$ against the kernel
\begin{align*}
{\mathcal K}(x,x',s) = - \frac{1}{1-\dot \tau} \frac{(1+x^2)\bar W_{xx}(x)}{x^2}   {\bf 1}_{[0,x]}(x')  \frac{(x')^2}{1+(x')^2} 
\, .
\end{align*}
Since we know $\bar W_{xx}$ exactly, we may show that pointwise in $x$ and $s$ we have the bound 
\begin{align}
\int_{\RR} \abs{\mathcal K(x,x',s)} dx' \leq  \frac{\abs{\bar W_{xx}}(1+x^2)}{(1-\dot \tau)x^2} \int_0^{\abs{x}}   \frac{(x')^2}{1+(x')^2} dx' \leq  \frac{3 (1+ 2\eps^{\sfrac 14} ) x^2}{1+8 x^2} \, .
\label{eq:also:awesome}
\end{align}
In view of \eqref{eq:awesome}, \eqref{eq:also:awesome}, and the bound $3 (1+  2\eps^{\sfrac 14}) \leq \sfrac 92 \cdot \sfrac 34$, which holds since $\eps$ is sufficiently small, the kernel ${\mathcal K}$ obeys the assumption \eqref{eq:max:princ:kernel} of Lemma~\ref{lem:max:princ}.

Next, we estimate the forcing term in \eqref{eq:V:evo} for $\abs{x} \geq \ell$ in order to identify the constant $\mathcal F_0$ from Lemma~\ref{lem:max:princ}. Indeed, using the explicit properties of $\bar W$, the first line on the right side of \eqref{eq:V:evo} is bounded from above by
\begin{align*}
&\norm{\frac{1+x^2}{x^2}\p_x F_W}_{L^\infty(\abs{x}\geq \ell)} 
 + \left( \norm{\frac{g_W}{x}}_{L^\infty(\abs{x}\geq \ell)} + 2 \abs{\dot \tau} \norm{\frac{ \bar W}{x}}_{L^\infty(\abs{x}\geq \ell)}   \right) \norm{\frac{(1+x^2)\bar W_{xx}}{x}}_{L^\infty(\abs{x}\geq \ell)}
\notag\\
&\quad + 2 \left(\abs{\dot \tau}  \norm{\bar W_x}_{L^\infty}  +   e^{\frac s2} \norm{Z_x}_{L^\infty} \right) \norm{\frac{(1+x^2) \bar W_x}{x^2}}_{\abs{x}\geq \ell} \\
& \les \ell^{-2} \norm{ \p_x F_W}_{L^\infty(\abs{x}\geq \ell)} +\norm{\frac{g_W}{x}}_{L^\infty(\abs{x}\geq \ell)} + \abs{\dot\tau}+ \ell^{-2}(\abs{\dot\tau}+ e^{\frac s2} \norm{Z_x}_{L^\infty})\\
&  \les \ell^{-2} M^2 e^{-s} + \ell^{-1}M e^{-\delta s}+\eps^{\frac14}+ \ell^{-2}(\eps^{\frac14}+ Me^{-\delta s} )\\
&  \les \ell^{-2}M\eps^{\delta} 
\end{align*}
where we have employed \eqref{eq:L:infinity}, \eqref{eq:acceleration:bound} \eqref{eq:bootstrap:3}, \eqref{eq:gW:Lip}, and assumed $\eps$ to be sufficiently small, dependent on $M$.  Therefore, taking $\eps$ smaller if need be, the estimate on the force required by \eqref{eq:max:princ:force}  in Lemma~\ref{lem:max:princ} holds, with ${\mathcal F}_0 =  \eps^{\frac{\delta}{2}} $. 
 
Lastly, we verify the bounds \eqref{eq:max:princ:ass}.   We already know that for $\abs{x}\leq \ell$, and for $s\geq -\log \eps$, we have the inequality $\abs{V(x,s)}\leq \sfrac{1}{40}$. Moreover, in view of the assumption \eqref{eq:big:assume:ab}, at the initial time $s=  - \log \eps$ we have that $x \eps^{\frac 32} = \theta$ and thus
\begin{align*}
\abs{V(x,-\log\eps)} 
= \frac{1+x^2}{x^2} \abs{W_x(x,-\log \eps) -\bar W_x(x)} 
=  \frac{\eps^3+\theta^2}{\theta^2} \abs{\eps (\partial_\theta w_0)(\theta)  -\bar W_x\left(\frac{\theta}{\eps^{\frac 32}}\right)} \leq \frac{1}{40} \, .
\end{align*} 
Thus, \eqref{eq:max:princ:ass} holds with $m= \sfrac{1}{20}$.

In order to apply Lemma~\ref{lem:max:princ} we finally need to verify the condition \eqref{eq:max:princ:cond}. In view of our determined values for $\lambda_D, {\mathcal F}_0$ and $m$, we have
\[
m \lambda_D = \frac{1}{20} \frac{9 \ell^2}{2(1+ 8 \ell^2)}   \geq    4 \eps^{\frac{\delta}{2}}  = 4 {\mathcal F_0} 
\]
once $\eps$ is chosen to be sufficiently small, in terms of $\ell \leq 1$ (and thus of $M$). Also, note that by Remark~\ref{rem:support} we have that $W_x$ is compactly supported, while from \eqref{eq:bar:Wx:infinity} we have that $\bar W_x$ decays as $\abs{x}\to \infty$. Therefore, we have $\abs{V(x,\cdot)} \to 0$ as $\abs{x}\to \infty$. We may thus apply Lemma~\ref{lem:max:princ} and conclude from \eqref{eq:max:princ} that 
\[
\norm{V(\cdot,s)}_{L^\infty(\RR)} \leq \frac{3}{80} 
\]
which proves the bootstrap assumption \eqref{eq:bootstrap:2*}.

\vspace{.1 in}
\noindent
{\bf  The second derivative.}
We note that from \eqref{eq:bootstrap:1}, the constraint $W_{xx}(0,s) = 0$ in \eqref{eq:constraints}, and the bound \eqref{eq:w:xxx:0:est}, we obtain that 
\begin{align}
\abs{W_{xx}(x,s)} 
&\leq \abs{x} W_{xxx}(0,s) + \frac{x^2}{2} \norm{\p_x^4 W}_{L^\infty} 
\notag\\
& \leq (6 + \eps^{\frac{\delta}{2}} )  \abs{x} + \frac{M}{2} x^2 \leq \frac{7\abs{x}}{(1+x^2)^{\frac 12}}\, , \qquad \mbox{for all} \qquad \abs{x} \leq \frac{1}{M}\, ,
\label{e:initial_W_xx_est_loc}
\end{align}
and all $s\geq -\log \eps$. Here we have assumed that $M$ is sufficiently large.  This shows that \eqref{eq:bootstrap:2@} automatically holds for  $\abs{x}\leq \sfrac{1}{M}$, with an even better constant.

Next, we observe that \eqref{eq:w0:bnd} implies $\norm{\p_x^2W(\cdot,-\log\eps)}_{L^\infty} \leq 1$ and $\norm{\p_x^4 W(\cdot,-\log \eps)}\leq 1$. Using \eqref{eq:w0:power:series}, and a Taylor expansion,  together with the uniform bound \eqref{eq:w0:bnd}, we conclude that 
\begin{align}\label{e:initial_W_xx_est}
 \norm{W_{xx}(x,-\log\eps)} \leq \min \left\{6 \abs{x} + \frac{x^2}{2}, 1\right\}   \leq \frac{7 \abs{x}}{(1+x^2)^{1/2}} 
\end{align}
for all $x\in \RR$. 

Similarly to the above subsection, in order to prove \eqref{eq:bootstrap:2@} for $\abs{x}$ large, we introduce a new variable which is a weighted version of $W_{xx}$; we define
\begin{align}
\tilde V(x,s) = \frac{(1+x^2)^{\frac 12} W_{xx}(x,s)}{x} \, .
\label{eq:tilde:V:def}
\end{align}
From \eqref{eq:Wxx:eq}, we see that $\tilde V(x,s)$ is a solution of
\begin{align}
&\partial_s \tilde V +\left( \frac52 +\frac{3W_x}{1-\dot \tau} +\frac{2(1-\alpha) e^{\frac s2}Z_x }{(1+\alpha)(1-\dot \tau)} + \frac{1}{x(1+x^2)} \left(g_W+\frac{3x}{2}+\frac{W}{1-\dot\tau}\right) - \frac{(3+2\alpha)e^{-s} A}{(1+\alpha)(1-\dot \tau)} \right)\tilde V \notag\\
&\qquad \qquad + \left(g_W+\frac{3x}{2} +\frac{W}{1-\dot\tau}\right) \p_x \tilde V \notag\\
&= - \frac{e^{-\frac s2}(1+x^2)^{\frac 12}}{x (1+\alpha)(1-\dot\tau)} \left( (1-2\alpha) (AZ)_{xx} - (3+2\alpha) \left( A_{xx}( e^{-\frac s2} W+\kappa) + 2 e^{-\frac s2} A_x W_x \right) \right) \notag\\
&\qquad \qquad - \frac{(1-\alpha) e^{\frac s2} (1+x^2)^{\frac 12} Z_{xx} W_x}{x(1+\alpha)(1-\dot \tau)} 
\, .
\label{eq:tilde:V:evo}
\end{align}
Here we have used that $\p_x^2 F_W$ contains a term with a factor of  $W_{xx}$; the corresponding weighted term has been grouped with the other damping terms on the left of \eqref{eq:tilde:V:evo}. The idea is simple: the damping term in \eqref{eq:tilde:V:evo} is larger than the forcing term, for all $\abs{x}\geq \sfrac{1}{M}$, once $\eps$ is chosen sufficiently small. 

In order to make this precise, we first estimate the damping term from below. The main observation is that for the exact self-similar profile $\bar W$, we have 
\begin{align}
\frac52 + 3 \bar W_x   + \frac{1}{x(1+x^2)} \left( \frac{3x}{2}+\bar W \right) \geq \frac{x^2}{ 1+ x^2}
 \label{eq:Shaq}
\end{align}
for all $x\in \RR$. This bound is similar to \eqref{eq:V:main:damping}, and it holds because we know $\bar W$ precisely. Using the estimates \eqref{eq:L:infinity}, \eqref{eq:acceleration:bound}, \eqref{eq:bootstrap:2**}, \eqref{eq:bootstrap:3}, \eqref{eq:gW:Lip} and \eqref{eq:Shaq}, we thus may bound from below
\begin{align}
&\frac52 +\frac{3W_x}{1-\dot \tau} +\frac{2(1-\alpha) e^{\frac s2}Z_x }{(1+\alpha)(1-\dot \tau)} + \frac{1}{1+x^2} \left(\frac{g_W}{x}+\frac{3}{2}+\frac{W}{x(1-\dot\tau)}\right) - \frac{(3+2\alpha)e^{-s} A}{(1+\alpha)(1-\dot \tau)} 
\notag\\
&\geq  \frac{x^2}{1+ x^2} - (3+6\eps^{\frac14})\abs{\tilde W_x}  - C M e^{-\delta s} - \frac{1}{1+x^2} \left( C\ell^{-1}M  e^{-\delta s}+ \abs{\frac{\tilde W}{x}}\right) -  C M e^{-s} 
\label{eq:tilde:V:damping*}
\end{align}
where $C>0$ only depends on $\alpha$.
Using \eqref{eq:bootstrap:2*} and the fundamental theorem of calculus, we have
\begin{align*}
(3+6\eps^{\frac14})\abs{\tilde W_x}+\frac{1}{1+x^2} \abs{\frac{\tilde W}{x}} &\leq \frac{3x^2}{20(1+x^2)}+ \eps^{\frac14}+\frac{1}{\abs{x}(1+x^2)}\abs{ \int_0^x \frac{y^2}{20(1+y^2)}dy}\\
&\leq \frac{x^2}{5(1+x^2)}+ \eps^{\frac14}
\end{align*}
where we used that $\abs{\frac{1+x^2}{x^3} \int_0^x \frac{y^2 dy}{1+y^2} } \leq 1$ for all $x\in \RR$.  Taking $\eps$ sufficiently small, depending on $M, \alpha, \delta$, we may thus bound the right hand side of \eqref{eq:tilde:V:damping*}, and thus the total damping terms on the left side of \eqref{eq:tilde:V:evo}, from below by
\begin{align}
\geq  \frac{4 x^2}{5(1+ x^2)}  - \eps^{\frac{\delta}{2}}  \geq \frac{1}{2 M^{2}}\, \qquad \mbox{for all} \qquad \abs{x}\geq \frac 1M \, ,
\label{eq:tilde:V:damping}
\end{align}
upon taking $\eps$ to be small enough in terms of $\delta$ and $ M$.

Similarly, for $\abs{x}\geq  \sfrac 1M$ the forcing term on the right hand side of \eqref{eq:tilde:V:evo} may be bounded by
\begin{align}
&\les \ e^{-\frac s2}\frac{(1+x^2)^{\frac 12}}{\abs{x}} \left( \abs{(AZ)_{xx}}+ \left(\abs{ A_{xx}}( e^{-\frac s2} \abs{W}+\kappa) +  e^{-\frac s2} \abs{A_x W_x} \right) \right)
+ \frac{e^{\frac s2} (1+x^2)^{\frac 12} \abs{Z_{xx} W_x}}{\abs{x}} \notag\\
&\les  (M^2e^{-s}+Me^{-\delta s} )\frac{(1+x^2)^{\frac 12}}{\abs{x}} 
\les M^2 e^{-\delta s} 
 \label{eq:tilde:V:forcing}
\end{align}
where we assumed $\eps$ to be sufficiently small dependent on $M$.

To close the bootstrap, we wish to apply Lemma~\ref{lem:max:princ} (with ${\mathcal K} \equiv  0$) to the evolution equation \eqref{eq:tilde:V:evo}. Using \eqref{e:initial_W_xx_est_loc} and \eqref{e:initial_W_xx_est}, the condition \eqref{eq:max:princ:ass} is satisfied with $m=14$ and $\Omega =\{x~:\abs{x}\leq \sfrac{1}{M}\}$. From \eqref{eq:tilde:V:forcing} we verify that \eqref{eq:max:princ:force} holds with ${\mathcal F}_0 = \eps^{\frac{\delta}{2}}$, after talking $\eps$ to be small enough to absorb the implicit constant and the $M^2$ factor. Owing to \eqref{eq:tilde:V:damping}, the condition \eqref{eq:max:princ:cond} then amounts to checking
\[14 \frac{1}{2M^{2}} \geq  e^{-\frac{\delta}{2}}  \]
which is easily seen to be satisfied by taking $\eps$ to be sufficiently small, dependent on $M$. Applying Lemma~\ref{lem:max:princ} we obtain
\[\norm{\tilde V}_{L^{\infty}}\leq \frac{21}{2} < 12\]
which closes the bootstrap \eqref{eq:bootstrap:2@} upon recalling the definition of $\tilde V$ in \eqref{eq:tilde:V:def}.

\vspace{.1 in}
\noindent
{\bf The fourth derivative.}
The evolution of the fourth derivative of $W$ is governed by \eqref{eq:multi:W} with $n=4$. The damping term in this equation may be bounded from below as
\begin{align}
 \frac{11}{2} + \frac{5}{1-\dot \tau} \partial_x W+ 4 \partial_x g_W
 &\geq \frac{11}{2} - 5 (1+2 \eps^{\frac 14})\left(1 + 4 e^{\frac s2} \norm{\p_x^4 Z}_{L^\infty} \right) \notag\\
 &\geq \frac{11}{2} - 5 (1+2 \eps^{\frac 14})\left(1 + 4 M \eps^{\delta s}  \right) \geq \frac{1}{4}
 \label{eq:Wxxxx:damping}
\end{align}
where we have used that $\abs{W_x} \leq 1$,  $\eps$ is sufficiently small, and \eqref{eq:bootstrap:3} holds.  On the other hand, the forcing term $F_W^{(4)}$ may be estimated using \eqref{eq:acceleration:bound}, \eqref{eq:bootstrap:-1}, \eqref{eq:bootstrap:1},  and \eqref{eq:bootstrap:2**}--\eqref{eq:bootstrap:3} as
\begin{align}
\norm{F_W^{(4)}}_{L^\infty} 
&\les   e^{-\frac s2} \norm{AZ}_{\dot C^4} +  e^{-\frac s2} \norm{A(e^{-\frac s2} W + \kappa)}_{\dot C^4} +  \norm{W}_{\dot C^3} \norm{W}_{\dot C^2} + \sum_{k=1}^3 \norm{W}_{\dot C^k} e^{\frac s2}\norm{Z }_{\dot C^{5-k}}   \notag\\
&\les M^2  e^{-s} +  M^{\frac 34}  + M^{\frac 74} e^{-\delta s} \les  M^{\frac 34} 
 \label{eq:Wxxxx:forcing}
\end{align}
assuming $\eps$ to be sufficiently small, dependent on $M$. Appealing to Lemma~\ref{lem:transport}, estimate \eqref{eq:lem:transport:1}, with $\lambda_F = 0$, $\lambda_D = 1/4$, and ${\mathcal F}_0 = C  M^{\frac 34}$ where $C$ is the (universal) implicit constant in \eqref{eq:Wxxxx:forcing}, we arrive at
\begin{align}
\norm{\p_x^4 W(\cdot,s)}_{L^\infty} 
\leq  \norm{\p_x^4 W(\cdot,-\log \eps)}_{L^\infty} e^{-\frac 14(s+\log \eps)} + 4 C M^{\frac 34}
\leq  1 + 4 C M^{\frac 34} \leq \frac{M}{2}  
\label{eq:wxxxx_bound00}
\end{align}
for any $s\geq -\log\eps$.
In the second inequality above,  we have used the initial datum assumption \eqref{eq:w0:bnd} on the fourth derivative of the initial datum, while in the third inequality we have used that $M$ is sufficiently large, in terms of the universal constant $C$. This estimate proves the fourth derivative bound in \eqref{eq:bootstrap:1}.

\vspace{.1 in}
\noindent
{\bf  Global bound for the third derivative.}
Using  the mean value theorem and the bound \eqref{eq:wxxxx_bound00}  we  have
\begin{align*}
\abs{W_{xxx}(x,s) - W_{xxx}(0,s)} \leq \abs{x} M  
\end{align*}
which may be combined with \eqref{eq:w:xxx:0:est} to arrive at
\begin{align}
 \abs{W_{xxx}(x,s)} \leq 6 + \eps^{\frac{\delta}{2}}+ \abs{x} M  \leq \frac{M^{\frac34}}{2} \qquad \mbox{for} \qquad \abs{x} \leq \frac{1}{ 4M^{\frac14}} 
\label{eq:w:xxx:local}
\end{align}
and all $s\geq -\log\eps$, assuming $M$ is sufficiently large. At the initial time, in view of \eqref{eq:w0:bnd}, the estimate 
\begin{align}
 \abs{W_{xxx}(x,- \log \eps)}  \leq 7 \leq \frac{M^{\frac34}}{2}
\label{eq:w:xxx:datum}
\end{align}
holds for all $x\in \RR$. We next claim that 
\begin{align}
\abs{W_{xxx}(x,s)}  \leq \frac{3M^{\frac34}}{4} 
\label{eq:w:xxx:global}
\end{align}
holds for all $s> -\log \eps$ and all $\abs{x}\geq 1/(4 M^{\sfrac 14})$. The estimate \eqref{eq:w:xxx:global} would then immediately imply the bootstrap assumption for the third derivative in \eqref{eq:bootstrap:1}.
The proof of \eqref{eq:w:xxx:global} is  based on Lemma \ref{lem:max:princ} (with ${\mathcal K}\equiv 0$), and a lower bound on the damping term for the $\p_x^3 W$ evolution.

We recall from~\eqref{eq:multi:W} with $n=3$, and carefully computing the forcing term $F_W^{(3)}$,  that 
\begin{align}
&\left(\partial_s + 4 \left(1+ \p_x W\right) + \frac{\dot \tau W_x}{1-\dot\tau}  + \frac{4 (1-\alpha)e^{\frac s2}Z_x  }{(1+\alpha)(1-\dot \tau)} - \frac{(3+2\alpha) e^{-s} A}{(1+\alpha)(1-\dot \tau)} \right) \p_x^3 W + \left( g_W + \frac{3x}{2} + \frac{W}{1-\dot\tau}\right) \p_x^4 W
\notag\\
& = \frac{e^{-\frac s2}}{(1+\alpha)(1-\dot \tau)} \left( (3+2\alpha) \left( \p_x^3 A (e^{-\frac s2} W + \kappa) + 3e^{-\frac s2} \p_x^2 A \p_x W + 3 e^{-\frac s2} \p_x  A \p_x^2 W \right) - (1-2\alpha) \p_x^3(AZ) \right)\notag\\
&\qquad - \frac{(1-\alpha)e^{\frac s2}}{(1+\alpha) (1-\dot \tau)} \left( \p_x^3 Z \p_x W +3 \p_x^2 Z \p_x^2 W\right) - \frac{3}{1-\dot\tau} \left( \p_x^2 W \right)^2 
\label{eq:W:xxx:evo}
\end{align}
holds.  In order to prove \eqref{eq:w:xxx:global}, we first estimate the right side of \eqref{eq:W:xxx:evo}. From \eqref{eq:L:infinity}, \eqref{eq:acceleration:bound}, \eqref{eq:bootstrap:1}, \eqref{eq:bootstrap:2**}, and \eqref{eq:bootstrap:3}, we may directly estimate the error term on the right side of \eqref{eq:W:xxx:evo} in absolute value by  
\begin{align}
 \les  M^2 e^{-\delta s} + 1  
\label{eq:W:xxx:forcing}
\end{align}
assuming M is sufficiently large, and $\eps$ is sufficiently small, dependent on $M$ and $\delta$. Returning to the damping term in the evolution for $\p_x^3 W$,  for any $x$  and any $s\geq -\log \eps$, we have that
\begin{align*}
&4 \left(1+ \p_x W\right) + \frac{\dot \tau W_x}{1-\dot\tau}  + \frac{4 (1-\alpha)e^{\frac s2}Z_x  }{(1+\alpha)(1-\dot \tau)} - \frac{(3+2\alpha) e^{-s} A}{(1+\alpha)(1-\dot \tau)}
\notag\\
&\quad \geq  4 \left(1+ \bar W_x  - \frac{x^2}{20(1+x^2)} \right) - 2 \eps^{\frac 14} - 8 M \eps^{\delta} - 6 M \eps
 \geq    \frac{ x^2}{1+x^2}   -  \eps^{\frac{\delta}{2}}  \, . 
\end{align*}
Above we have appealed to \eqref{eq:L:infinity}, \eqref{eq:acceleration:bound}, \eqref{eq:bootstrap:2*}, \eqref{eq:bootstrap:2**}, and \eqref{eq:bootstrap:3}, and have taken $\eps$ to be sufficiently small, in terms of $M$ and $\delta$. In the second inequality above we have also appealed to the pointwise estimate $1+ \bar W_x - \frac{3 x^2}{4(1+x^2)} \geq 0$ holds for all $x\in \RR$. Now, for  $\abs{x}>1/(4M^{\frac14})$ we obtain that 
\begin{align}
4 \left(1+ \p_x W\right) + \frac{\dot \tau W_x}{1-\dot\tau}  + \frac{4 (1-\alpha)e^{\frac s2}Z_x  }{(1+\alpha)(1-\dot \tau)} - \frac{(3+2\alpha) e^{-s} A}{(1+\alpha)(1-\dot \tau)}\geq \frac{1}{1+16  M^{\frac 12} } -  \eps^{\frac{\delta}{2}} \geq   \frac{1}{32 M^{\frac 12}} 
\,,
\label{eq:W:xxx:dissipation}
\end{align}
upon taking $\eps$ sufficiently small, solely in terms of $M$ and $\delta$.

We return to \eqref{eq:W:xxx:evo} with the information \eqref{eq:W:xxx:forcing} and \eqref{eq:W:xxx:dissipation} in hand. In view of \eqref{eq:w:xxx:datum}, we know that at the initial time and on the compact set $\Omega = \{ x\colon \abs{x} \leq 1/(4M^{\frac 14})\}$, the inequality \eqref{eq:w:xxx:global} holds, with the constant $3/4$ being replaced by the constant $1/2$, i.e.\ condition \eqref{eq:max:princ:ass} is satisfied with $m=M^{\frac34}$. 
Moreover, from \eqref{eq:W:xxx:forcing} and \eqref{eq:W:xxx:dissipation}, condition \eqref{eq:max:princ:cond} amounts to checking
\[
M^{\frac 34} \frac{1}{32 M^{\frac 12}} \geq 4 (C M^2 \eps^\delta + 1)
\]
where $C$ is the implicit constant in \eqref{eq:W:xxx:forcing}. This condition is true so long as $M$ is sufficiently large and $\eps$ is chosen sufficiently small, dependent on $M$. Hence we may apply Lemma~\ref{lem:max:princ} to deduce that \eqref{eq:w:xxx:global} holds for all $s\geq -\log\eps$.

\subsection{Proof of Theorem~\ref{eq:thm:bnd:1}}
In this section we show that the already established bootstrap bounds \eqref{eq:speed:bound}--\eqref{eq:bootstrap:3}, together with a number of {\em a-posteriori estimates} give the proof of Theorem~\ref{eq:thm:bnd:1}. First, we note that from \eqref{eq:x:s:def}--\eqref{eq:ss:ansatz}, the definition of $T_*$ in Remark~\ref{rem:blowup:details}, and  \eqref{eq:speed:bound}--\eqref{eq:bootstrap:3}, we obtain that the solutions $(w,z,a)$ remain  $C^4$ smooth at all times prior to $T_*$. Second, we remark that \eqref{eq:constraints} implies $\p_\theta w(\xi(t),t) = e^{s} W_x(0,s) = - e^{s}$, while \eqref{eq:bootstrap:2**} yields $\norm{\p_\theta w(\cdot,t)}_{L^\infty} \leq e^{s}$. These bounds prove the claimed blowup behavior of $\p_\theta w$ as $t\to T_*$, upon recalling that $e^{s}$ and $\sfrac{1}{(T_*-t)}$ only differ by a factor $\leq 2$. Third, we notice that the claimed $\eps$ dependent bounds on $T_*$ and $\theta_*$ were established in Remark~\ref{rem:blowup:details}, while Remark~\ref{rem:Linfinity} (see also estimate~\eqref{eq:improved:L:infty} below) give the claimed amplitude bounds for $(w,z,a)$.  

It remains for us to prove  that $\norm{\p_\theta a(\cdot,t)}_{L^\infty}$, $[w(\cdot,t)]_{C^{\sfrac 13}}$, and $\norm{\p_\theta z(\cdot,t)}_{L^\infty}$ remain uniformly bounded on 
$[-\eps,T_*)$, that the claimed upper and lower bounds for the vorticity hold, and that the lower bound for the density also holds. In Proposition~\ref{prop:ballistic} below, we prove the 
desired vorticity, density and $\p_\theta a$ bounds. The uniform-in-time H\"older $C^ {\sfrac{1}{3}} $ bound is more delicate and it does not directly follow from the proven bootstrap estimates.  Rather, to establish this $C^ {\sfrac{1}{3}} $ bound,
we use the second estimate on the right side of  \eqref{eq:W:x:asymptotic} and prove that it can be propagated forward in time, in self-similar variables. This is achieved 
in Section~\ref{rem:Holder:13}. As explained in Remark~\ref{rem:Holder:13} below, these improved bounds on the   blowup profile $W$ as $\abs{x}\to \infty$, imply the desired 
H\"older estimate.  Using this information, we  prove in Section~\ref{sec:z:prime} that the distance between the  Lagrangian flow of the transport velocity in the $z$ equation 
and $\xi(t)$ remains
too  large  as $t\to T_*$ for a blowup to occur;  namely, this distance is  $\OO(T_*-t)$ instead of $\OO((T_*-t)^{\sfrac 32})$, which in turn implies that $\p_\theta z$ remains 
uniformly bounded all the way up  to the blowup  time $T_*$. 

Finally, once these a posteriori estimates for $(w,z,a)$ as well as for $\upomega$ and $P$ are established, the estimates for solutions $(u_r,u_\theta,\rho)$  of the Euler equations
\eqref{eq:Euler:polar} immediately follow from the 
 the definition of the Riemann variables \eqref{eq:riemann} together with our homogeneity assumption  \eqref{scale0} on the solutions. We note that the blowup  segment $\Gamma(T_*)$ is the natural extension of the blowup point $\theta_*$ in the radial direction.

\subsubsection{Density, vorticity, and $\p_\theta a$ bounds}
\begin{proposition} 
 \label{prop:ballistic}
Let $\nu_0, \kappa_0, M, \eps$, and $T_*$ be as in the statement of Theorem~\ref{thm:general}, and assume that $(w_0,z_0,a_0)$ satisfy the bounds \eqref{eq:w0:power:series}--\eqref{eq:w0:C0}. Then, we have that the $L^\infty$ bound \eqref{eq:L:infinity} holds, 
and additionally that the bounds
\begin{align*} 
\frac{\nu_0}{2} \leq P(\theta,t) \leq M, \qquad \frac{1}{M^2} \leq \upomega(\theta,t) \leq M^2, \qquad \abs{\p_\theta a(\theta,t)} \leq 3 M^2 \, ,
\end{align*} 
hold for all $\theta \in \TT$ and for all $t \in [0,T_*)$. 
\end{proposition}
\begin{proof}[Proof of Proposition~\ref{prop:ballistic}]
From \eqref{eq:euler:wza}  we see that 
any $\varphi \in \{w,z,a\}$ satisfies an equation of the type $\partial_t \varphi + \lambda(w,z)  \varphi' = Q(w,z,a)$ where $Q$ is an explicit quadratic polynomial which obeys $\abs{Q(w,z,a)} \leq C_\alpha (\max\{\abs{w},\abs{z},\abs{a}\})^2$ for some constant $C_\alpha$ that only depends on $\alpha$, and $\lambda$ is a speed that is explicitly computable in terms of $w,z$ and $\alpha$. Recall that our initial datum assumptions imply $\sfrac{\kappa_0}{2} \leq w_0 \leq \sfrac{3\kappa_0}{2}$ on $\TT$, and that $\norm{z_0}_{L^\infty} + \norm{a_0}_{L^\infty} \leq 1$. From the maximum principle for forced transport equations, upon recalling that $\abs{T_*} \leq  \eps$, and upon taking $\eps$ to be sufficiently small, we deduce that 
\begin{align}
\label{eq:improved:L:infty}
\frac{\kappa_0}{4} \leq w(\cdot,t) \leq 2\kappa_0 \, \qquad  \norm{z(\cdot,t)}_{L^\infty} \leq 2 \, \qquad \norm{a(\cdot,t)}_{L^\infty} \leq 2
\end{align}
for any $t\in [-\eps,T_*)$. The above estimate shows that \eqref{eq:L:infinity} holds as soon as $M\geq 4 + 2\kappa_0$, as claimed in Remark~\ref{rem:Linfinity}. 

Since $P =  (\tfrac{\alpha }{2} (w-z))^{\sfrac{1}{ \alpha }}$,
from \eqref{eq:improved:L:infty} we deduce that
\begin{equation}\label{eq:P:upperbound}
\sup_{ t\in [- \epsilon , T_*)} \| P( \cdot , t)\|_{ L^ \infty(\T)} \le ( \alpha (\kappa_0+1))^{\sfrac 1 \alpha } \leq M\,,
\end{equation} 
upon taking $M$ to be sufficiently large (in terms of $\alpha$ and $\kappa_0$), 
and moreover that 
\begin{align}
P(\theta,t) \geq  \left(\tfrac{\alpha }{2} \left(\tfrac{\kappa_0}{4} - 2\right) \right)^{\sfrac{1}{ \alpha }} \geq \tfrac{\nu_0}{2} >0 
\label{eq:P:lowerbound}
\end{align}
for all $\theta \in \TT$ and $t\in [-\eps,T_*)$, by appealing to the lower bound \eqref{eq:kappa:0} on $\kappa_0$.
The above two bounds give the desired density estimates.
 
 Next, we consider estimates related to the vorticity. Since $\upomega_0 = 2 b_0 - \p_\theta a_0 = w_0 + z_0 - \p_\theta a_0$,  from \eqref{eq:z0:a0:Cn}, \eqref{eq:kappa:0}, and \eqref{eq:kappa:0} we deduce that
$$
\tfrac{\kappa_0}{4} \leq \tfrac{\kappa_0}{2}-2  \le \upomega_0 \le    \tfrac{3 \kappa_0}{2}  + 2 \leq 2 \kappa_0 \,,  
$$
and since $\varpi_0 = \frac{ \upomega_0}{P_0}$, from \eqref{eq:P:upperbound}--\eqref{eq:P:lowerbound} we obtain
$$
\tfrac{\kappa_0}{4 M}  \le \varpi_0 \le \tfrac{4 \kappa_0}{\nu_0}  \,.
$$
Furthermore, from equation \eqref{eq:varpi}, we have that $\varpi$ obeys a forced transport equation, and upon composing this equation with the flow of $b$, and exponentiating,  the standard Gr\"onwall inequality and the previously established bound \eqref{eq:improved:L:infty} imply that 
$$
\frac{\kappa_0}{8 M} \leq \frac{\kappa_0}{4 M} e^{- \frac{2t}{\alpha}} \le \varpi( \cdot , t) \le \frac{4 \kappa_0}{\nu_0} e^{\frac{2t}{\alpha}} \leq \frac{8 \kappa_0}{\nu_0} \, ,
\qquad \mbox{for all} \qquad t \in [- \epsilon , T_*)\, .
$$
Here we have used that $\eps$ is taken sufficiently small in terms of $\alpha, \kappa_0, M$ and $\nu_0$.
Combining the above bound with \eqref{eq:P:upperbound}--\eqref{eq:P:lowerbound} and the identity $\upomega = \varpi P$, we deduce that 
\begin{align*}
\frac{1}{M^2} \leq \frac{\kappa_0 \nu_0}{16 M}   \le \upomega( \cdot , t) \le  \frac{8 \kappa_0 M}{\nu_0} \leq M^2\, ,
\qquad \mbox{for all} \qquad t \in [- \epsilon , T_*)\, ,
\end{align*}
which is the desired vorticity upper and lower bound. Here we have assumed that $M$ may be taken to be sufficiently large, in terms of $\kappa_0$ and $\nu_0$. Finally, since $\p_\theta a = w + z - \omega$ we deduce from the above bound and \eqref{eq:improved:L:infty} that
\begin{align}
\norm{\p_\theta a( \cdot , t)}_{L^\infty}  \le  2 \kappa_0 + 2  \frac{8 \kappa_0 M}{\nu_0} \leq 3 M^2\, ,\qquad   \text{ for all } \qquad t \in [- \epsilon , T_*) \,,
\label{eq:a:prime:bnd}
\end{align}
upon taking $M$ sufficiently large, in terms of $\kappa_0$ and $\nu_0$. 
\end{proof}

\subsubsection{Sharp bounds for $W$ and $W_x$ as $\abs{x}\to \infty$ and H\"older $1/3$ estimates}
\label{sec:Holder:13}
From the bootstrap assumption \eqref{eq:bootstrap:2*} we know that as $\abs{x}\to \infty$ we have $\abs{\tilde W_x} = \abs{W_x - \bar W_x} \leq \sfrac{1}{20}$. Note,  however,
 that \eqref{eq:bar:Wx:infinity} implies the asymptotic behavior $\abs{x^{\sfrac 23} \bar W_x}  \to \sfrac 13$ as $\abs{x}\to\infty$. Our goal is to show that in fact $W_x$ itself also has a 
 $\abs{x}^{-\sfrac 23}$ decay rate as $\abs{x}\to \infty$, uniformly in $s$. To prove this, we show that this asymptotic behavior is valid for $\tilde W_x$, which we recall satisfies the evolution equation \eqref{eq:tilde:Wx:evo}. In order to normalize the behavior at infinity, we consider the function $\mathcal V$ defined as 
\begin{equation}\label{eq:def:mathcalV}
\mathcal V(x,s) =   (x^{\sfrac 23} + 8) \tilde W_x(x,s) \, ,
\end{equation} 
where the translation of $x^{\sfrac 23}$ by $8$ will be explained below  in the course of the argument. Our objective is to show that 
\begin{align}
\norm{\mathcal V(\cdot, s)}_{L^\infty} \leq 1
\label{eq:cal:V:bootstrap}
\end{align}
for all $s\geq -\log \eps$.  We remark that at the initial time $s = -\log \eps$, we have 
 \begin{align*}
 \abs{\mathcal V(x,-\log \eps)} = \left(\frac{\theta^{\sfrac 23}}{\eps} + 8 \right) \abs{\eps (\p_\theta w_0)(\theta) - (\bar W_x)\left(\frac{\theta}{\eps^{\sfrac 32}}\right)} \leq \frac 12 \, , \qquad\mbox{for all} \qquad x\in \RR\, ,
\end{align*}
in view of assumption \eqref{eq:big:assume:ab} on the initial datum.  Additionally, note that by  \eqref{eq:bootstrap:2*},  we have
\begin{align}
\abs{\mathcal V(x,s)} \leq \frac{x^2(x^{2/3}+8)}{20 (1+x^2)} \leq \frac 12\, ,\qquad \mbox{for all} \qquad \abs{x} \leq 2 \,  ,
\label{eq:V:cal:origin}
\end{align}
and thus \eqref{eq:cal:V:bootstrap} is automatically satisfied with a better constant ($\sfrac 12$ instead of $1$) for $\abs{x}\leq 2$. 

Similarly to \eqref{eq:V:evo}, a simple computation shows that $\mathcal V$ satisfies
\begin{align}
&\p_s \mathcal V + \left(1 + \tilde W_x +  2 \bar W_x    -    \frac{2 x^{\sfrac 23}}{3(x^{\sfrac 23}+8)}  \left(\frac{3}{2} + \frac{\bar W + \tilde W}{x}  \right)  \right) \mathcal V
+ \left( g_W + \frac{3x}{2} + \frac{W}{1-\dot \tau} \right) {\mathcal V}_{x}  \notag\\
&  = (x^{\sfrac 23} + 8) \partial_x F_W    - \left( g_W + \frac{\dot\tau \bar W}{1-\dot \tau}   \right) (x^{\sfrac 23} + 8) \bar W_{xx} - \left( \frac{\dot \tau  \bar W_x}{1-\dot \tau}  +  \frac{(1-\alpha) e^{\frac s2} Z_x}{(1+\alpha) (1-\dot \tau)} \right)  (x^{\sfrac 23} + 8) \bar W_x   \notag\\
&\qquad - \left(  \frac{\dot \tau(\tilde W_x +  2 \bar W_x ) + \tfrac{1-\alpha}{1+\alpha} e^{\frac s2} Z_x}{1-\dot\tau} - \frac{2 x^{\sfrac 23}}{3(x^{\sfrac 23}+8)}  \left( \frac{\dot \tau W}{(1-\dot \tau)x} + \frac{g_W}{x}\right)  \right) \mathcal V \notag\\
&\qquad - \frac{1}{1-\dot \tau}  (x^{\sfrac 23} + 8) \bar W_{xx}  \int_0^x {\mathcal V}(x') \frac{1}{ (x')^{\sfrac 23} + 8} dx' \, .
\label{eq:cal:V:evo}
\end{align}
It is convenient to rewrite \eqref{eq:cal:V:evo} schematically as
\begin{align}
\p_s \mathcal V + {\mathcal D}(x,s) \mathcal V + \mathcal U (x,s) \mathcal V_x = \mathcal F_1 (x,s) + \mathcal F_2 (x,s) + \int_0^\infty \mathcal V(x',s) {\mathcal K}(x,x',s) dx' 
\label{eq:cal:V:evo:1}
\end{align}
where ${\mathcal D}$ and $\mathcal U$ are determined by the first  line on the left side of \eqref{eq:cal:V:evo}, the forcing term $\mathcal F_1$ is given by the first line on the right side of \eqref{eq:cal:V:evo},  the forcing term $\mathcal F_2$ is given by the second line on the right side of \eqref{eq:cal:V:evo}, and ${\mathcal K}$ is defined by the last line of the 
$\mathcal V$ evolution as ${\mathcal K}(x,x',s) = - \frac{1}{1-\dot \tau}  (x^{\sfrac 23} + 8) \bar W_{xx}(x)  \frac{{\bf 1}_{[0,x]}(x') }{ (x')^{\sfrac 23} + 8}$. 
The argument fundamentally consists of a  comparison between the  damping term ${\mathcal D}$ with the $L^1_{x'}$-norm of the kernel $\mathcal K$, similar in spirit to the one used to prove Lemma~\ref{lem:max:princ}.

Using the fundamental theorem of calculus, the fact that $\tilde W(0,s) = 0$, and the bootstrap assumption  \eqref{eq:cal:V:bootstrap}, we obtain the following lower bound on 
the damping term:
\begin{align*}
{\mathcal D}(x,s) \geq  1    - \frac{1}{x^{2/3} + 8} +  2 \bar W_x    -    \frac{2 x^{\sfrac 23}}{3(x^{\sfrac 23}+8)}  \left(\frac{3}{2} + \frac{\bar W}{x} + \frac{1}{x} \int_0^x \frac{dx'}{(x')^{\sfrac 23} + 8}\right) =: {\mathcal D}_{\operatorname{upper} }(x)
\end{align*}
On the other hand, using our bound for $\dot \tau$ \eqref{eq:dot:tau:decay}, we have that
\begin{align*}
\int_\RR \abs{\mathcal K(x,x',s)} dx' \leq (1+2 \eps^{\sfrac 14}) (x^{\sfrac 23} + 8) \abs{\bar W_{xx}(x) } \int_0^{\abs{x}} \frac{dx'}{(x')^{\sfrac 23} + 8} =: {\mathcal D}_{\operatorname{lower} }(x) \,.
\end{align*}
The choice of the translation constant $8$ in the weight appearing in \eqref{eq:def:mathcalV} was chosen so that 
by letting $\eps$ be sufficiently small, we  ensure that 
\begin{align}
0< {\mathcal D}_{\operatorname{lower} }(x) \le {\mathcal D}_{ \operatorname{upper} }(x), \qquad \mbox{for all} \qquad \abs{x}\geq 2.
\label{eq:V:cal:damping}
\end{align}
While, in fact, ${\mathcal D}_{\operatorname{lower} }(x) \le \tfrac{3}{4} {\mathcal D}_{ \operatorname{upper} }(x)$ for $\abs{x}\ge 2$
as required by \eqref{eq:max:princ:kernel}, the reason we cannot apply Lemma~\ref{lem:max:princ} is that for $\abs{x} \gg 1$ we have ${\mathcal D}_{ \operatorname{upper} }(x) = 5 x^{-\sfrac 23} + \OO(\abs{x}^{-1})$, and so we cannot obtain a uniform in $x$ lower bound on the damping, as required by \eqref{eq:max:princ:damping}. Nonetheless, we will still apply an argument similar to the one used to prove Lemma~\ref{lem:max:princ}.

Next, we estimate the forcing term $\mathcal{F} _1$. The most delicate term is the one due to $\p_x F_W$, which is bounded using \eqref{eq:cal:V:bootstrap} and the support property discussed in Remark~\ref{rem:support}, as
\begin{align*}
\norm{(x^{\sfrac 23}+8) \p_x F_W}_{L^\infty} 
&\les e^{-\sfrac s2} \norm{(x^{\sfrac 23}+8) \p_x (AZ)}_{L^\infty} + e^{-\sfrac s2} \norm{(x^{\sfrac 23}+8) \p_x A}_{L^\infty} \norm{e^{-\sfrac s2}W + \kappa}_{L^\infty} \notag\\
&\qquad + e^{-s}  \norm{A}_{L^\infty}   \left(1+ \norm{(x^{\sfrac 23}+8) \bar W_x}_{L^\infty}  \right)
\notag\\
&\les M^2 e^{-\delta s} + M e^{-s} 
\end{align*}
where the implicit constant depends only on $\alpha$.  The remaining forcing terms are easier to estimate since we already know the decay rates 
$\bar W_x = \OO(\abs{x}^{-\sfrac 23})$ and $\bar W_{xx} = \OO(\abs{x}^{-\sfrac 53})$ as $\abs{x}\to \infty$. Using the available estimate \eqref{eq:dot:tau:decay} for $\dot \tau$, 
the bound  
\eqref{eq:Zx:bnd} for $\p_x Z$, and the third line of \eqref{eq:gW:Lip}  to bound $g_W$, after a computation we deduce that the total forcing term may be estimated as 
\begin{align}
\norm{\mathcal F_1(\cdot,s)}_{L^\infty} \leq C  M^2 e^{-\delta s} + C  M e^{-s}   \leq  e^{-\sfrac{\delta s}{2}}  
\label{eq:V:cal:force}
\end{align}
by choosing $\eps$ to be sufficiently small in terms of $M$ and the constant $C$ which only depends on $\alpha$.   Similarly, we have that
\begin{align}
\norm{\mathcal{F} _2(\cdot,s)}_{L^\infty(\abs{x}\geq 2)} \leq   e^{- \sfrac{\delta s}{2}}  \,,
\label{eq:V:cal:bad}
\end{align}
which follows from the previously established properties of $\dot \tau$, $W_x$, $\bar W_x$, $Z_x$, and $g_W$, after choosing $\eps$ to be sufficiently small in terms of $\alpha, \delta, M$.

In order to conclude the proof of \eqref{eq:cal:V:bootstrap}, we claim that 
\begin{align}
\norm{{\mathcal V}(\cdot,s)}_{L^\infty}\leq \frac{3}{4} 
\label{eq:cal:V:bootstrap:*}
\end{align}
which would show that the bootstrap assumption
\eqref{eq:cal:V:bootstrap} holds with an even better constant ($\sfrac 34$ instead of $1$), thereby closing it. 
If \eqref{eq:cal:V:bootstrap:*} were to fail at some time $s_1>-\log \eps$, by continuity in time there exists a   time $s_0 \in (-\log \eps, s_1)$ such that 
$ \norm{\mathcal V(\cdot,s)}_{L^\infty} \geq \norm{\mathcal V(\cdot,s_0)}_{L^\infty} = \sfrac 58 $ for all $s\in [s_0,s_1]$. Then, for $s\in [s_0,s_1)$ we may evaluate \eqref{eq:cal:V:evo:1} at the global maximum of 
$\abs{\mathcal V}$, which is ensured to be attained at a point $x_*=x_*(s)$ with $\abs{x_*}\geq 2$, since $W_x$ is compactly supported, 
$(x^{\sfrac 23} + 8) \abs{\bar W_x} \to \sfrac 13 < \sfrac 58$ as $\abs{x} \to \infty$, and \eqref{eq:V:cal:origin} holds.  Without loss of generality, let us consider the case when $\mathcal V(x_*(s),s)$ is 
the global maximum for $\mathcal V$ (the case of a global minimum is treated similarly). At this maximum point 
$\mathcal V_x$ vanishes, and using \eqref{eq:V:cal:damping} we obtain
\begin{align*}
{\mathcal D}(x_*(s),s) \mathcal V(x_*(s),s) 
&\geq {\mathcal D}_{\operatorname{upper} }(x_*(s)) \norm{\mathcal V(\cdot,s)}_{L^\infty}   \notag\\
&\ge {\mathcal D}_{\operatorname{lower} }(x_*(s)) \norm{\mathcal V(\cdot,s)}_{L^\infty} 
\geq \abs{\int_\RR  \mathcal K(x_*(s),x',s) {\mathcal V}(x',s) dx'} \,.
 \end{align*}
Therefore, at $x_*(s)$ the second term on the left side of \eqref{eq:cal:V:evo:1} dominates the third term on the right side of \eqref{eq:cal:V:evo:1}.  Next, via a standard Rademacher 
argument (applicable since $\mathcal V$ is smooth), and using the bounds \eqref{eq:V:cal:force}--\eqref{eq:V:cal:bad} we obtain that a.e. in $s$ 
\begin{align*}
\frac{d}{ds} \norm{\mathcal V(\cdot,s)}_{L^\infty} \leq 2 e^{-\sfrac{\delta s}{2}}\,.
\end{align*} 
Using that by assumption $\norm{\mathcal V(\cdot,s_0)}_{L^\infty} = \sfrac 58$, we  integrate the above inequality for $s\geq s_0$ and deduce that 
\begin{align*}
\norm{\mathcal V(\cdot,s)}_{L^\infty} \leq (\sfrac 58 + 1) e^{\frac{4}{\delta} \eps^{\sfrac{\delta}{2}}} - 1 < \sfrac 34
\end{align*}
for all $s> s_0> - \log \eps$, upon taking $\eps$ to be sufficiently small. This provides the desired contradiction  and thus \eqref{eq:cal:V:bootstrap:*} holds, concluding the proof.

\begin{remark}[\bf Uniform H\"older bounds]
\label{rem:Holder:13}
Estimate \eqref{eq:cal:V:bootstrap} and properties of the function $\bar W_x$ imply that 
\begin{align}
\abs{W_x(x,s)} \leq \frac{1}{x^{\sfrac 23} + 8} + \abs{\bar W_x(x)} \leq \frac{2}{x^{\sfrac 23}}
\label{eq:W:x:asymptotic}
\end{align}
for all $x\in \RR$ and $s\geq - \log \eps$. Since $W(0,s) = 0$ for all $s$, integrating the above estimate   in $x$ we arrive at
\begin{align}
\abs{W(x,s)} \leq 6 \abs{x}^{\sfrac 13}
\label{eq:W:asymptotic}
\end{align}
for all $x\in \RR$ and $s\geq - \log \eps$. The bounds \eqref{eq:cal:V:bootstrap}--\eqref{eq:W:asymptotic}   imply that $w \in L^\infty([-\eps,T_*);C^{\sfrac 13}(\TT))$. To see this, consider any two points $\theta\neq \theta' \in \TT$. Accordingly, define the points $x = \frac{\theta - \xi(t)}{(\tau-t)^{\sfrac 32}} \neq x' =  \frac{\theta' - \xi(t)}{(\tau-t)^{\sfrac 32}}$ by   the scaling \eqref{eq:x:s:def}. Due to the description \eqref{eq:ss:ansatz}  of $w$ we have that 
\begin{align}
\frac{\abs{w(\theta,t) - w(\theta',t)}}{\abs{\theta-\theta'}^{\sfrac 13}} 
= 
\frac{\abs{W(x,s) - W(x',s)}}{\abs{x-x'}^{\sfrac 13}}
\, .
\label{eq:Holder:13}
\end{align}
At this stage we remark that when $x'=0$, and $x$ is taken to be arbitrary, the bound \eqref{eq:W:asymptotic} implies that the right side of \eqref{eq:Holder:13} is bounded by $6$ uniformly in $s$. To consider the general case of $x\neq x'$, we combine \eqref{eq:W:x:asymptotic} with \eqref{eq:bootstrap:2**} to deduce that 
$\abs{W_x(x,s)} \les (1+x^2)^{-\sfrac 13}$  where the implicit constant is universal. Then, using the fundamental theorem of calculus we estimate
\begin{align*}
\sup_{x>x'} \frac{\abs{W(x,s) - W(x',s)}}{\abs{x-x'}^{\sfrac 13}} \les \sup_{x>x'}\frac{\int_{x'}^x  (1+y^2)^{-\sfrac 13} dy}{(x-x')^{\sfrac 13}}  \les 1
\end{align*}
where the implicit constant is universal, and is in particular independent of $s$.  This concludes the proof of the uniformy in time H\"older $\sfrac 13$ estimate for $w$. It is not hard to see that $C^\alpha$ H\"older norms of $w$, with $\alpha>\sfrac 13$ blow up as $t\to T_*$ with a rate proportional to $(T_*-t)^{\sfrac{(1-3\alpha)}{2}}$. 
\end{remark}

\subsubsection{Bounds for $\p_\theta z$ as $t\to T_*$}
\label{sec:z:prime}
In view of the relation $\p_\theta z = e^{\sfrac{3s}{2}} \p_x Z$, and the already established bound \eqref{eq:bootstrap:3}, we have that 
$\norm{\p_\theta z(\cdot,t)}_{L^\infty} \leq 2 M (T_* - t)^{-1 + \delta}$, for $t\in [-\eps,T_*)$.
Here we have used that 
\begin{equation}\label{eq:new1}
(1-\eps^{\sfrac 14})(T_*-t) \leq \tau(t) - t \leq (1+\eps^{\sfrac 14}) (T_*-t)\,,
\end{equation} 
which is a consequence of Remark~\ref{rem:blowup:details} and the identity 
$\tau(t) - t = \eps - \int_{-\eps}^t (1-\dot \tau) = \int_{t}^{T_*}(1-\dot \tau)$, and the fact that $\tau(t) - t = e^{-s}$. We may, however,  show that $\p_\theta z$ remains in fact bounded 
as $t \to T_*$.

Upon differentiating \eqref{eq:z:evo} with respect to $\theta$, we obtain
\begin{align}
\left( \partial_t  +\left(z+\tfrac{1-\alpha}{1+\alpha}w\right)\partial_{\theta}\right) (\p_\theta z)
&= - \left(\p_{\theta} z+\tfrac{1-\alpha}{1+\alpha}\p_{\theta} w\right)(\p_{\theta} z)  
- \tfrac{1-2\alpha}{1+\alpha} a (\p_\theta w) - \tfrac{3+2\alpha}{1+\alpha} a (\p_\theta z)
\notag\\
&\qquad -\tfrac{1}{1+\alpha}\p_\theta a \big((1-2\alpha)w+(3+2\alpha)z\big) 
\,.
\label{eq:z:prime}
\end{align}
Note that by \eqref{eq:improved:L:infty} and \eqref{eq:a:prime:bnd}, we know that $a, z, w$, and  $\p_\theta a$ remain uniformly bounded in $L^\infty(\T)$ over $[-\eps,T_*)$, and so 
we may think of these terms as constants in \eqref{eq:z:prime}.  Moreover, since $\norm{\p_\theta z_0}_{L^\infty} \leq 1$, the term $- (\p_\theta z)^2$ on the right side of
\eqref{eq:z:prime} cannot by itself cause a finite time singularity  in time $\OO(\eps)$. The blowup of $\p_\theta z$ could only be caused by the terms involving $\p_\theta w$ on the 
right side of \eqref{eq:z:prime}; specifically the $- \frac{1-\alpha}{1+\alpha} (\p_\theta z) (\p_\theta w)$ term is dominant near a putative singularity of $\p_\theta z$. 
Indeed, $\norm{\p_\theta w}_{L^\infty} = e^s \norm{W_x}_{L^\infty} = e^s \geq (\sfrac 12) (T_*-t)^{-1}$, and so $\int_{-\eps}^{T_*} \norm{\p_\theta w(\cdot, t)}_{L^\infty} = + \infty$, 
which could be sufficient to cause a singularity.

Our main observation is that if we compose \eqref{eq:z:prime} with its natural Lagrangian flow $\zeta_{x_0}(t)$, defined as
\begin{align}
\tfrac{d}{dt} \zeta_{\theta_0}(t) = z(\zeta_{\theta_0}(t),t) +\tfrac{1-\alpha}{1+\alpha} w(\zeta_{\theta_0}(t),t)\, , \qquad \zeta_{\theta_0}(- \eps) = \theta_0 \, ,
\label{eq:zeta:def}
\end{align}
then  the quantity $\int_{-\eps}^{T_*} \p_\theta w(\zeta_{\theta_0}(t),t) dt$ is the relevant one to study for bounding $\norm{\p_\theta z}_{L^\infty}$. We claim that as $t\to T_*$ the 
quantity $\abs{\p_\theta w(\zeta_{\theta_0}(t),t)}$ does not  blow up at a non-integrable rate. Once the claim is proven, standard ODE arguments imply that the solution $\p_\theta z$ of \eqref{eq:z:prime} remains bounded in $L^\infty$ as $t\to T_*$.

For the remainder of this proof we drop the subindex $\theta_0$ of $\zeta_{\theta_0}$ (it is frozen) and we use $e^{-s}$ and $T_* - t$ interchangeably 
as they are comparable up to a factor of $1 \pm \eps^{\sfrac 14}$ by \eqref{eq:new1}.  By the definition of $W$  in \eqref{eq:ss:ansatz} and the previously established bound \eqref{eq:W:x:asymptotic},  
we have that 
\begin{align}
\abs{\p_\theta w (\zeta(t),t)} = e^s \abs{W_x\left( (\zeta(t) - \xi(t)) e^{\frac{3s}{2}}, s\right)}  \lesssim  \frac{1}{T_*-t} \left(1+\frac{\abs{\zeta(t) - \xi(t)}}{(T_*-t)^{\sfrac 32}}\right)^{-\sfrac 23}  \, .
\label{eq:Kawhi}
\end{align}

Consider the case that $ \zeta(T_*) \neq \xi(T_*)$.  Then, by continuity, $|\zeta(t) - \xi(t)| \ge c$ for $t$ sufficiently close to $T_*$.   Therefore, from \eqref{eq:Kawhi}, 
$\abs{\p_\theta w (\zeta(t),t)} $ is bounded.   Otherwise,  $\zeta(t) - \xi(t) \to 0$ as $t\to T_*$.
Our goal is to show that there exists a constant $c_*$ such that for all $t$ sufficiently close to $T_*$ we have $\abs{\zeta(t) - \xi(t)} \geq c_* (T_*-t)$. Once this claim is established, it follows from \eqref{eq:Kawhi} that $\int_{-\eps}^{T_*} \abs{\p_\theta w (\zeta(t),t)} dt < \infty$, as desired.    

It remains to prove the claimed lower bound for $\zeta - \xi$.   Using the definition of $\zeta(t)$ in \eqref{eq:zeta:def} and the 
definition of $\dot \xi$ in \eqref{eq:dot:xi}, we derive that
\begin{align}
\zeta(t) - \xi(t) &= \int_{t}^{T_*} \dot\xi(t') - z(\zeta(t'),t') - \tfrac{1-\alpha}{1+\alpha} w(\zeta(t'),t') dt' \notag\\
&= \tfrac{2\alpha}{1+\alpha}  \int_{t}^{T_*} \kappa(t')  dt' + \int_{t}^{T_*} \tfrac{1-\alpha}{1+\alpha} Z^0(s') - Z\left((\zeta(t')-\xi(t')) e^{\frac{3s'}{2}},s'\right)  dt' \notag\\
&\qquad - \int_{t}^{T_*}  \tfrac{1-\alpha}{1+\alpha} e^{-\frac{s'}{2}} W\left((\zeta(t')-\xi(t')) e^{\frac{3s'}{2}},s'\right) + (1-\dot \tau) e^{-\frac{s'}{2}} \frac{F_W^{0,(2)}(s')}{W_{xxx}^0(s')} dt' 
\notag \\
&= I_1(t) + I_2(t) - I_3(t)
\label{eq:Kobe}
\end{align}
where $e^{-s'} = \tau(t') - t'$. From \eqref{eq:acceleration:bound} we deduce that $I_1(t) \geq \frac{\alpha \kappa_0}{1+\alpha}(T_*-t)$, upon taking $\eps$ sufficiently small in terms of $M$ and $\kappa_0$. It is essential here that $\alpha > 0$, i.e. $\gamma>1$. Using \eqref{eq:improved:L:infty} we immediately obtain that $\abs{I_2(t)} \leq \frac{4}{1+\alpha} (T_* - t)$. Lastly, using our   bootstrap assumptions and the estimate \eqref{eq:W:asymptotic}, after a tedious computation we deduce that the integrand of $I_3$ may be bounded in absolute value as $\les e^{- \delta s'} \les (T_*-t')^{\delta}$, and therefore $\abs{I_3(t)} \les  (T_*-t)^{1+\delta} \les \eps^\delta (T_* -t)$. We collect the above estimates and insert them in \eqref{eq:Kobe}, to deduce that 
$\abs{\zeta(t) - \xi(t)} \geq \frac{\alpha \kappa_0}{1+\alpha}(T_*-t) - \frac{5}{1+\alpha} (T_* - t) \geq \frac{1}{1+\alpha}(T_*-t)$, by taking $\kappa_0$ sufficiently large, in terms of $\alpha$. 
As discussed above, this lower bound concludes our proof for the boundedness of $\p_\theta z$. 

\section{Concluding remarks}\label{sec:conclusion}
By considering homogeneous solutions to the isentropic 2D compressible Euler equations, and using a transformation to self-similar coordinates with dynamic modulation
variables, we have 
proven that for  an open set of smooth initial data with $\OO(1)$ amplitude, $\OO(1)$ vorticity, and with minimum initial slope  $ -{\sfrac{1}{\eps}} $,  there exist smooth solutions of 
the Euler equations  which form an asymptotically self-similar shock within  $\OO( \eps)$ time.       Our method 
is based on perturbing purely azimuthal waves which inherently possess nontrivial vorticity, and thus,  our constructed 
solutions have $\OO(1)$ vorticity at the shock, as well as a lower-bound on the density, so that no vacuum regions can form during the formation of the shock singularity.   

A key feature of our method is that the purely azimuthal wave is governed exactly by the Burgers equations (as demonstrated for the special case that $\gamma=3$), and thus
our construction uses precise information on the stable self-similar solution $\bar W$ of the Burgers equation.  This allows us to provide detailed information about the blowup: by using the ODEs solved by $\tau(t)$ and $\xi(t)$, it is possible to compute the exact blowup time and location for our solutions to the 2D Euler equations.  Moreover, we have shown that the blowup profiles  have cusp singularities with H\"{o}lder $C^ {\sfrac{1}{3}} $ regularity.

We have shown in Remark \ref{rem:Euler_continue} that  in the case that $\gamma=3$, the first singularity can be continued as a discontinuous propagating shock wave for all 
time.\footnote{Note that even the purely azimuthal shock solution has vorticity, and this is extremely important for the shock continuation problem as
 initially irrotational flows can generate vorticity after the shock \cite{Ch2007}.}
In fact, we believe that the solutions we have constructed have  this type of continuation property for general $\gamma>1$.

\begin{conjecture} Given that the asymptotically self-similar shock solutions constructed in Theorem \ref{thm:general} form a $C^ {\sfrac{1}{3}} $ cusp at the initial blowup
time $t=T_*$, these solutions can be continued for short time as propagating piecewise smooth discontinuous  (possibly non-unique) shock profiles   which solve  the Euler equations on either side
of the time-dependent {\it curve of discontinuity}, and the evolution of this shock (or discontinuity) is governed by the Rankine-Hugoniot conditions.
\end{conjecture}

The solution we have constructed consists of a sound wave  which steepens  and shocks  in the azimuthal direction as well as the azimuthal velocity which also steepens and shocks
in the azimuthal direction.  The radial component of velocity can steepen in the azimuthal direction but does not shock.

\begin{conjecture}  Suppose that $(\rho,u_r,u_\theta)$ denotes the solution to the Euler equations given in Theorem \ref{thm:general}.   Then at the first blowup time $t=T_*$, the variable
$\p_\theta u_r$  is Lipschitz and no better.  In turn, let $\Omega(t)$ denote the material curve defined in \eqref{Omega}.  Then $\p\Omega(T_*)$ forms
a corner singularity.
\end{conjecture}

\appendix

\section{Toolshed}\label{sec:toolshed}

\begin{lemma}
\label{lem:transport}
Assume that the function $f = f(x,s)$ obeys the forced and damped transport equation
\[
\partial_s f + {\mathcal D} f + {\mathcal U} \partial_x f = {\mathcal F}
\]
for $s \in [s_0,\infty)$ and $x\in \RR$. Assume that ${\mathcal U}$, ${\mathcal D}$ and ${\mathcal F}$ are smooth, that 
\[
\inf_{(x,s) \in \RR\times[s_0,\infty)} {\mathcal D}(x,s) \geq \lambda_D
\]
for some $\lambda_D  \in \RR$, and that 
\[
\norm{{\mathcal F}(\cdot,s)}_{L^\infty(\RR)} \leq  {\mathcal F_0} e^{- s \lambda_F }
\]
for all $s\geq s_0$, for some $ {\mathcal F_0}\in [0,\infty)$ and $\lambda_F \in \RR$. 
For $\lambda_F < \lambda_D$ the function $f$ obeys the estimate
\begin{align}
\norm{f(\cdot,s)}_{L^\infty} \leq \norm{f(\cdot,s_0)}_{L^\infty} e^{-\lambda_D (s-s_0)} + \frac{{\mathcal F}_0}{\lambda_D - \lambda_F} e^{-s \lambda_F} \, .
\label{eq:lem:transport:1}
\end{align}
for all $s\geq s_0$.
On the other hand, for $\lambda_F > \lambda_D$, we have
\begin{align}
\norm{f(\cdot,s)}_{L^\infty} \leq \norm{f(\cdot,s_0)}_{L^\infty} e^{-\lambda_D (s-s_0)} + \frac{{\mathcal F}_0 e^{- s_0 \lambda_F}}{\lambda_F - \lambda_D} e^{-\lambda_D (s-s_0)} \, .
\label{eq:lem:transport:2}
\end{align}
for all $s\geq s_0$.
\end{lemma}
\begin{proof}[Proof of Lemma~\ref{lem:transport}]
Let $\p_s \psi = \mathcal{U} \circ \psi$ for $s> s_0$ and $\psi(x,s_0)=x$.   Then $ \tfrac{d}{ds}\left( e^{\int_{s_0}^ s ( \mathcal{D} \circ \psi )ds'} (f \circ \psi)  \right)  
= e^{\int_{s_0}^ s ( \mathcal{D} \circ \psi )ds'} (\mathcal{F} \circ \psi )$, from which it follows by integration that
$$
f(x,s) = f(x,s_0) e^{-\int_{s_0}^ s ( \mathcal{D} \circ \psi )dr}    +   \int_{s_0}^s  e^{-\int_{s'}^{s}( \mathcal{D} \circ \psi )ds''} (\mathcal{F} \circ \psi ) ds'  \,.
$$
From this identity, the inequalities \eqref{eq:lem:transport:1} and  \eqref{eq:lem:transport:2} immediately follow.
\end{proof}

The following lemma is a version of the maximum principle which is tailored to the needs of this paper.
\begin{lemma}
\label{lem:max:princ}
Assume that the function $f$ obeys the damped and non-locally forced transport equation
\begin{align}
\partial_s f(x,s) + {\mathcal D}(x,s) f(x,s) + {\mathcal U}(x,s) \partial_x f (x,s) = {\mathcal F}(x,s) + \int_{\RR} f(x',s) {\mathcal K}(x,x',s) d x'
\label{eq:nonlocal:transport}
\end{align}
for $s \in [s_0,\infty)$ and $x\in \RR$. Assume that the drift ${\mathcal D}$, the transport velocity ${\mathcal U}$, the forcing ${\mathcal F}$ and the kernel ${\mathcal K}$ are smooth functions, and assume we are given that the solution $f$ decays at spatial infinity: $\lim_{|x|\to \infty} \abs{f(x,s)} = 0$. Let $\Omega\subset \RR$ be a compact set, and assume that on its complement the damping obeys
\begin{align}
\inf_{(x,s) \in \Omega^c \times[s_0,\infty)} {\mathcal D}(x,s) \geq \lambda_D > 0 
\label{eq:max:princ:damping}
\end{align}
and that the forcing is bounded as
\begin{align}
\norm{{\mathcal F}(\cdot,s)}_{L^\infty(\Omega^c)} \leq  {\mathcal F_0} < \infty
\label{eq:max:princ:force}
\end{align}
for all $s\geq s_0$. For the kernel $\mathcal K$ we assume the estimate  
\begin{align}
\int_{\RR} \abs{\mathcal K(x,x',s)} dx' \leq \tfrac 34 {\mathcal D}(x,s)\qquad \mbox{for all} \qquad (x,s) \in \Omega^c \times [s_0,\infty)
\label{eq:max:princ:kernel}
\,.
\end{align}
Then, if for some $m>0$ we have 
\begin{align}
\norm{f(\cdot,s_0)}_{L^\infty( \mathbb{R} )} \leq \tfrac{1}{2} m \, , \qquad \mbox{and} \qquad \norm{f(\cdot,s)}_{L^\infty(\Omega)}\leq \tfrac{1}{2}m \, ,
\label{eq:max:princ:ass}
\end{align}
and the the forcing-to-damping relation
\begin{align}
m \lambda_D \geq 4 {\mathcal F_0} 
\label{eq:max:princ:cond}
\end{align}
holds, then the solution $f$ obeys 
\begin{align}
\norm{f(\cdot,s)}_{L^\infty(\RR)}\leq \tfrac{3 }{4}m
\label{eq:max:princ}
\end{align}
for all $s\geq s_0$. 
\end{lemma}
\begin{proof}[Proof of Lemma~\ref{lem:max:princ}]
 Assume for the sake of contradiction that \eqref{eq:max:princ} fails. Then, by the smoothness of solutions to \eqref{eq:nonlocal:transport} and the assumption that the solution $f$ vanishes as $\abs{x} \to \infty$, there exists a {\em first time} $s_*$ and a  {\em location} $x_*$ such that $\abs{f(s_*,x_*)} = \sfrac{3m}{4}$.  In view of \eqref{eq:max:princ:ass} we must have $x_* \in \Omega^c$. We may first assume  that $f$ attains its {\em global maximum} at this point, i.e. that $f(s_*,x_*) = \sfrac{3m}{4}$. By the minimality of $s_*$, we must have 
$(\partial_s f)(x_*,s_*) \geq 0$.
We will prove that the opposite inequality holds, thereby contradicting the existence of the breakthrough point $(x_*,s_*)$. For this purpose, evaluate the forced and damped transport equation at $(x_*,s_*)$, and note that because $f$   attains its global maximum  at this point, we have $\p_x f (x_*,s_*) = 0$. Additionally, from the assumption on the kernel, we have
\begin{align*}
\abs{\int_{\RR} f(x',s_*) {\mathcal K}(x_*,x',s_*) d x'}  \leq \tfrac 34 \norm{f(\cdot,s_*)}_{L^\infty(\RR)} {\mathcal D}(x_*,s_*) = \tfrac{3}{4} f (x_*,s_*) {\mathcal D}(x_*,s_*)
\end{align*}
and therefore, using \eqref{eq:max:princ:cond} we obtain
\[
(\partial_s f)(x_*,s_*) \leq \abs{{\mathcal F}(x_*,s_*)} - \tfrac 14 {\mathcal D}(x_*,s_*) f(x_*,s_*) \leq {\mathcal F_0} -  \tfrac{3}{16}m \lambda_D \leq - \tfrac{ {\mathcal F_0}}{4}  < 0
\]
which yields the desired contradiction. 

If on the other hand $f$ attains its {\em global minimum at this point}, i.e. $f(s_*,x_*) = - \sfrac{3m}{4}$, then by the minimality of $s_*$, we must have 
$(\partial_s f)(x_*,s_*) \leq 0$. 
We prove that the opposite inequality holds, yielding the contradiction. For this purpose, evaluate the forced and damped transport equation at $(x_*,s_*)$, and note that because $f$   attains its global minimum  at this point, we have $\p_x f (x_*,s_*) = 0$. Also, we have 
\begin{align*}
\abs{\int_{\RR} f(x',s_*) {\mathcal K}(x_*,x',s_*) d x'}  \leq \tfrac 34 \norm{f(\cdot,s_*)}_{L^\infty(\RR)} {\mathcal D}(x_*,s_*) = - \tfrac{3}{4} f (x_*,s_*) {\mathcal D}(x_*,s_*)
\end{align*}
so that
\[
(\partial_s f)(x_*,s_*) \geq {\mathcal F}(x_*,s_*) - \tfrac 14 {\mathcal D}(x_*,s_*) f(x_*,s_*) \geq - {\mathcal F_0} +  \tfrac{3}{16} m\lambda_D \geq \tfrac{ {1}}{4} \mathcal F_0 > 0 \, .
\]
Therefore, the breakthrough point $(x_*,s_*)$ does not exist, concluding the proof of \eqref{eq:max:princ}.
\end{proof}

\subsection*{Acknowledgments} 
T.B. was supported by the NSF grant DMS-1600868.  S.S. was supported by the Department of Energy Advanced Simulation and Computing (ASC) Program. V.V. was supported by the NSF grant DMS-1652134. The authors are grateful to Tarek Elgindi for numerous conversations on the topic of self-similar blowup in fluids and for stimulating discussions in the early phase of this project.  We are also  grateful to Jared Speck who explained to us the results of the paper \cite{LuSp2018}, 
as well as  some of the history of the subject.

 

\end{document}